\documentclass[reqno]{amsart}

\usepackage[utf8]{inputenc}
\usepackage[T1]{fontenc}
\usepackage[british]{babel,isodate}
\cleanlookdateon
\usepackage{lmodern}
\usepackage{enumerate}
\usepackage{tikz}
\usepackage{tikz-cd}
\usetikzlibrary{arrows}
\usepackage{cancel}
\usetikzlibrary{babel}
\usepackage{lscape}
\usepackage{url}

\usepackage{hyperref}
\hypersetup{linktocpage,colorlinks=true,citecolor=purple,linkcolor=blue,urlcolor=blue,filecolor=black}

\usepackage{float}
\usepackage{graphicx, import}
\usepackage{mathtools}
\usepackage{dsfont}
\usepackage{pinlabel}
\usepackage{yhmath}
\usepackage{enumitem}

\usepackage{xargs}                      
\usepackage{xcolor} 
\definecolor{OliveGreen}{rgb}{0,0.6,0}
 
\usepackage[colorinlistoftodos, prependcaption,textsize=tiny]{todonotes}
\newcommandx{\unsure}[2][1=]{\todo[linecolor=red,backgroundcolor=red!25,bordercolor=red,#1]{#2}}
\newcommandx{\change}[2][1=]{\todo[linecolor=blue,backgroundcolor=blue!25,bordercolor=blue,#1]{#2}}
\newcommandx{\info}[2][1=]{\todo[linecolor=OliveGreen,backgroundcolor=OliveGreen!25,bordercolor=OliveGreen,#1]{#2}}

\newcommand{\centre}[1]{\begin{array}{c} #1 \end{array}}

\usepackage{stackengine}

\usepackage{amsmath,amsthm,amssymb, amsfonts, amscd}
\usepackage{rotating,subcaption}
\usepackage{url}
\usepackage{graphicx,bm}
\usepackage{mathtools}
\usepackage[mathscr]{euscript}
\allowdisplaybreaks
\usepackage{lipsum}
\usepackage{bm}
\usepackage[capitalize]{cleveref}

\usepackage{xpatch}
\makeatletter
\AtBeginDocument{\xpatchcmd{\@thm}{\thm@headpunct{.}}{\thm@headpunct{}}{}{}}
\makeatother

\newtheorem{theorem}{Theorem}[section]
\newtheorem{proposition}[theorem]{Proposition}
\newtheorem{lemma}[theorem]{Lemma}
\newtheorem{corollary}[theorem]{Corollary}
\newtheorem{conjecture}[theorem]{Conjecture}
\newtheorem{problem}[theorem]{Problem}

\theoremstyle{definition}

\newtheorem{example}[theorem]{Example}

\newtheorem{construction}[theorem]{Construction}

\theoremstyle{remark}

\newtheorem{remark}[theorem]{Remark}

\numberwithin{equation}{section}

\newcommand{\hooklongrightarrow}{\lhook\joinrel\longrightarrow}
\newcommand{\longtwoheadrightarrow}{\relbar\joinrel\twoheadrightarrow}
\newcommand{\N}{\mathbb{N}}
\newcommand{\Z}{\mathbb{Z}}

\newcommand{\R}{\mathbb{R}}

\newcommand{\C}{\mathcal{C}}

\newcommand{\T}{\mathcal{T}}
\newcommand{\Tup}{\mathcal{T}^{\mathrm{up}}}
\newcommand{\TXC}{\mathcal{T}^{\mathrm{XC}}}
\newcommand{\id}{\mathrm{Id}}

\newcommand{\eend}[2]{\mathrm{End}_{#1}(#2)}
\let\hom\relax
\newcommand{\hom}[3]{\mathrm{Hom}_{#1}(#2,#3)}
\renewcommand{\to}{\longrightarrow}

\renewcommand{\SS}{\mathfrak{S}}

\DeclareMathOperator{\sign}{sign}
\newcommand{\XC}{\mathsf{XC}}
\newcommand{\E}{\mathcal{E}}

\DeclareRobustCommand\longtwoheadrightarrow
     {\relbar\joinrel\twoheadrightarrow}

\newcommand{\toiso}{\overset{\cong}{\to}}

\usepackage{accents}
\newcommand{\uhat}{\underaccent{\check}}

\makeatletter
\newcommand{\cupr@tip}{\text{\raisebox{-0.1ex}{$\m@th\hat{}$}}}
\newcommand{\cupr}{\mathbin{\cup\cupr@}}

\newcommand{\cupr@}{%
  \mathchoice
  {\mkern-1.35mu\cupr@tip}
  {\mkern-1.35mu\cupr@tip}
  {\mkern-1.55mu\cupr@tip}
  {\mkern-1.875mu\cupr@tip}
}
\makeatother

\makeatletter
\newcommand{\capr@tip}{\text{\raisebox{0.47ex}{$\m@th\uhat{}$}}}
\newcommand{\capr}{\mathbin{\capr@\cap}}

\newcommand{\capr@}{%
  \mathchoice
  {\mkern11.6mu\capr@tip\mkern-11.6mu}
  {\mkern11.4mu\capr@tip\mkern-11.4mu}
  {\mkern11.1mu\capr@tip\mkern-11.1mu}
  {\mkern10.2mu\capr@tip\mkern-10.2mu}
}
\makeatother

\makeatletter
\newcommand{\capl@tip}{\text{\raisebox{0.47ex}{$\m@th\uhat{}$}}}
\newcommand{\capl}{\mathbin{\capl@\cap}}

\newcommand{\capl@}{%
  \mathchoice
  {\mkern2.1mu\capl@tip\mkern-2.1mu}
  {\mkern2.1mu\capl@tip\mkern-2.1mu}
  {\mkern2.3mu\capl@tip\mkern-2.3mu}
  {\mkern2.1mu\capl@tip\mkern-2.1mu}
}
\makeatother

\makeatletter
\newcommand{\cupl@tip}{\text{\raisebox{-0.1ex}{$\m@th\hat{}$}}}
\newcommand{\cupl}{\mathbin{\cupl@\cup}}

\newcommand{\cupl@}{%
  \mathchoice
  {\mkern1.35mu\cupl@tip\mkern-1.35mu}
  {\mkern1.35mu\cupl@tip\mkern-1.35mu}
  {\mkern1.55mu\cupl@tip\mkern-1.55mu}
  {\mkern1.875mu\cupl@tip\mkern-1.875mu}
}
\makeatother

\DeclareFontFamily{U}{mathx}{}
\DeclareFontShape{U}{mathx}{m}{n}{ <-> mathx10 }{}
\DeclareSymbolFont{mathx}{U}{mathx}{m}{n}
\DeclareFontSubstitution{U}{mathx}{m}{n}
\DeclareMathAccent{\widecheck}{0}{mathx}{"71}

\usetikzlibrary{decorations.markings}
\tikzset{double line with arrow/.style args={#1,#2}{decorate,decoration={markings,%
mark=at position 0 with {\coordinate (ta-base-1) at (0,1pt);
\coordinate (ta-base-2) at (0,-1pt);},
mark=at position 1 with {\draw[#1] (ta-base-1) -- (0,1pt);
\draw[#2] (ta-base-2) -- (0,-1pt);
}}}}

\begin{document}


\title{XC-tangles and universal invariants}

\date{\today}

\author{Jorge Becerra}
\address{Université Bourgogne Europe, CNRS, IMB UMR 5584, F-21000 Dijon, France}
\email{\href{mailto:Jorge.Becerra-Garrido@ube.fr}{Jorge.Becerra-Garrido@ube.fr}}
\urladdr{ \href{https://sites.google.com/view/becerra/}{https://sites.google.com/view/becerra/}} 




\begin{abstract}
We introduce a class of decorated  abstract graphs, that we call XC-tangles,  that provides a very convenient framework to study quantum invariants of tangles and virtual tangles. These can be viewed as a far-reaching generalisation of rotational tangle diagrams for (virtual) upwards tangles, and constitute the topological analogue of XC-algebras, the minimum algebraic structure needed to construct an knot isotopy invariant following the construction of Lawrence and Lee. XC-tangles admit a very natural description in terms of the so-called XC-Gauss diagrams, and this equivalence lifts the well-known equivalence between virtual upwards tangles and upwards Gauss diagrams. For every XC-algebra $A$, there is a naturally defined strict monoidal full functor $Z_A: \TXC \rightarrow v\E(A)$ from the category of XC-tangles to the ``virtual category of elements of $A$''. When $A$ is the endomorphism algebra of a finite-dimensional representation of a ribbon Hopf algebra, this functor can be viewed as an extension of the corresponding Reshetikhin-Turaev functor. Lastly, we also initiate the study of a theory of finite type invariants for one-component XC-tangles that lifts that for virtual long knots.
%
%
%
\end{abstract}



\maketitle

\setcounter{tocdepth}{1}
\tableofcontents


\section{Introduction}

Quantum invariants of knots and links are the cornerstone in the connection between low-dimensional topology and the family of algebraic structures motivated by mathematical physics that we know today as quantum algebra. Concretely, many of the constructions of 3-manifold quantum invariants start by defining a link invariant that only then is promoted to a 3-manifold invariant after some normalisation \cite{RT3mnfd,hennings,virelizier,BBG,CC}.

The starting point in the construction of these quantum knots invariants is very often a ribbon Hopf algebra or a ribbon category. A ribbon structure on a Hopf $\Bbbk$-algebra $A$ is the choice of two invertible elements $R\in A \otimes_{\Bbbk} A$ and $v\in A$ obeying certain axioms that are the algebraic analogues of properties that the positive crossing and the full twist satisfy geometrically. The category $\mathsf{fMod}_A$ of finite-free left modules over a ribbon Hopf algebra $A$ is the prototypical example of a ribbon category; other sources (in fact much abundant) include modules over certain vertex operator algebras \cite{huang1,huang2,DW}.

If $V$ is a finite-free $A$-module, there is a unique strict monoidal functor $RT_V: \T \to \mathsf{fMod}_A$ from the category of oriented, framed tangles in the cube to $\mathsf{fMod}_A$ mapping the monoidal generator of $\T$ to $V$ and preserving the ribbon structure, this is the so-called \textit{Reshetikhin-Turaev functor} \cite{RT,turaev}; in particular it defines a functorial isotopy invariant of tangles. It is sometimes convenient to restrict this functor to the category $\Tup$ of \textit{upwards tangles} in the cube, that is, the monoidal subcategory of $\T$ on tangles without closed components whose strands are all oriented from bottom to top. The whole family $(RT_V)_{V \in \mathsf{fMod}_A}$ of Reshetikhin-Turaev functors associated to representations of $A$ can actually be recovered from a single invariant $\mathfrak{Z}_A$, sometimes called the \textit{universal invariant} associated to $A$ \cite{lawrence,reshetikhin,lee1,lee2,ohtsukibook,habiro}. 
In its simplest form, for a (long) knot $K$, we have  $\mathfrak{Z}_A(T) \in Z(A)$ and it satisfies that $$RT_V(T)= \rho (\mathfrak{Z}_A(T))$$ where $\rho: A \to \eend{\Bbbk}{V}$ is the structure map. This quantity is defined by placing copies of $R$ in the positive  crossings and copies of $\kappa :=uv^{-1}$ (here $u$ is the Drinfeld element) on the negative  full rotations,  we illustrate below the construction for the right-handed trefoil $T_{2,3}$, where we write $R= \sum_i \alpha_i \otimes \beta_i$:

\vspace{0.3cm}\noindent
 \begin{minipage}{.45\textwidth}
 \begin{equation*} 
\labellist \small  \hair 2pt
\pinlabel{$ \color{magenta} \bullet$} at 61 123
\pinlabel{$ \color{magenta} \alpha_i$} at -12 123
\pinlabel{$ \color{magenta} \bullet$} at 138 123
\pinlabel{$ \color{magenta} \beta_i$} at 206 123
\pinlabel{$ \color{blue} \bullet$} at 56 268
\pinlabel{$ \color{blue} \alpha_j$} at -12 268
\pinlabel{$ \color{blue} \bullet$} at 140 268
\pinlabel{$ \color{blue} \beta_j$} at 206 268
\pinlabel{$ \color{OliveGreen} \bullet$} at 56 419
\pinlabel{$ \color{OliveGreen} \alpha_\ell$} at -12 419
\pinlabel{$ \color{OliveGreen} \bullet$} at 135 419
\pinlabel{$ \color{OliveGreen} \beta_\ell$} at 206 419
\pinlabel{$ \color{orange} \bullet$} at 347 313
\pinlabel{$ \color{orange} \kappa$} at 423 325
\endlabellist
\includegraphics[width=0.3\textwidth]{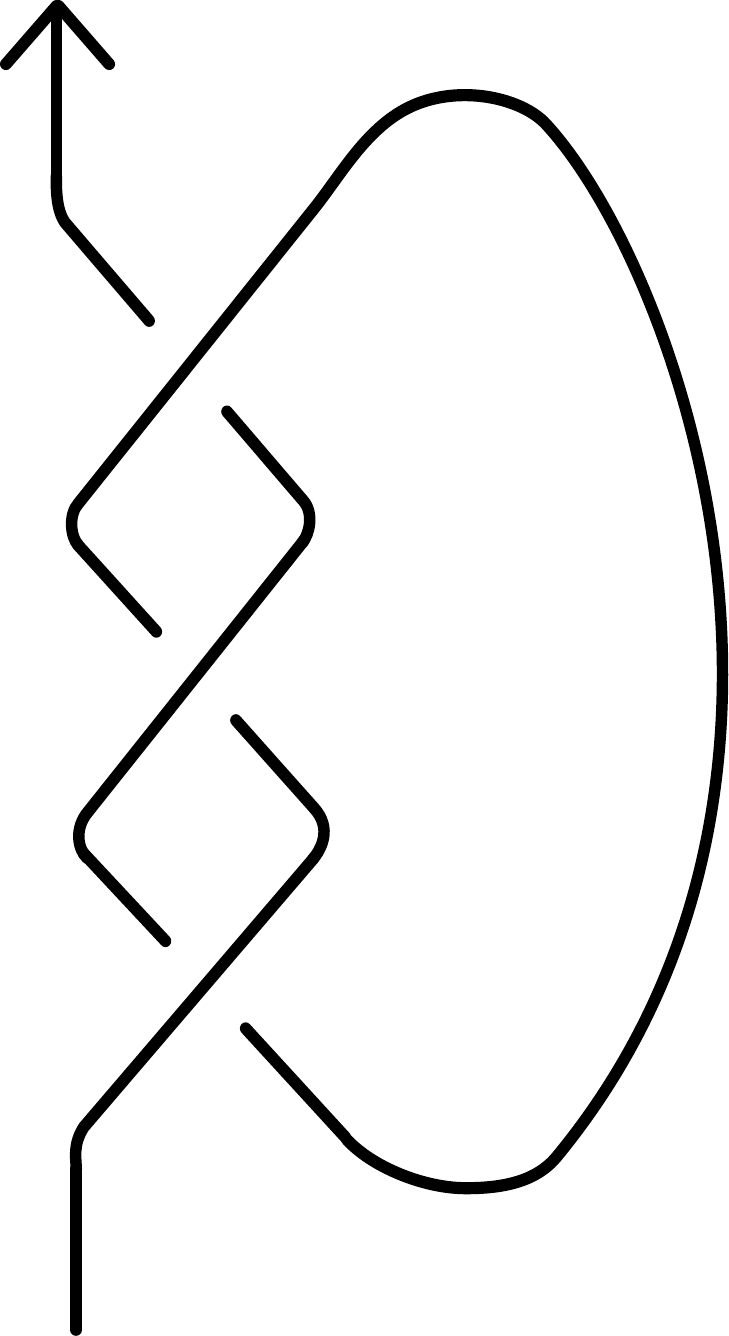} 
\end{equation*}
 \end{minipage}
  \begin{minipage}{.45\textwidth}
 $$\mathfrak{Z}_A(T_{2,3}) = \sum_{\color{magenta}{i} \color{black}{,} \color{blue}{j} \color{black}{,} \color{OliveGreen}{\ell}}\color{OliveGreen}{\beta_\ell}  \  \color{blue}{\alpha_j} \  \color{magenta}{\beta_i} \  \color{orange}{\kappa} \  \color{OliveGreen}{\alpha_\ell}\  \color{blue}{\beta_j}   \  \color{magenta}{\alpha_i} \color{black}{.}
$$
 \end{minipage}
 \vspace{0.3cm}

Observe that, in order to define the universal invariant $\mathfrak{Z}_A$, only the multiplicative structure of the underlying Hopf algebra has been used. This is a reflection of the fact that a much more weaker algebraic structure is needed to define an isotopy invariant of tangles following the construction of $\mathfrak{Z}_A$. \textit{XC-algebras}  are the minimal algebraic structure that produces an invariant in this way \cite{becerra_thesis,BH_reidemeister}. Such a structure even defines a functorial invariant and it is closely related to the Reshetikhin-Turaev functor \cite{becerra_refined}.

The aim of this paper is to introduce a  class of  decorated abstract graphs, that we call \textit{XC-tangles}, that are the geometrical counterpart of XC-algebras and therefore provides a very convenient framework to study quantum invariants. This class of tangles extends in fact the class of virtual, framed, oriented tangles and furthermore admits a theory of finite type invariants that enhances that for virtual tangles. The title of the paper is in fact a nod to Habiro's influential paper \textit{Bottom tangles and universal invariants} \cite{habiro}, where he similarly studied bottom tangles as a convenient class of tangles to study universal invariants.

\subsection{The category \texorpdfstring{$\TXC$}{T^XC} of XC-tangles}

During the last years rotational tangle diagrams have proven to be extremely convenient in the development of quantum knot invariants \cite{barnatanveenpolytime,barnatanveengaussians, barnatanveenAPAI, becerra_gaussians,becerra_thesis, becerra_refined,barnatanveenfast}. These are upwards tangle diagrams where all crossings point upwards and all maxima and minima appear in pairs of the following two forms, 
\begin{equation*}
\centre{
\centering
\includegraphics[width=0.27\textwidth]{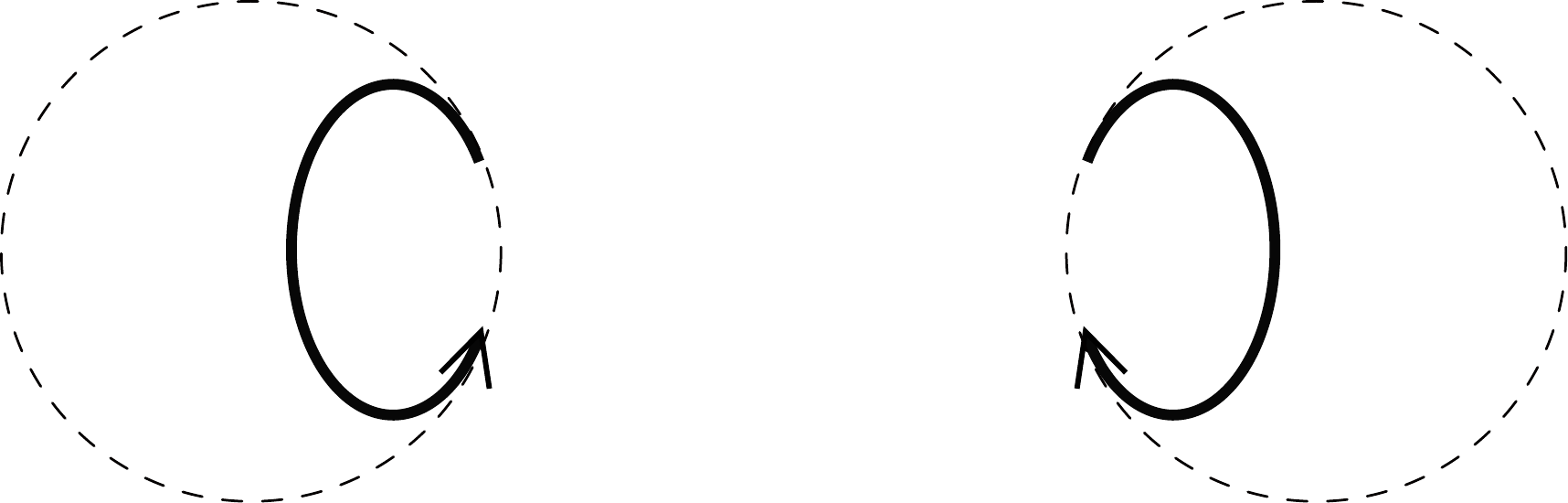}}
\end{equation*}
It immediately follows that any such a diagram can be decomposed as the merging in the plane of copies of the following elementary pieces,
\begin{equation}\label{eq:crossings_and_spinners_INTRO}
\centre{
\labellist \small \hair 2pt
\endlabellist
\centering
\includegraphics[width=0.6\textwidth]{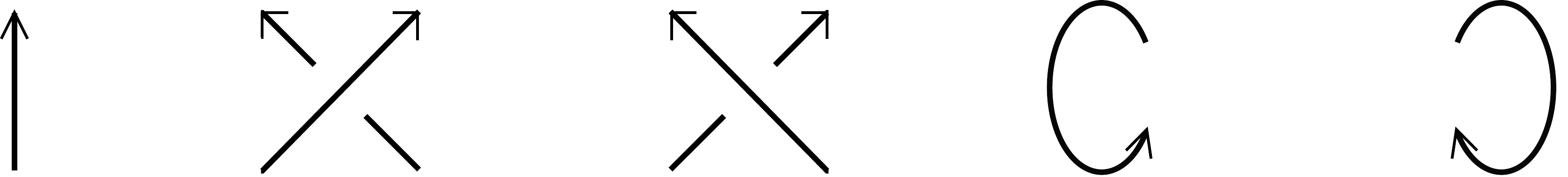}}.
\end{equation}
XC-tangles are abstract graphs that are locally made of the elementary pieces  \eqref{eq:crossings_and_spinners_INTRO}. We do not want to burden the reader with details so for now let us content ourselves with an incomplete definition (see \cref{subsec:XC-tangles} for the complete one). An \textit{XC-tangle diagram} is a directed graph with vertices of valence 1, 2 or 4, with tetravalent vertices  signed and vertex oriented, every edge carrying an integer $-1,0$ or $+1$ called its \textit{rotation number}, and \textit{strands} (maximal directed paths) having as endpoints univalent vertices. An \textit{XC-tangle} is an equivalence class of XC-tangle diagram modulo the so-called ``XC-Reidemeister moves'', namely the XC-analogues of the set of rotational Reidemeister moves for rotational diagrams; several generating sets of rotational Reidemeister moves have been described in \cite{BH_reidemeister}.

The data of an XC-tangle diagram can be equivalently encoded in an \textit{XC-Gauss diagram}. These are an enhanced version of the standard Gauss diagrams where the data of the rotation numbers is also added. More precisely, there is a monoidal isomorphism (\cref{thm:TXC=GDXC})
\begin{equation}\label{eq:TXC=GDXC_intro}
\TXC \overset{\cong}{\to}   \mathcal{GD}^{\mathrm{XC}}
\end{equation}
between the strict monoidal category $\TXC$ of XC-tangles and the category $\mathcal{GD}^{\mathrm{XC}}$ of XC-Gauss diagrams modulo the corresponding analogues of the XC-Reidemeister moves.

As we just mentioned, XC-Gauss diagrams are ordinary Gauss diagrams with extra decorations. This suggests a passage between XC-tangles and the geometrical counterpart of Gauss diagrams, namely virtual framed tangles (these are called rotational virtual in \cite{K99,kauff_rv}, here the name clash is quite unfortunate). In fact XC-tangles and virtual framed upwards tangles are related in two different ways.

\begin{theorem}[\cref{prop:forgetful_XC_virtual} and \cref{thm:vTup->TXC}]
There is a canonical monoidal full functor $$  U: \TXC \longtwoheadrightarrow v\Tup $$ as well as a monoidal embedding $$  I: v\Tup \hooklongrightarrow \TXC  $$ between the category $\TXC$ of XC-tangles and the category $v\Tup$ of virtual upwards tangles, such that $I$ is a section for $U$, that is $U \circ I= \id$. There are analogous functors between the categories of XC-Gauss diagrams and upwards Gauss diagrams, fitting in the following bigger diagram:
$$
\begin{tikzcd}
\TXC \arrow[d,two heads, bend right, "U"' pos=0.53] \rar{\cong} & \mathcal{GD}^{\mathrm{XC}}  \arrow[d,two heads, bend right, "U'"' pos=0.50] \\
v\Tup \dar[hook] \rar{\cong} \arrow[u,hook, bend right, "I"' pos=0.47] & \mathcal{GD}^{\mathrm{up}} \dar[hook] \arrow[u,hook, bend right, "I'"' pos=0.50] \\
v\T \dar[two heads] \rar{\cong} & \mathcal{GD} \dar[two heads] \\
v\T_u \rar{\cong} & \mathcal{GD}_u 
\end{tikzcd}
$$
\end{theorem}

Let us say some words about the functors appearing in the previous theorem. The functor $U'$ is the canonical functor that forgets the extra decorations keeping track of the rotation numbers of the edges. At the level of tangle diagrams, it translates into the  canonical functor  $U$. The existence of its section $I$ strongly relies on the existence of rotational diagrams for virtual upwards tangles, this runs parallel to the non-virtual case \cite{becerra_refined}. In particular, this means that the category $v\Tup$ can be seen as the category of rotational virtual tangle diagrams modulo the rotational Reidemeister moves from \cite{BH_reidemeister} together with its virtual analogues, these are made explicit in \cref{subsec:vTup}. Most importantly, the rotational Reidemeister moves are mapped under $I$ exactly to their XC-analogues, hence the well-definedness of $I$.

\subsection{A functorial invariant from XC-algebras}

XC-algebras are the minimal algebraic structure needed to obtain a tangle invariant  following the construction of the universal invariant $\mathfrak{Z}_A$ \cite{BH_reidemeister}. They can be defined in any symmetric monoidal category, but for simplicity here let us give the simplest description. Given a $\Bbbk$-algebra $A$, an \textit{XC-structure} on $A$ is the choice of two invertible elements $$R \in A\otimes_\Bbbk A \qquad , \qquad \kappa \in  A  $$ 
satisfying 
\begin{enumerate}[leftmargin=4\parindent, itemsep=2mm]
\item[(XC0)] $R^{\pm 1}=(\kappa \otimes \kappa) \cdot R^{\pm 1} \cdot (\kappa^{-1} \otimes \kappa^{-1})$,
\item[(XC1f)] $\sum_i \beta_i \kappa \alpha_i = \sum_i \alpha_i \kappa^{-1} \beta_i   $ ,
\item[(XC2c)] $  1\otimes \kappa^{-1} =  \sum_{i,j}  \alpha_i \bar{\alpha}_j \otimes \bar{\beta}_j \kappa^{-1} \beta_i    $,
\item[(XC2d)] $\kappa \otimes 1 =   \sum_{i,j}   \bar{\alpha}_i \kappa  \alpha_j \otimes \beta_j \bar{\beta}_i$,
\item[(XC3)] $R_{12}R_{13}R_{23}=R_{23}R_{13}R_{12}$,
\end{enumerate}
where we have put $R = \sum_i \alpha_i \otimes \beta_i$ and $R^{-1} = \sum_i \bar{\alpha}_i \otimes \bar{\beta}_i$. Ribbon Hopf algebras and endomorphisms algebras of representations of these constitute the main examples of XC-algebras.

In \cite{becerra_refined}, it was shown that the construction of the universal invariant using an XC-algebra $A$ gives rise to a full strict monoidal functor
\begin{equation}\label{eq:Z_A_intro}
Z_A: \Tup \to \E(A)  ,
\end{equation}
where $\E(A)$ is the so-called ``category of elements'' of the XC-algebra $A$. This functor in fact arises canonically as $\Tup$ is the ``free open-traced monoidal category generated by a single object'', and $Z_A$ the unique functor mapping the generator of $\Tup$  to the generator of $\E(A)$ and preserving the open-trace structure.  When $A=\mathrm{End}_\Bbbk(V)$ is the endomorphism algebra  of a finite-free module over a ribbon Hopf algebra, then $Z_{\mathrm{End}_\Bbbk(V)}$ essentially amounts to the Reshetikhin-Turaev functor $RT_V$. When the XC-algebra $A$ is in fact traced, then the functor $Z_A$ can be extended to $\T^+$, the free traced category (in the sense Joyal-Street-Verity \cite{joyal_street_traced}) generated by a single object.

In the framework of XC-tangles, there is a canonical analogue of the functor \eqref{eq:Z_A_intro}, which is in fact much more natural and easier to define as $\TXC$ is monoidally generated by a single object, the morphisms \eqref{eq:crossings_and_spinners_INTRO} and a so-called ``external algebra action''; this is the content of \cref{prop:generators_TXC}. This is a desirable feature that in some sense makes $\TXC$ behave as the free strict ribbon category on one object  $\T$. In fact, the following theorem should be understood as the XC-analogue of the Reshetikhin-Turaev theorem:

\begin{theorem}[\cref{thm:Z_A}]\label{thm:Z_A_intro}
Let $A$ be an XC-algebra. Then there exists a unique symmetric monoidal functor $$Z_A: \TXC \to v \E (A)  $$ with target the ``virtual category of elements of $A$'' which is the identity on objects, compatible with the external algebra action and
\begin{alignat*}{3}
Z_A(X) &= (R, (12)) \qquad  &, \qquad Z_A(X^-) &= (R_{21}^{-1}, (12))\\ Z_A(C) &= (\kappa^{-1}, \id) \qquad &, \qquad  Z_A(C^-)  &= (\kappa, \id).
\end{alignat*}
Furthermore, this functor is full.
\end{theorem}

In the previous statement, $X^\pm$ and $C^\pm$ refer to the generators of $\TXC$ from \eqref{eq:crossings_and_spinners_INTRO}.  Also, $\mathrm{End}_{v\E(A)} (n) \subset A^{\otimes n} \times \SS_n$, similarly to the category of elements $\E(A)$ from \cite{becerra_refined}. The precise definition makes use of the ``PROP $\mathsf{XC}$ for XC-algebras'', see \cref{subsec:vEA} for details. Composing with the aforementioned embedding $I$, we obtain an invariant of upwards virtual tangles 
\begin{equation}\label{eq:juju}
 v\Tup \to v\E(A)
\end{equation}
(in particular of virtual framed long knots).

The functor from \cref{thm:Z_A_intro} is, in fact, not merely an analogue of \eqref{eq:Z_A_intro}, but a genuine extension of it:

\begin{theorem}[\cref{thm:comparison_Z_As}]
Let $A$ be an XC-algebra.  There is a monoidal embedding $$ \E(A) \hooklongrightarrow v\E(A) $$
making the following diagram commute:
$$
\begin{tikzcd}
\T^{\mathrm{up}}\dar[hook] \rar[two heads]{Z_A} & \E(A) \dar[hook]\\
\T^{\mathrm{XC}}  \rar[two heads]{Z_A} & v\E(A)  
\end{tikzcd}
$$
\end{theorem}

Combining this with \cite[Theorem 5.4]{becerra_refined}, we obtain

\begin{corollary}[\cref{cor:extension_RT}]
If $V$ is a finite-free left module over a ribbon Hopf algebra, then the universal XC-tangle invariant $$ Z_{\mathrm{End}_\Bbbk(V)}: \TXC \to v\E(\mathrm{End}_\Bbbk(V))  $$ extends the Reshetikhin-Turaev invariant $RT_V: \Tup \to \mathsf{fMod}_A^{\mathrm{str}}$ to XC-tangles.
\end{corollary}

In particular, this implies that when $A=\mathrm{End}_\Bbbk(V)$ then the functor \eqref{eq:juju} extends the Reshetikhin-Turaev functor to upwards virtual tangles. This is closely related to the invariant obtained from the universal property of $v\T$ described in \cite{brochier}, we elaborate on this in \cref{subsec:brochier}. From that it follows that there is a commutative cube of strict monoidal categories and strict monoidal functors as displayed below (\cref{cor:big_diag}):
\begin{equation*}
\begin{tikzcd}
             & v\T \arrow[rr, "\widetilde{RT}_V "]  \arrow[from=dd, hook]      &                                                              & \mathsf{fMod}_\Bbbk^{\mathrm{str}}        \\
v\Tup \arrow[ru, hook] \arrow[rr,pos=0.8, "Z_{\mathrm{End}_\Bbbk(V)}",crossing over]                 &                                             & v\E (\mathrm{End}_\Bbbk(V)) \arrow[ru,  hook,"\iota_V"]         &                                           \\
 & \T \arrow[rr, pos=0.3, "RT_V"] &                                                              & \mathsf{fMod}_A^{\mathrm{str}} \arrow[uu] \\
\Tup \arrow[uu, hook] \arrow[rr, "Z_{\mathrm{End}_\Bbbk(V)}"] \arrow[ru, hook] &                                             & \E (\mathrm{End}_\Bbbk(V)) \arrow[uu, hook,crossing over] \arrow[ru, hook] &                                          
\end{tikzcd}
\end{equation*}

\subsection{Finite type invariants of XC-tangles}

We also lay the foundations of a theory of finite type invariants for one-component XC-tangles (that we call \textit{XC-knots}) that is closely related to that for virtual long knots.

Let us write $\mathcal{K}^{\mathrm{XC}} := \eend{\TXC}{1}$  for the set of XC-knots, and fix $\Bbbk$ a commutative ring with unit. There is a sequence of $\Bbbk$-linear isomorphisms
\begin{equation}\label{eq:sequence_intro}
\Bbbk \mathcal{K}^{\mathrm{XC}} \toiso \Bbbk \mathcal{GD}^{\mathrm{XC}}(\uparrow) \toiso \mathcal{P}^{\mathrm{XC}}
\end{equation}
that arises in the following way: the  isomorphism of \eqref{eq:TXC=GDXC_intro} restricts to a bijection $\mathcal{K}^{\mathrm{XC}} \toiso \mathcal{GD}^{\mathrm{XC}}(\uparrow), $ and this induces the first of the isomorphisms in \eqref{eq:sequence_intro}. Given an XC-Gauss diagram, taking the sum over all its subdiagrams induces the second isomorphism; we call $\mathcal{P}^{\mathrm{XC}}$ the \textit{XC-Polyak algebra} for the analogy with the classical case.

Let $n \geq 0$. A \textit{finite type invariant of degree $n$} for XC-tangles with values in a $\Bbbk$-module $M$ is a $\Bbbk$-module map $$ v:   \mathcal{P}^{\mathrm{XC}} \to M $$ such that $v$ vanishes in XC-Gauss diagrams with at least $n$ decorations.
Composing with  \eqref{eq:sequence_intro}, we get a ``genuine'' XC-knot invariant $v: \Bbbk \mathcal{K}^{\mathrm{XC}} \to M$.

There is a sequence
\begin{equation}\label{eq:maps_V_intro}
\mathcal{V}^{\mathrm{XC}}_n(M) \to v\mathcal{V}_n(M) \to \mathcal{V}_n(M)
\end{equation}
between the sets of finite type invariants of degree $n$ for XC-knots, virtual long knots and long knots. It is a classical result of Goussarov that for  $\Bbbk = \Z$ the second arrow is surjective  \cite[Theorem 3.A]{GPV}. It turns out that the same property applies to the first arrow.

\begin{proposition}[\cref{prop:surjection}]
For any commutative ring $\Bbbk$, we have that the map $\mathcal{V}^{\mathrm{XC}}_n(M) \to v\mathcal{V}_n(M)$ from \eqref{eq:maps_V_intro} is a surjection.
\end{proposition}

As an immediate consequence of this, we obtain the existence of Gauss diagram formulas for finite type invariants of XC-knots (\cref{cor:GD_formulas}).





\subsection*{Acknowledgments} The author is deeply grateful to Roland van der Veen, who introduced the idea of XC-tangles to him and their potential use in the study of universal invariants. The author would also like to thank  Federica Gavazzi, Kevin van Helden, Edwin Kitaeff, Gwénaël Massuyeau, Luis Paris and Lukas Woike  for helpful discussions related to this project.  The author was supported by the ARN project CPJ number ANR-22-CPJ1-0001-0  at the Institut de Mathématiques de Bourgogne (IMB). The IMB receives support from the EIPHI Graduate School (contract ANR-17-EURE-0002).


\section{Classical virtual links and tangles}\label{sec:1}

In this section, we will briefly review  the theory of classical virtual links and their categorification. Some of the constructions that will appear later on in the present article will be elaborations of classical ideas that we make explicit here. All links and tangles will be considered to be oriented.

\subsection{Virtual links: three takes}\label{subsec:virtual_links}

To start off, let us recall virtual links from three different perspectives. The first approach, the closest to usual knots diagrams on the plane, is that of virtual diagrams. For $n \geq 1$, an   \textit{$n$-component virtual link diagram} is an oriented immersion $\amalg_n  S^1 \looparrowright \R^2$ which has only finitely many transversal self-intersections. The set of double points is divided into two subsets of real and virtual crossings. Real crossings are enhanced  with a sign $+$ or $-$ (equivalently, the under/over-pass information). In pictures, virtual crossings are depicted with a small circle to distinguish them from the real crossings.

Similarly to the non-virtual case, we regard virtual link diagrams up to planar isotopy and the set of real, virtual and mixed Reidemeister moves depicted below:

\begin{equation*}
\centre{
\labellist \small \hair 2pt
\pinlabel{$\leftrightsquigarrow$}  at 434 190
\pinlabel{{\scriptsize $\Omega 1$}}  at 440 260
\pinlabel{$\leftrightsquigarrow$}  at 950 190
\pinlabel{{\scriptsize $\Omega 1$}}  at 950 260
\pinlabel{$\leftrightsquigarrow$}  at 2160 190
\pinlabel{{\scriptsize $\Omega 2$}}  at 2160 260
\endlabellist
\centering
\includegraphics[width=0.80\textwidth]{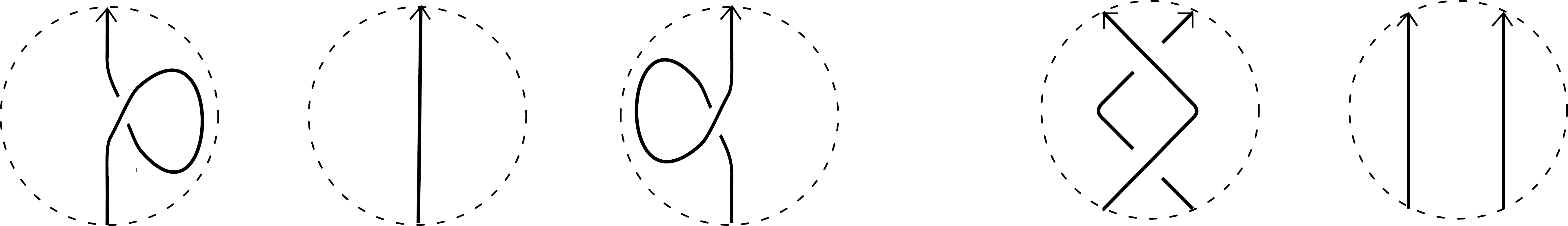}}
\end{equation*}
\begin{equation*}
\centre{
\labellist \small \hair 2pt
\pinlabel{$\leftrightsquigarrow$}  at 434 190
\pinlabel{{\scriptsize $\Omega 3$}}  at 440 260
\pinlabel{$\leftrightsquigarrow$}  at 1645 190
\pinlabel{{\scriptsize $v\Omega 1$}}  at 1645 260
\pinlabel{$\leftrightsquigarrow$}  at 2155 190
\pinlabel{{\scriptsize $v\Omega 1$}}  at 2155 260
\endlabellist
\centering
\includegraphics[width=0.80\textwidth]{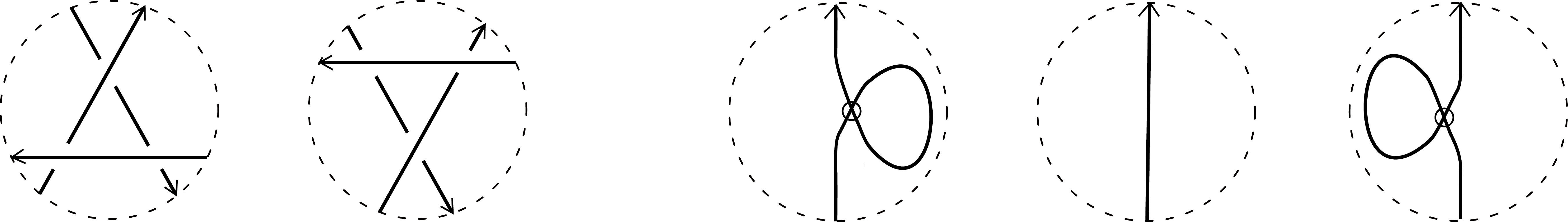}}
\end{equation*}
\begin{equation*}
\centre{
\labellist \small \hair 2pt
\pinlabel{$\leftrightsquigarrow$}  at 434 190
\pinlabel{{\scriptsize $v\Omega 2$}}  at 440 260
\pinlabel{$\leftrightsquigarrow$}  at 1590 190
\pinlabel{{\scriptsize $v\Omega 3$}}  at 1590 260
\pinlabel{$\leftrightsquigarrow$}  at 2720 190
\pinlabel{{\scriptsize $m\Omega 3$}}  at 2720 260
\endlabellist
\centering
\includegraphics[width=0.98\textwidth]{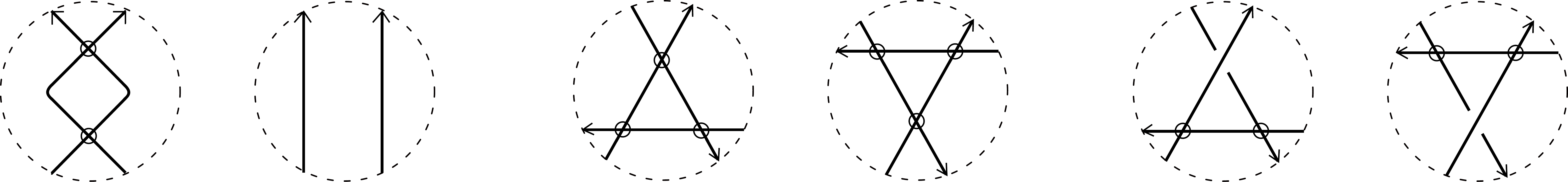}}
\end{equation*}

Each of the depicted equivalences should be understood as identifying two virtual knot diagrams that are identical except in some open neighbourhoods of the diagrams, represented by the dotted circle, where they look as shown. Mimicking the arguments of \cite{polyak10}, it is readily seen that the former is a minimal generating set of Reidemeister moves for virtual links, see also \cite{KMV}. The virtual moves $(v\Omega 1)$ -- $(m\Omega 3)$ can be in fact encoded in the so-called \textit{detour move}, that asserts that a given arc that contains only virtual crossings can be replaced by any other arc on the plane joining the endpoints of the original arc marking the newly produced crossings as virtual.  We write $\mathsf{vL}$ for the set of \textit{virtual links}, that is, the set of oriented virtual link diagrams modulo planar isotopy and the Reidemeister moves above. If $\mathsf{L}$ denotes the set of oriented links in the plane, that is, the set of oriented link diagrams modulo the real Reidemeister moves $(\Omega 1)$--$(\Omega 3)$, then the canonical map 
\begin{equation}\label{eq:L->vL}
\mathsf{L} \hooklongrightarrow \mathsf{vL}
\end{equation}
is an injection \cite{GPV}.

The second approach to virtual links, more topological, is to regard them as isotopy classes of links in thickened surfaces, up to thickened cylinder addition/removal. More precisely, we are going to consider pairs $(\Sigma, L)$ where $\Sigma$ is a closed, oriented  surface, and $L$ is an isotopy class of an  embedding $L: \amalg_n S^1 \hooklongrightarrow \Sigma \times D^1$, where $D^1$ is the one-dimensional disc. We require that the map $\pi_0 (\amalg_n S^1) \to \pi_0(\Sigma \times D^1)$ induced by $L$ is surjective, that is, every path-component of $\Sigma \times D^1$ contains at least one component of $L$. It is clear that we can view $L$ as a link diagram on $\Sigma$, modulo the (real) Reidemeister moves $(\Omega 1)$--$(\Omega 3)$. Given a pair  $(\Sigma, L)$, we can perform a 0-surgery on $\Sigma$ with the embedding $S^0 \times D^2 \hooklongrightarrow \Sigma$ far from the link diagram, in order to obtain a surface $\Sigma '$ whose genus is increased by one. We then obtain a new pair $(\Sigma', L)$ by viewing $L$ inside the new surface, that we call the \textit{stabilisation} of  $(\Sigma, L)$. We write $\mathsf{LTS}$ for the set of pairs $(\Sigma, L)$ modulo the relation that identifies two pairs if they can be connected by a zig-zag of stabilisations.

The third approach to virtual knots is the most combinatorial. A \textit{Gauss diagram} on $\amalg_n  S^1$ is a oriented, finite, univalent graph $D$ such that the set of vertices is embedded in $\amalg_n  S^1$ and each connected component of $D$, called \textit{chord}, carries a sign $\pm$. We identify two Gauss diagrams $D,D'$ in $\amalg_n S^1$ if there is a orientation-preserving homeomorphism of pairs $(\amalg_n  S^1 \cup D, \amalg_n  S^1) \toiso (\amalg_n  S^1 \cup D', \amalg_n  S^1)$ preserving the signs of the chords. The oriented $\amalg_n  S^1$ will be depicted with a thick line whereas the graph $D$ will be drawn with thin arrows. Here is an example of a Gauss diagram in $S^1 \amalg S^1$:
\begin{equation*}
\centre{
\labellist \small \hair 2pt
\pinlabel{\scriptsize  $+$}  at 50 170
\pinlabel{\scriptsize  $-$}  at 200 300
\pinlabel{\scriptsize  $-$} at 433 44
\pinlabel{\scriptsize  $+$}  at 243 116
\pinlabel{\scriptsize  $+$}  at 427 310
\pinlabel{\scriptsize  $-$}  at 665 300
\pinlabel{\scriptsize  $+$}  at 570 210
\endlabellist
\centering
\includegraphics[width=0.35\textwidth]{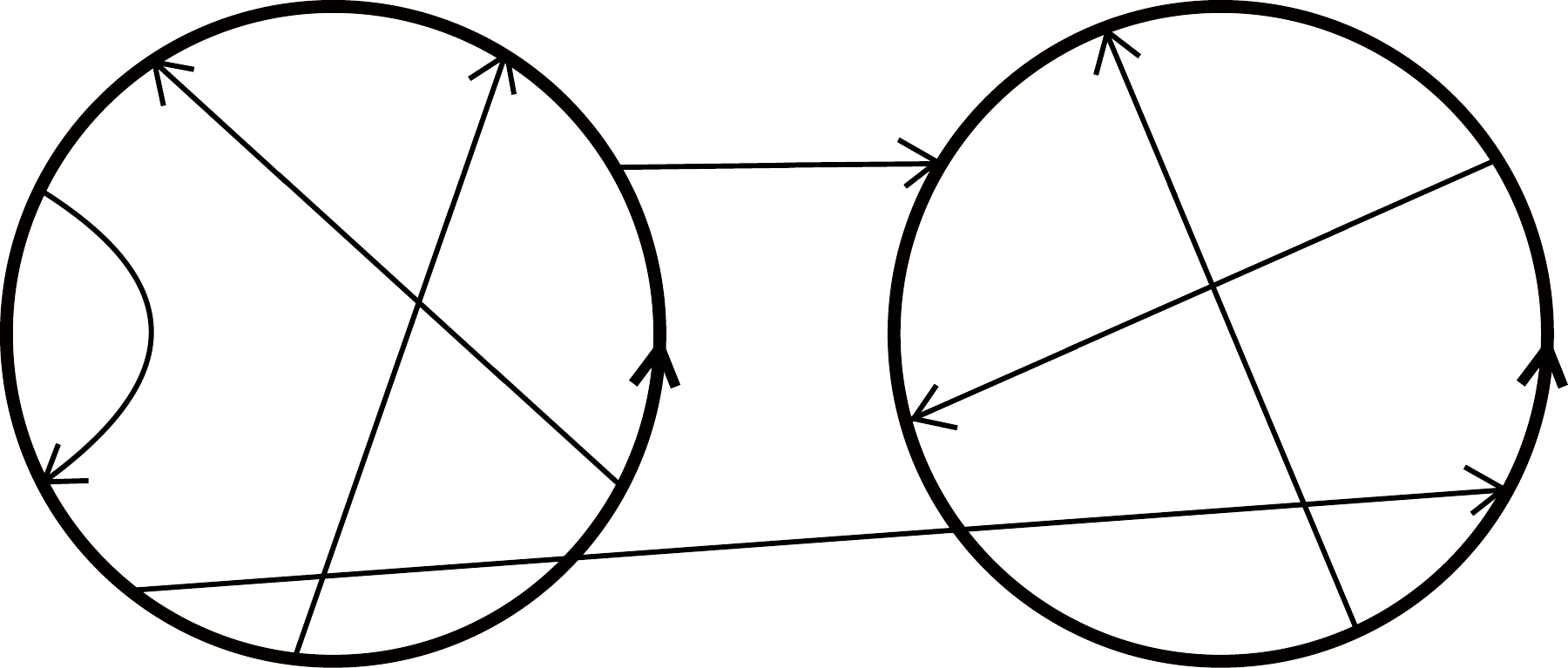}}
\end{equation*}

We write $\mathsf{GD}$ for the set of Gauss diagrams in disjoint unions of oriented circles modulo the equivalence relation generated by the following analogues of the (real) Reidemeister moves:

\begin{equation*}
\centre{
\labellist \small \hair 2pt
\pinlabel{\scriptsize  $+$}  at 392 93
\pinlabel{\scriptsize  $+$}  at 1800 93
\pinlabel{$\leftrightsquigarrow$}  at 600 63
\pinlabel{{\scriptsize $G 1$}}  at 600 120
\pinlabel{$\leftrightsquigarrow$}  at 1294 63
\pinlabel{{\scriptsize $G 1'$}}  at 1294 120
\endlabellist
\centering
\includegraphics[width=0.75\textwidth]{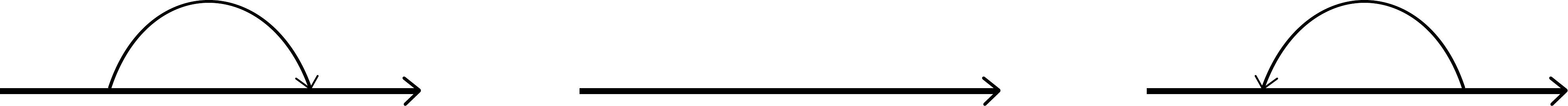}}
\end{equation*}
\begin{equation*}
\centre{
\labellist \small \hair 2pt
\pinlabel{$\leftrightsquigarrow$}  at 625 125
\pinlabel{{\scriptsize $G 2$}}  at 625 185
\pinlabel{\scriptsize  $+$}  at 145 145
\pinlabel{\scriptsize  $-$}  at 370 145
\endlabellist
\centering
\includegraphics[width=0.5\textwidth]{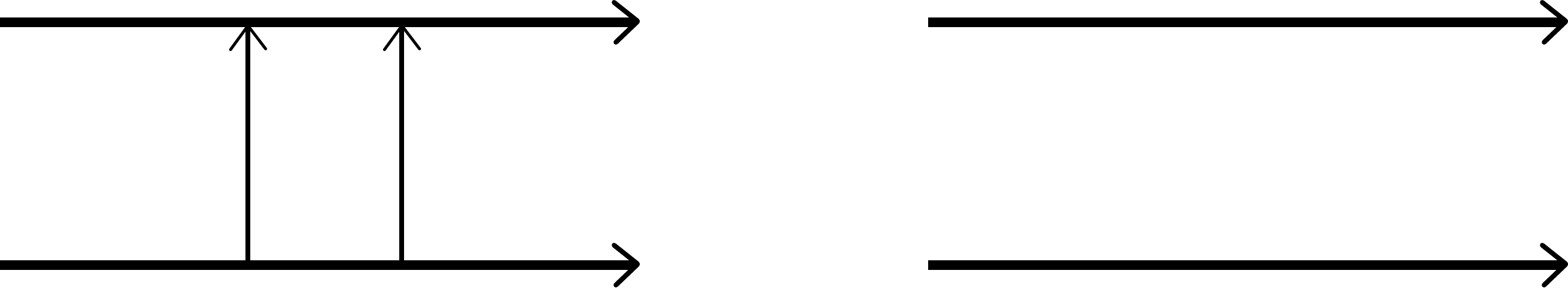}}
\end{equation*}
\begin{equation*}
\centre{
\labellist \small \hair 2pt
\pinlabel{\scriptsize  $-$}  at 120 350
\pinlabel{\scriptsize  $+$}  at 366 343
\pinlabel{\scriptsize  $-$} at 145 86
\pinlabel{$\leftrightsquigarrow$}  at 616 211
\pinlabel{{\scriptsize $G 3$}}  at 616 270
\pinlabel{\scriptsize  $+$}  at 860 348
\pinlabel{\scriptsize  $-$}  at 1125 343
\pinlabel{\scriptsize  $-$} at 1100 90
\endlabellist
\centering
\includegraphics[width=0.5\textwidth]{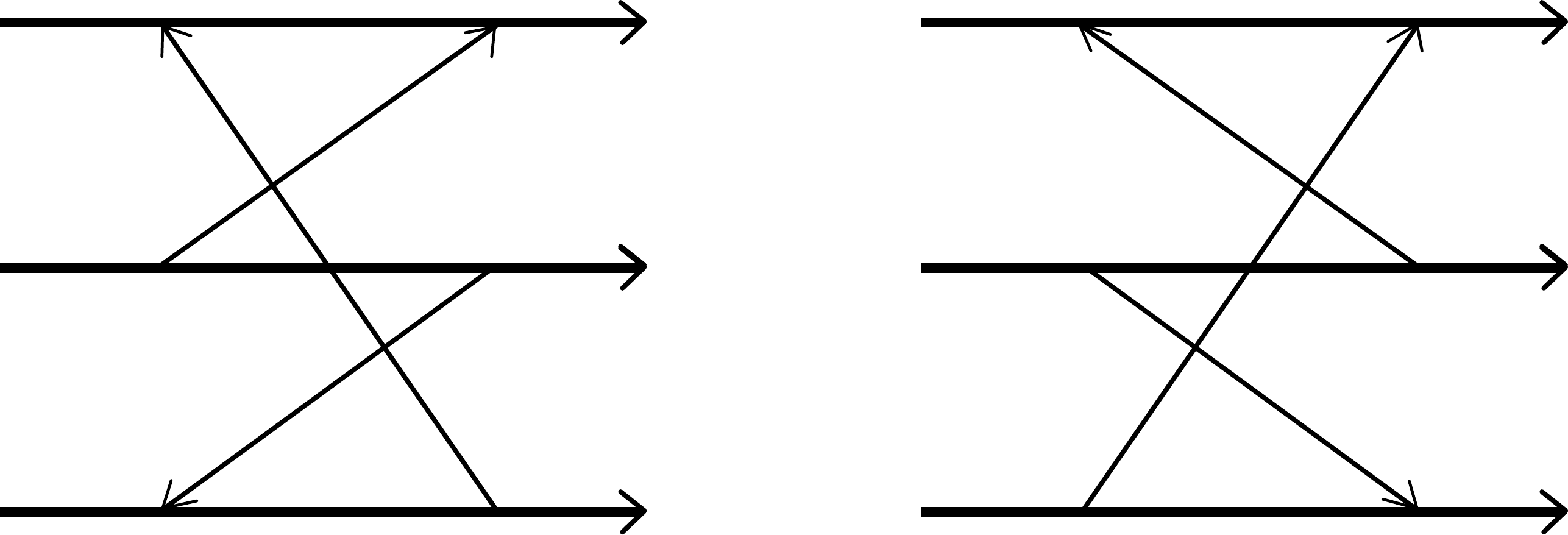}}
\end{equation*}

These moves should be understood as follows: each of the pictures represents a Gauss diagram on $\amalg_k D^1$ where $k=1,2,3$. Let us write $D_L$ and $D_R$ for the Gauss diagram on the left-hand side and right-hand side, respectively.  Then we identify two Gauss diagrams $D,D'$ in $\amalg_n  S^1$ whenever we have a diagram of  orientation- and chord-sign-preserving embeddings of pairs as follows:

\begin{equation}\label{eq:diagram_for_GD}
\begin{tikzcd}[row sep=0.2em]
& (\amalg_k  D^1 , D_L) \rar[hook] & (\amalg_n S^1, D) \\
(\amalg_k D^1, \emptyset) \arrow[hook]{ur}  \arrow[hook]{dr} & &\\
& (\amalg_k  D^1, D_R) \rar[hook] & (\amalg_n S^1, D')
\end{tikzcd}
\end{equation}

More informally, we identify Gauss diagrams that are identical except ``locally'', where they differ as shown.

It turns out that these three sets are different realisations of the same mathematical object:

\begin{theorem}[\cite{K99,CKS,MI}]\label{thm:3_descriptions}
There are bijections 
$$  \mathsf{LTS}  \overset{\cong}{\longleftarrow} \mathsf{vL} \toiso  \mathsf{GD}.$$
\end{theorem}

The bijection   $\mathsf{vL} \toiso \mathsf{LTS}$ assigns, for a virtual link with $g$ virtual crossings, a link in the closed, connected, oriented surface $\Sigma_g$  of genus $g$ obtained from $S^2 = \R^2 \cup \{ \infty \}$ by performing $g$ 0-surgeries around every virtual crossing, and the ``lifting'' the virtual crossings to the newly created handles. On the other hand, the map $\mathsf{vL} \toiso \mathsf{GD}$ associates, for every virtual  link with $n$ components, a Gauss diagram as follows: first, parametrise the skeleton $\amalg_n  S^1$ of the Gauss diagram according to the virtual knot. For every real crossing, consider a chord in $\amalg_n S^1$ whose endpoints are determined by the chosen parametrisation, oriented from the overpass to the underpass, and carrying the sign of the crossing. Virtual crossings are completely ignored.

\subsection{Virtual tangles}

Next we would like to explain how to categorify the bijection  $\mathsf{vL} \toiso  \mathsf{GD}$. The author believes that this is well-known to experts but it is rarely found in the literature. We will do this both in the unframed and framed case.

A  \textit{virtual tangle diagram} is an oriented immersion $$D: \left( \coprod_n D^1 \right) \amalg \left( \coprod_m  S^1 \right) \looparrowright D^1 \times D^1$$   which has only finitely many transversal self-intersections and such that the image of $\amalg_n \partial D^1$ lies uniformly distributed on $D^1 \times \partial D^1$. As before, the set of double points is divided into two subsets of real and virtual crossings, with real crossings carrying the under/over-pass information. If $m=0$, we say that the diagram is \textit{open}.

Here is an example of a virtual tangle:
\begin{equation*}
\centre{
\centering
\includegraphics[width=0.20\textwidth]{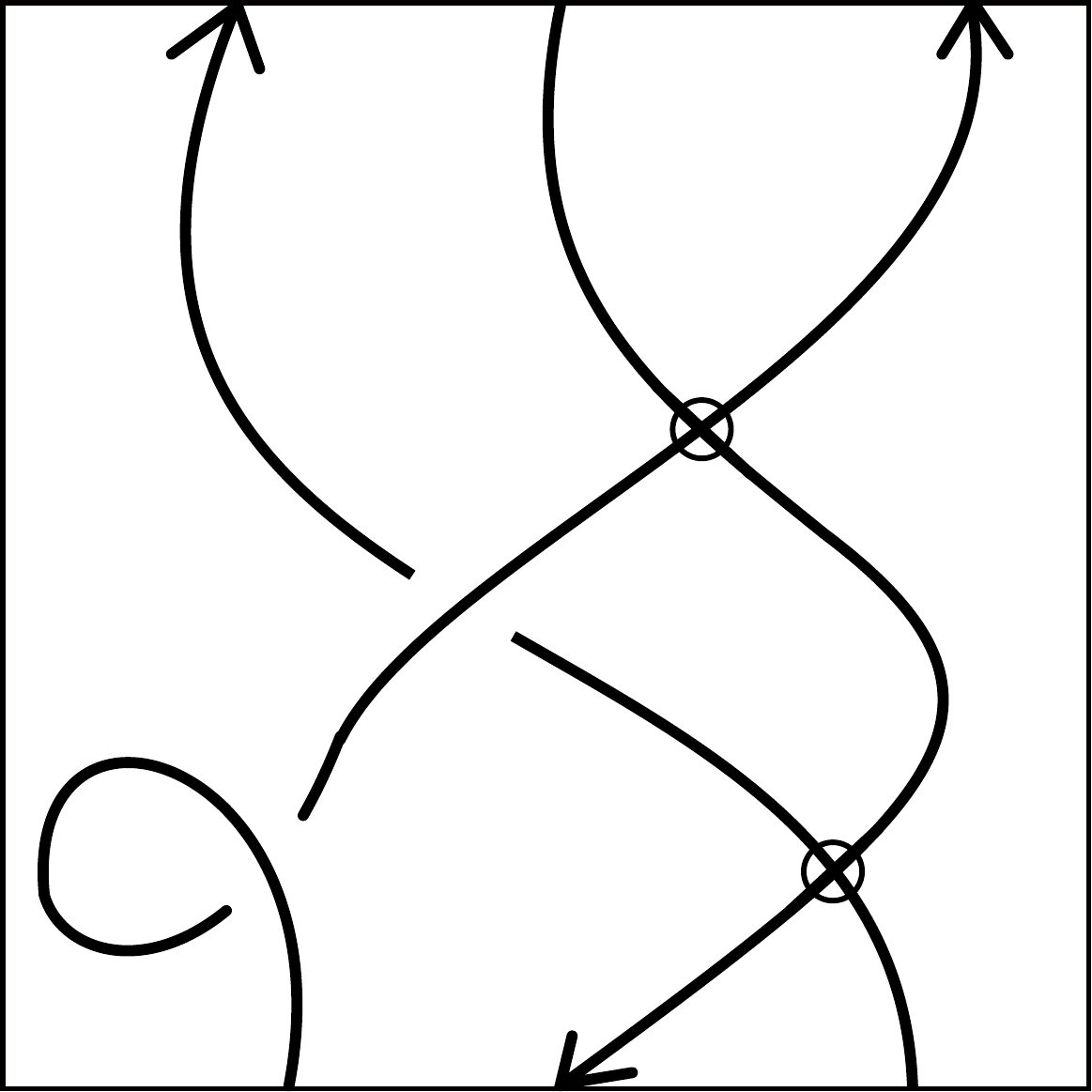}}
\end{equation*}

We assign the symbols $+$ or $-$ to the endpoints of $D$ depending on whether $D$ points upwards or downwards, respectively. Let $\mathrm{Mon}(+,-)$ be the free monoid   on the set $\{ +,- \}$. Given a tangle diagram $D$, let $s(D)$ (resp. $t(D)$) be the element of $\mathrm{Mon}(+,-)$ resulting from reading from left to right the symbols attached in $D^1 \times \{-1\}$ (resp. $D^1 \times \{-1\}$). We call $s(D)$ and $t(D)$ the \textit{source} and \textit{target} of $D$.

Let us write $v\T$ for the strict monoidal category with objects the free monoid $\mathrm{Mon}(+,-)$  on the set $\{ +,- \}$ and morphisms $\hom{v\T}{s}{t}$ the set of virtual tangle diagrams $D$ with $s(D)=s$ and $t(D)=t$ modulo planar isotopy, the Reidemeister moves $(\Omega 2)$, $(\Omega 3)$, $(v\Omega 2)$, $(v\Omega 3)$ and $(m\Omega 3)$ from  \cref{subsec:virtual_links} and additionally the following framed version of the $(\Omega 1)$ and the $(\Omega 2')$ and $(\Omega 3')$ moves below:
\begin{equation*}
\centre{
\labellist \small \hair 2pt
\pinlabel{$\leftrightsquigarrow$}  at 434 190
\pinlabel{{\scriptsize $\Omega 1f$}}  at 440 260
\pinlabel{$\leftrightsquigarrow$}  at 1590 190
\pinlabel{{\scriptsize $\Omega 2'$}}  at 1590 260
\pinlabel{$\leftrightsquigarrow$}  at 2720 190
\pinlabel{{\scriptsize $\Omega 3'$}}  at 2720 260
\endlabellist
\centering
\includegraphics[width=0.98\textwidth]{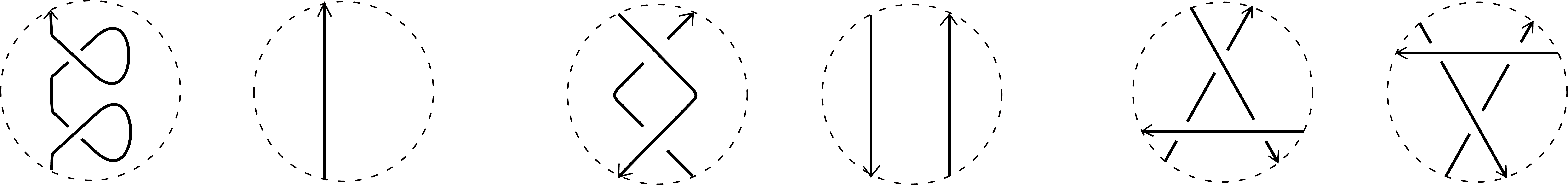}}
\end{equation*}
%
(the two latter moves do indeed need to be added according to the results of \cite[\S 4.2]{BH_reidemeister}).

The composition and monoidal product are given by vertical and horizontal stacking as in the usual tangle category, see e.g. \cite{RT,turaev,becerra_thesis,becerra_refined}. In particular, the unit of the monoidal structure is the empty word $\emptyset$. We call $v\T$ the \textit{framed virtual tangle category} and its arrows \textit{framed virtual tangles} (these are called ``rotational'' in \cite{K99,kauff_rv}). In  \cite{brochier} it is shown that $v\T$ is monoidally equivalent to the category of ``framed oriented tangles in marked surfaces''.

The \textit{unframed virtual tangle category} $v\T_u$ is the quotient category of $v\T$ modulo the unframed real and virtual Reidemeister moves $(\Omega 1)$ and $(v\Omega 1)$; this is equivalent to considering tangle diagrams modulo all the Reidemeister moves from \cref{subsec:virtual_links}. By construction, there is a strict monoidal full functor
\begin{equation}\label{eq:vT->vTu}
v \T \longtwoheadrightarrow v\T_u.
\end{equation}
In fact, this functor fits in a more general commutative diagram, as follows: let us write $\T$ (resp. $\T_u$) for the category of framed, oriented tangles (resp. unframed, oriented tangles) in the cube. Then we have the following commutative diagram of monoidal categories and strict monoidal functors:
\begin{equation} 
\begin{tikzcd}
\T \arrow[d,two heads]  \rar[hook] & v \T  \arrow[d,two heads] \\ \T_u  \rar[hook] & v\T_u
\end{tikzcd}
\end{equation}
All the functors here are the identity on objects. The right-hand side vertical functor is \eqref{eq:vT->vTu}, and the left-hand side one is defined similarly for ordinary tangles; these both are full. On the other hand, the horizontal functors are naturally defined by viewing a tangle diagram only with real crossings as a virtual tangle diagram. Mimicking \cite{GPV}, one can check that these are embeddings.

In fact, both categories $v\T$ and $v\T_u$ admit a symmetric braiding, where $P_{n,m}: n+m \to n+m$ is identical to the usual braiding in the category $\T$ of tangles except that all crossings are now considered virtual. In particular this makes \eqref{eq:vT->vTu} a symmetric strict monoidal functor between strict symmetric monoidal categories. 


Note that $\T_u$ (resp. $v\T_u$) categorifies the set of unframed links (resp. unframed virtual links) in the sense that $$\hom{\T_u}{\emptyset}{\emptyset} \cong \mathsf{L} \qquad , \qquad \hom{v\T_u}{\emptyset}{\emptyset} \cong \mathsf{vL}.  $$
In particular, the bottom horizontal embedding $\T_u \hooklongrightarrow v\T_u$  categorifies the injection \eqref{eq:L->vL}.

For later use, we will introduce another category: the category $v\Tup$ of \textit{virtual upwards tangles} is the monoidal subcategory of $v\T$ on the objects $\mathrm{Mon}(+) \subset \mathrm{Mon}(+,-)$ and arrows virtual tangles without closed components. This is the virtual version of the category $\Tup$ of upwards tangles exploited in \cite{becerra_refined}.

\subsection{Gauss diagrams in polarised 1-manifolds}

We will now  categorify the set of Gauss diagrams on disjoint union of circles. The first step will be to generalise the skeleton to arbitrary 1-manifolds. 

Let $X$ be a compact, oriented 1-manifold (that is, a disjoint union of oriented intervals and circles). A \textit{Gauss diagram} on $X$ is an oriented, finite, univalent graph $D$ such that the set of vertices is embedded in the interior of $X$ and each connected component of $D$, called \textit{chord}, carries a sign $\pm$. As before, we identify two Gauss diagrams $D,D'$ in $X$ if there is an orientation-preserving homeomorphism of pairs $(X\cup D, X) \toiso (X \cup D', X)$ which also preserves  the signs of the chords.

\begin{example}

The following is a Gauss diagram on the disjoint union of two oriented intervals:
\begin{equation*}
\centre{
\labellist \scriptsize \hair 2pt
\pinlabel{ $  +$}  at 115 243
\pinlabel{$  +$}  at 340 288
\pinlabel{$  -$}  at 330 95
\pinlabel{$  -$}  at 100 112
\endlabellist
\centering
\includegraphics[width=0.25\textwidth]{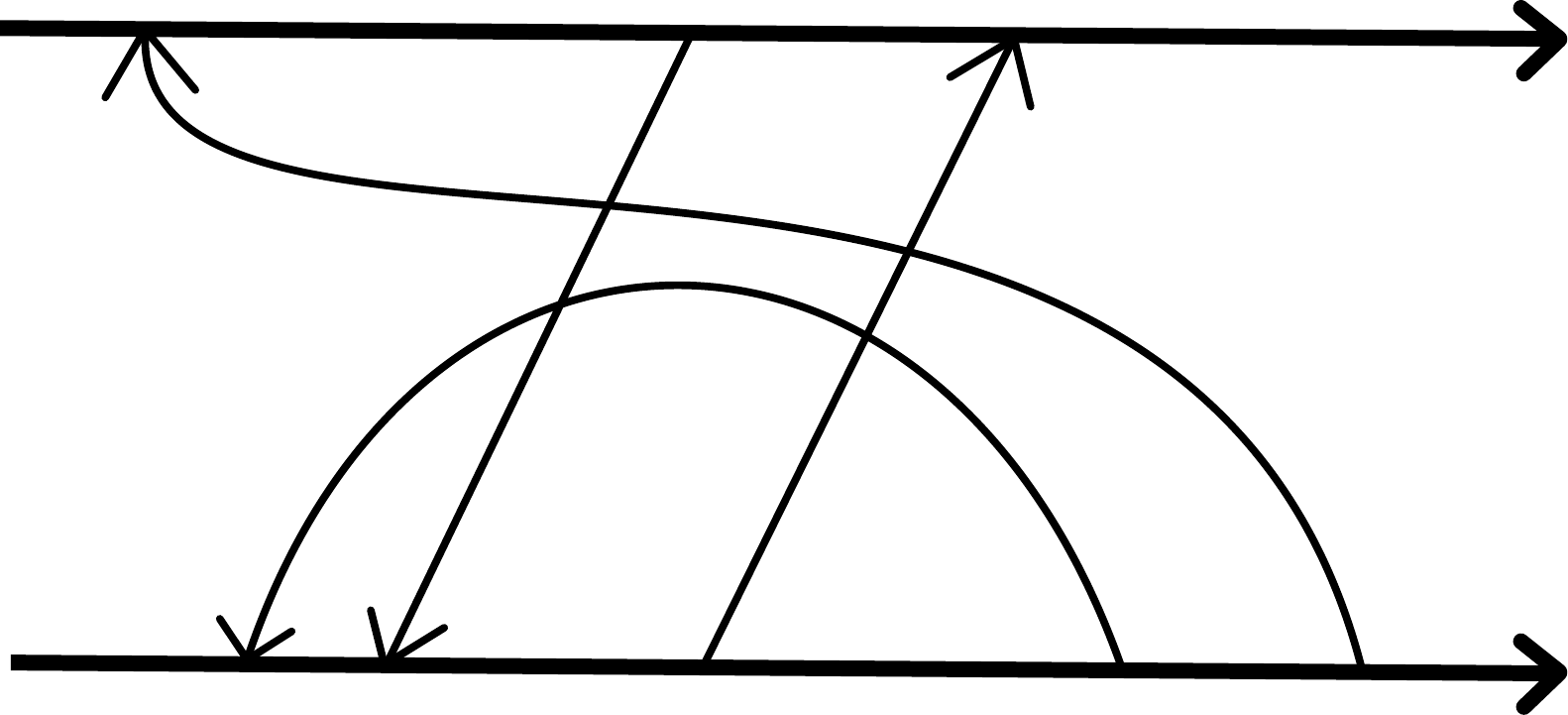}}
\end{equation*}
\end{example}

We will write $\mathcal{GD}(X)$ for the set of Gauss diagrams on $X$, modulo the relations $(G2)$ and $(G3)$ from \cref{subsec:virtual_links} and additionally the  analogues of the moves $(\Omega 1 f)$, $(\Omega 2')$ and $(\Omega 3')$ below 
\begin{equation*}
\centre{
\labellist \small \hair 2pt
\pinlabel{\scriptsize  $+$}  at 50 93
\pinlabel{\scriptsize  $-$}  at 470 93
\pinlabel{$\leftrightsquigarrow$}  at 600 63
\pinlabel{{\scriptsize $G 1f$}}  at 600 120
\endlabellist
\centering
\includegraphics[width=0.5\textwidth]{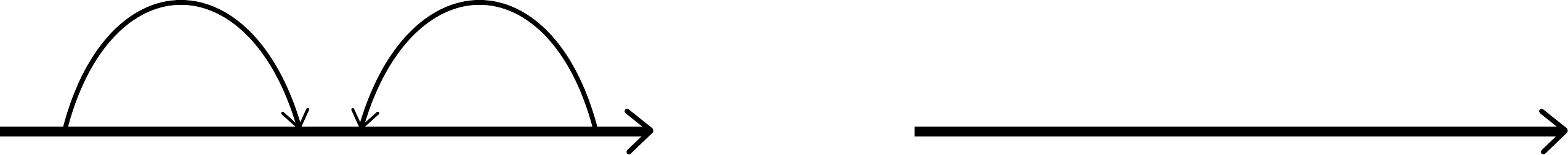}}
\end{equation*}
\begin{equation*}
\centre{
\labellist \small \hair 2pt
\pinlabel{$\leftrightsquigarrow$}  at 625 125
\pinlabel{{\scriptsize $G 2'$}}  at 625 185
\pinlabel{\scriptsize  $+$}  at 193 70
\pinlabel{\scriptsize  $-$}  at 322 70
\endlabellist
\centering
\includegraphics[width=0.5\textwidth]{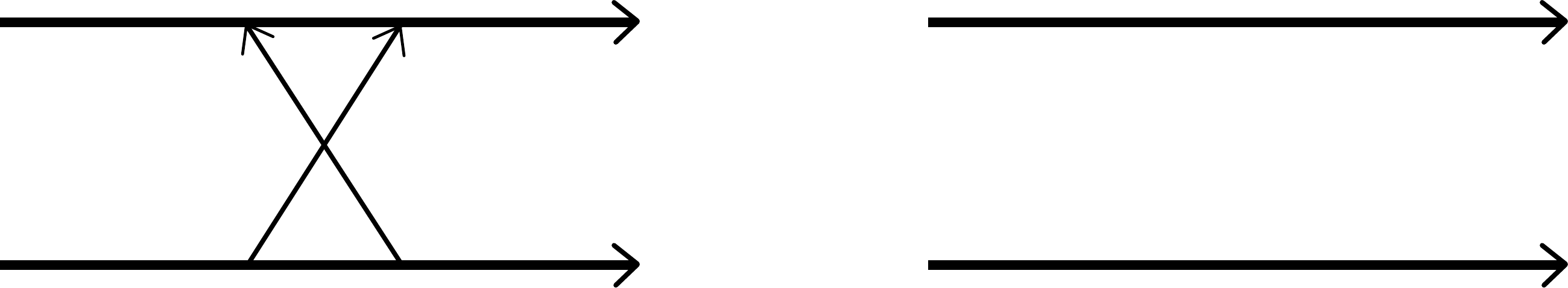}}
\end{equation*}
\begin{equation*}
\centre{
\labellist \small \hair 2pt
\pinlabel{\scriptsize  $+$}  at 120 350
\pinlabel{\scriptsize  $+$}  at 366 343
\pinlabel{\scriptsize  $-$} at 145 86
\pinlabel{$\leftrightsquigarrow$}  at 616 211
\pinlabel{{\scriptsize $G 3'$}}  at 616 270
\pinlabel{\scriptsize  $+$}  at 860 348
\pinlabel{\scriptsize  $+$}  at 1125 343
\pinlabel{\scriptsize  $-$} at 1100 90
\endlabellist
\centering
\includegraphics[width=0.5\textwidth]{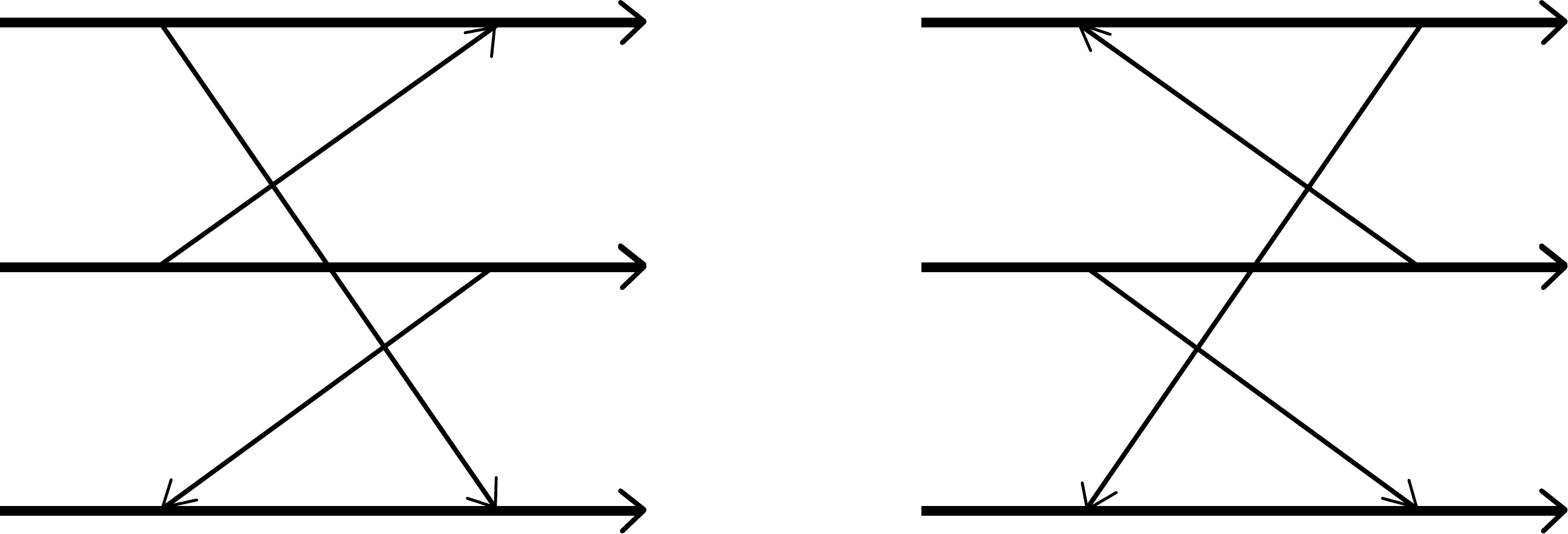}}
\end{equation*}
Similarly, we write  $\mathcal{GD}_u(X)$ for the quotient of $\mathcal{GD}(X)$ modulo the relation $(G1)$ from  \cref{subsec:virtual_links}; equivalently, this is the set of Gauss diagrams on $X$ modulo the relations $(G1)$ -- $(G3)$ from \cref{subsec:virtual_links}.

To make the connection with the categories $v\T$ and $v \T_u$ precise, we need the notion of polarised 1-manifold, that was introduced in \cite{habiromassuyeau} in the framework of the Kontsevich invariant. Say that a compact, oriented 1-manifold is \textit{polarised} if  the oriented 0-manifold $\partial X $ comes with a splitting   $\partial X = \partial_{bot} X \amalg \partial_{top} X$ and each part is endowed with a total order. For a polarised 1-manifold $X$, define the \textit{target} $t(X) \in \mathrm{Mon}(+,-)$ of $X$ as the element resulting from $\partial_{top} X$ by replacing every positive (res. negative) point by $+$ (resp. $-$) according to the total order. Define the \textit{source} $s(X) \in \mathrm{Mon}(+,-)$ of $X$ similarly using $\partial_{bot} X$ but replacing every positive (res. negative) point by $-$ (resp. $+$).

\begin{example}
Every virtual tangle gives rise to a polarised structure on the   1-manifold  $\left( \coprod_n D^1 \right) \coprod \left( \coprod_m  S^1 \right) $ in a natural way. 
\end{example}

We can now define the category $\mathcal{GD}$ of \textit{Gauss diagrams on polarised 1-manifolds} as follows: its objects are elements of $\mathrm{Mon}(+,-)$, and the hom-sets are given by 
\begin{equation}\label{eq:def_homGD}
\hom{\mathcal{GD}}{s}{t} := \coprod_{X} \mathcal{GD} (X),
\end{equation}
where $X$ runs through homeomorphism classes of polarised 1-manifolds such that $s(X)= s$ and $t(X)=t$. The composite law is defined as $$(X_2,D_2) \circ (X_1, D_1) := (X_1 \cup_{t(X_1)=s(X_2)} X_2, D_1 \amalg D_2)$$ and the identity of $w \in \mathrm{Mon}(+,-)$ is the empty chord diagram on the polarised 1-manifold $\uparrow_w$, the disjoint union of oriented intervals with orientations determined by $w$. The monoid product of $\mathrm{Mon}(+,-)$ and the disjoint union of chord diagrams determine a structure of strict monoidal category on $\mathcal{GD}$. The category $\mathcal{GD}_u$ is defined in an analogous way using the sets $\mathcal{GD}_u(X)$. Observe that this categorifies the set of Gauss diagrams on disjoint union of circles, that is, $$ \hom{\mathcal{GD}_u}{\emptyset}{\emptyset} \cong \mathsf{GD}  .$$ Note also that there is a canonical strict monoidal full functor $\mathcal{GD} \twoheadrightarrow \mathcal{GD}_u$.

In analogy to what we did for tangles, we will highlight a certain subcategory of Gauss diagrams. The category $\mathcal{GD}^{\mathrm{up}}$ of \textit{upwards Gauss diagrams} is the monoidal subcategory of $\mathcal{GD}$ on the objects $\mathrm{Mon}(+) \subset \mathrm{Mon}(+,-)$ and arrows Gauss diagrams on polarised 1-manifolds without closed components.

The following theorem  is proven in a similar fashion to the link case in \cref{thm:3_descriptions}. 

\begin{theorem}\label{thm:3_equivalences}
We have the following commutative diagram of strict monoidal categories and strict monoidal functors, where the horizontal functors are in fact  isomorphisms:
$$
\begin{tikzcd}
v\Tup \dar[hook] \rar{\cong} & \mathcal{GD}^{\mathrm{up}} \dar[hook] \\
v\T \dar[two heads] \rar{\cong} & \mathcal{GD} \dar[two heads] \\
v\T_u \rar{\cong} & \mathcal{GD}_u 
\end{tikzcd}
$$
\end{theorem}

We would like to remark that, in analogy to the knot and link case, each of the isomorphisms in the previous theorem is closely related to the theory of finite type invariants of virtual tangles, see also \cref{sec:fti}. More generally, it is possible to adapt the arguments of \cite{polyak00,ohtsukibook} to define a universal finite type invariant of framed tangles $$ Z: v\T \to \hat{\vec{\mathcal{A}}} ,$$ where $\hat{\vec{\mathcal{A}}} $ stands for the ``degree-completion of the category of oriented acyclic Jacobi diagrams on polarised 1-manifolds'', which is the virtual analogue of the Kontsevich invariant.

\section{XC-tangles and their Gauss diagrams}

The three isomorphisms from \cref{thm:3_equivalences} express that virtual crossings of virtual tangles are actually not there and that they are a deficiency of the geometrical description requiring to view the virtual tangle on the plane. In particular they tell us that taking a combinatorial approach, like that of Gauss diagrams, is more natural to the study of virtual tangles. However with Gauss diagrams we lose sight of the topological shape of the tangle, that sometimes is good to keep in mind.

The approach that we develop here will reconcile these two sides. This is heavily motivated by the construction of quantum knots invariants, where not only one pays attention to the crossings but also to their ``relative position'', in the sense that the winding or rotation number of an edge merging two crossings is an important piece of data to take into account.

For the rest of the paper, we will focus on open tangles, that is, tangles that have no closed components. It is possible to incorporate closed components to our descriptions, but it adds some complications to our definitions that we prefer not to deal with. The reader is however asked to be mindful that, unlike the non-virtual case, the theories of virtual knots and virtual long knots (i.e. one-component virtual tangles) differ, see \cite{GPV}.

\subsection{XC-tangles}\label{subsec:XC-tangles}

We start by introducing the central notion of this paper, namely that of an XC-tangle. Essentially, this is a refined version of the  ``snarl diagrams'' introduced by Bar-Natan and van der Veen in \cite{barnatanveenpolytime}, see also \cite{barnatanveengaussians}.

An \textit{(ordered, framed, oriented, open) XC-tangle diagram} is a directed  graph $T$ with vertices of valence 1, 2 or 4 such that
\begin{enumerate}[label=(\textit{\roman*})]
\item each tetravalent vertex is signed (i.e., carries a sign $+$ or $-$) and  \textit{vertex oriented}, that is, equipped with a cyclic order of the incident half-edges, such that two adjacent half-edges are incoming and two outgoing,
\item every edge is labelled by an integer $-1,0$ or $+1$, called its \textit{rotation number},
\item every bivalent vertex has one incoming half-edge and one outcoming half-edge, 
\item each maximal directed path going straight at the tetravalent vertices (that we call \textit{strand}) starts and finishes at univalent vertices,
\item each of the sets  $V_{out}$ and $ V_{in}  $  of univalent vertices adjacent to outgoing or incoming half-edges  is endowed with a total order.
\end{enumerate}

We will write $\mathcal{S}(T)$ for the set of strands of $T$. We will endow this set with the total order induced by $V_{out}$. Note also that considering the endpoints of each strand gives rise to a preferred bijection $V_{out} \toiso V_{in}  $.

%
%

Even though a XC-tangle diagram is an abstract graph, it will be convenient to make  pictures. As customary in knot theory, tetravalent vertices with the sign $+$ (resp. $-$) will be depicted as a positive (resp. negative) crossing: 
\begin{equation}\label{eq:pos_neg_crossing}
\centre{
\labellist \small \hair 2pt
\endlabellist
\centering
\includegraphics[width=0.18\textwidth]{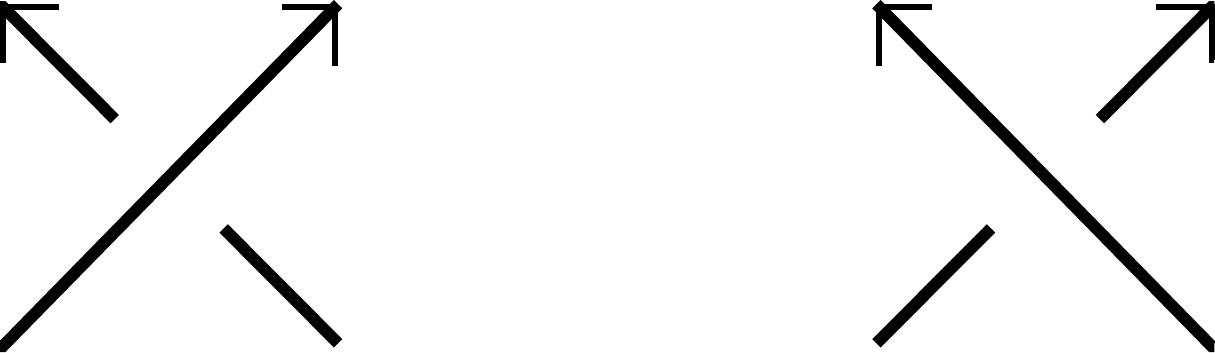}}
\end{equation}
The rotation numbers of each edge are labelled $[-1]$, $[0]$ or $[1]$. In general, we omit labelling the rotation number zero in edges. Edges with rotation numbers $+1$ and $-1$ will be sometimes depicted as a full positive and negative rotation, respectively:
\begin{equation}\label{eq:rotations}
\centre{
\labellist \small \hair 2pt
\endlabellist
\centering
\includegraphics[width=0.18\textwidth]{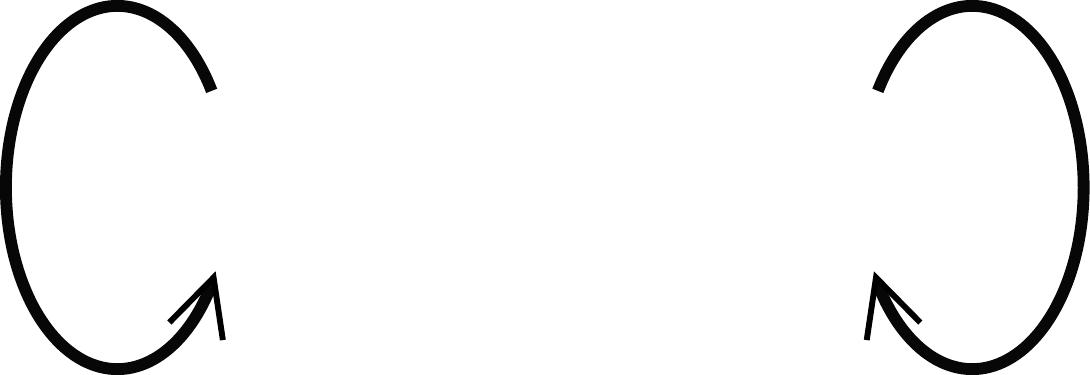}}
\end{equation}
If $p$ consecutive edges of rotation number $+1$ or $-1$ appear, we will sometimes depict a single edge labelled with $[p]$ or $[-p]$, respectively.

\begin{example}\label{ex:XC_tangle}
This is a picture of a (planar) XC-tangle diagram with three strands:
\begin{equation*}
\centre{
\labellist \small \hair 2pt
\pinlabel{\scriptsize  $1$}  at 396 406
\pinlabel{\scriptsize  $3$}  at 28 211
\pinlabel{\scriptsize  $2$}  at 200 116
\pinlabel{\scriptsize  $3$}  at 411 246
\pinlabel{\scriptsize  $1$}  at 260 750
\pinlabel{\scriptsize  $2$}  at 30 750
\pinlabel{\scriptsize  $[-2]$}  at -75 490
\pinlabel{\scriptsize  $[3]$}  at 253 633
\endlabellist
\centering
\includegraphics[width=0.25\textwidth]{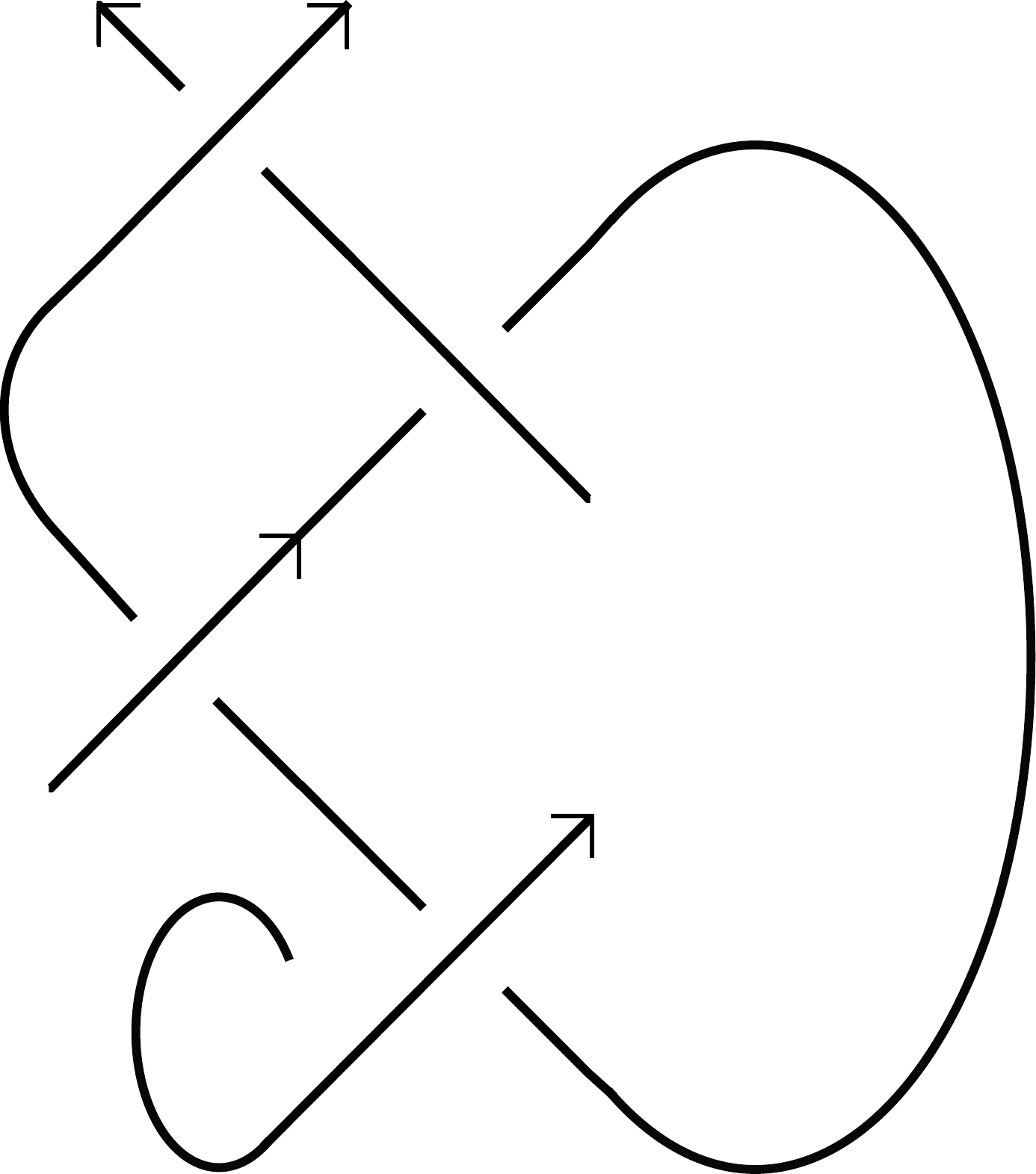}}
\end{equation*}
The numbering next to the outgoing univalent vertices denotes  the ordering of $V_{out}$ (and therefore also of the set of strands), whereas the numbering next to the incoming univalent vertices indicates the ordering of $V_{in}$. 
\end{example}

An \textit{XC-tangle} is an equivalence class of an XC-tangle diagram under the equivalence relation generated by the following XC-analogues of the (real)  framed Reidemeister moves:

\begin{equation*}
\centre{
\labellist \small \hair 2pt
\pinlabel{$,$}  at 1320 390
\pinlabel{$\overset{\widehat{\Omega} 0a}{\leftrightsquigarrow}$}  at 800 440
\pinlabel{$\overset{\widehat{\Omega} 0b}{\leftrightsquigarrow}$}  at 2300 440
\endlabellist
\centering
\includegraphics[width=0.75\textwidth]{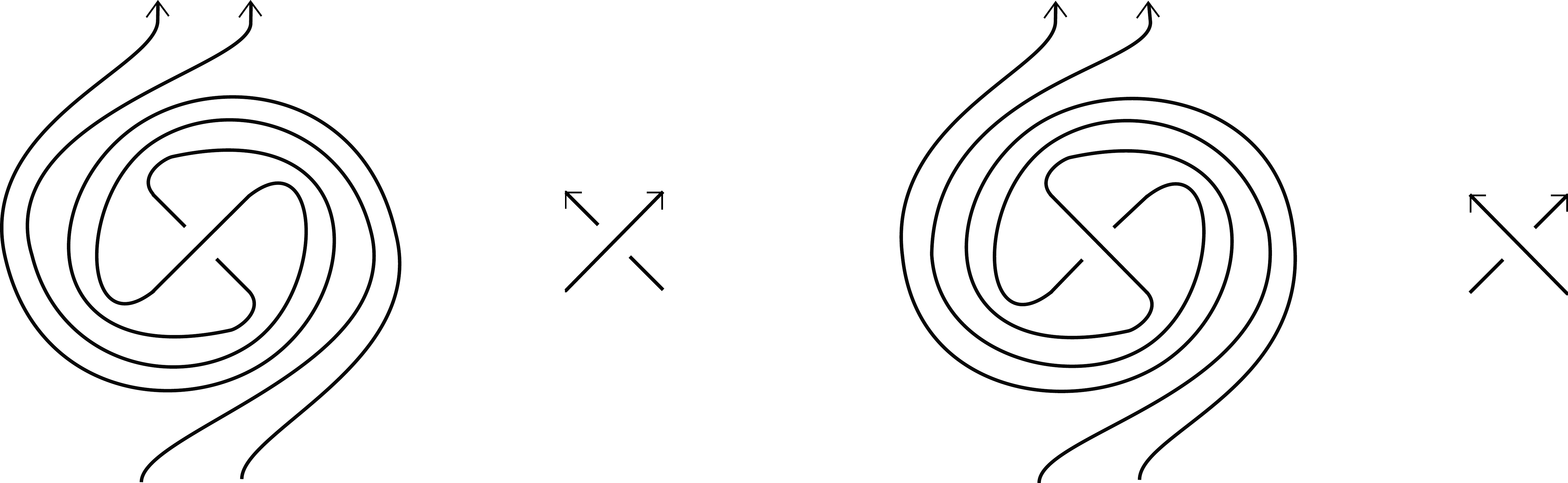}}
\end{equation*}
\vspace*{10pt}
\begin{equation*}
\centre{
\labellist \small \hair 2pt
\pinlabel{$\overset{\widehat{\Omega} 0r}{\leftrightsquigarrow}$}  at 187 230
\pinlabel{$\overset{\widehat{\Omega} 2a}{\leftrightsquigarrow}$}  at 2367 230
\pinlabel{$\overset{\widehat{\Omega} 2b}{\leftrightsquigarrow}$}  at 2800 230
\pinlabel{$\overset{\widehat{\Omega} 1f}{\leftrightsquigarrow}$}  at 1240 230
\pinlabel{\tiny $[i]$}  at 60 104
\pinlabel{\tiny $[j]$}  at 60 300
\pinlabel{\tiny $[i+j]$}  at 490 300
\pinlabel{$,$}  at 580 170
\pinlabel{$,$}  at 1870 170
\endlabellist
\centering
\includegraphics[width=0.98\textwidth]{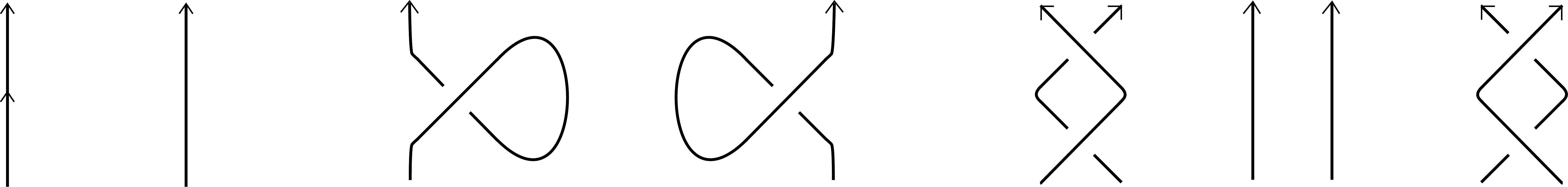}}
\end{equation*}
\vspace*{10pt}
\begin{equation*}
\centre{
\labellist \small \hair 2pt
\pinlabel{$\overset{\widehat{\Omega} 2c}{\leftrightsquigarrow}$}  at 400 300
\pinlabel{$\overset{\widehat{\Omega} 2d}{\leftrightsquigarrow}$}  at 1617 300
\pinlabel{$\overset{\widehat{\Omega} 3}{\leftrightsquigarrow}$}  at 2930 300
\pinlabel{$,$}  at 1030 275
\pinlabel{$,$}  at 2260 275
\endlabellist
\centering
\includegraphics[width=0.98\textwidth]{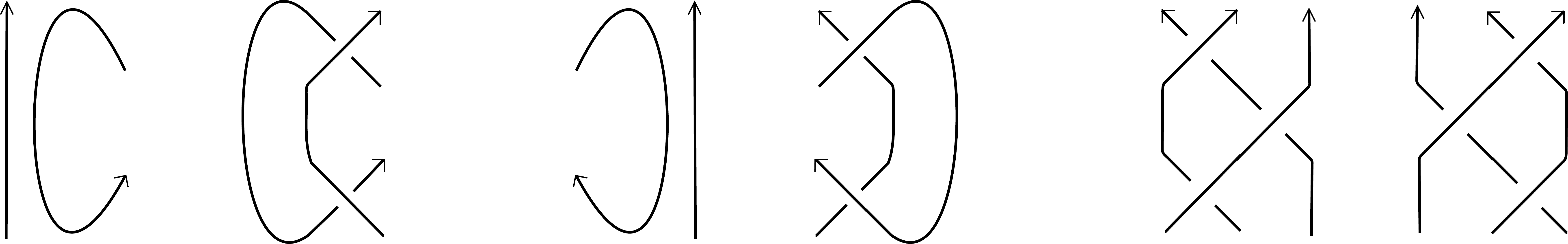}}
\end{equation*}
%
%
\vspace*{5pt}

\noindent Let us explain precisely what we mean. To begin with, each of the previous pictures represents a certain (decorated) subgraph of a given tangle. We identify two XC-tangles that are identical except in a certain subgraph, where they are as shown.  In  $(\widehat{\Omega} 0r)$, we have $i,j=-1,0,1$ (and when $(i,j)=(1,1),(-1,-1)$ this is to be understood as notational).

We ask the reader to be mindful that despite these pictures are planar, no planarity is assumed in XC-tangles: these should be thought of as local pictures of abstract graphs, that are depicted according to the above conventions.

The collection of XC-tangles can be assembled into a strict symmetric monoidal category $\T^{\mathrm{XC}}$ as follows: its objects are the nonnegative integers, and  $\operatorname{End}_{\TXC}(n)$ is the set of XC-tangles with $n$ strands, with no arrows between different objects. The composite $ T_2 \circ T_1 $ of two XC-tangles with $n$ strands is given by glueing the incoming univalent vertices of $T_1$ to the outgoing univalent vertices of $T_2$ in an orderly fashion. The identity of the object $n$ is the ordered disjoint union of $n$ oriented intervals, each of them with rotation number zero, with the $i$-th strand going from the $i$-th outcoming vertex to the $i$-th incoming vertex.

The monoidal product is given by addition of integers for objects,  and by the disjoint union of graphs for the arrows. The sets in/out of univalent vertices are ordered according to the ordinal sum, $$ V_{out}(T \amalg T')= V_{out}(T ) \amalg V_{out}(T') \qquad , \qquad V_{in}(T \amalg T')= V_{in}(T ) \amalg V_{in}(T') $$
so that $x < y$ whenever $x \in V_{out}(T ) $ and $y \in V_{out}(T')$ (or $x \in V_{in}(T ) $ and $y \in V_{in}(T')$). The unit of the monoidal structure is the object 0. The symmetric braiding $P_{n,m}: n+m \to n+m$ is the XC-tangle which consists of $n+m$ oriented intervals where the $i$-th strand connects the $i$-th outgoing vertex with the $(i+n)$-th incoming vertex (mod $n+m$).

\subsection{XC-Gauss diagrams}\label{sec:XC-Gauss}

Next we would like to describe the XC-analogue of Gauss diagrams. Essentially this idea already appeared in \cite[\S 4.5]{becerra_thesis} under the name ``rotational Gauss diagram''.

Here is the definition: let $X$ be a compact, oriented 1-manifold without closed components (that is, a disjoint union of oriented intervals). An \textit{XC-Gauss diagram} on $X$ is  a Gauss diagram $D$ on $X$ together with a finite marking $P$ in $\mathrm{int}(X)- \partial D$, where each of the marked points carries a sign $+$ or $-$. We  will identify two XC-Gauss diagrams $(D, P), (D', P ')$ on $X$ if there is an orientation-preserving homeomorphism of pairs $(X\cup D, X) \toiso (X \cup D', X)$ which also preserves  the signs of the chords, the marking and the signs of the marked points.

In pictures, we will decorate with the symbol $\blacklozenge$ each of the marked points on $\mathrm{int}(X)- \partial D$. If $p$ consecutive markings with the same sign appear in a given path-component of $\mathrm{int}(X) - \partial D$, we will sometimes write a single decoration $\blacklozenge$ with the label $[p]$ or $[-p]$, depending on the sign.

\begin{example}

The following is an XC-Gauss diagram on the disjoint union of two oriented intervals:
\begin{equation*}
\centre{
\labellist \scriptsize \hair 2pt
\pinlabel{ $  +$}  at 115 243
\pinlabel{$  +$}  at 340 288
\pinlabel{$  -$}  at 330 95
\pinlabel{$  -$}  at 100 112
\pinlabel{\normalsize $  \blacklozenge$}  at 441 24
\pinlabel{\normalsize $  \blacklozenge$}  at 193 332
\pinlabel{\normalsize $  \blacklozenge$}  at 626 330
\pinlabel{$  +$}  at 441 -40
\pinlabel{$  +$}  at 193 400
\pinlabel{\scriptsize $  [3]$}  at 626 390
\endlabellist
\centering
\includegraphics[width=0.25\textwidth]{figures/ex_gauss_diag}}
\end{equation*}
\end{example}

We will write $\mathcal{GD}^{\mathrm{XC}}(X)$ for the set of Gauss diagrams on $X$, modulo the equivalence relation generated by the following analogues of the XC-Reidemeister moves, with $\varepsilon=+,-$:
\vspace*{7pt}
\begin{equation*}
\centre{
\labellist \small \hair 2pt
\pinlabel{\scriptsize $\varepsilon$}  at 300 150
\pinlabel{\scriptsize $\varepsilon$}  at 1000 150
\pinlabel{\normalsize $  \blacklozenge$}  at 835 19
\pinlabel{\normalsize $  \blacklozenge$}  at 835 213
\pinlabel{\normalsize $  \blacklozenge$}  at 1080 19
\pinlabel{\normalsize $  \blacklozenge$}  at 1080 213
\pinlabel{\scriptsize $+$}  at 835 -50
\pinlabel{\scriptsize $+$}  at 835 290
\pinlabel{\scriptsize $-$}  at 1080 -50
\pinlabel{\scriptsize $-$}  at 1080 290
\pinlabel{$,$}  at 1423 107
\pinlabel{$\overset{\widehat{G} 0}{\leftrightsquigarrow}$}  at 600 150
\pinlabel{\normalsize $  \blacklozenge$}  at 1790 106
\pinlabel{\normalsize $  \blacklozenge$}  at 1945 106
\pinlabel{\normalsize $  \blacklozenge$}  at 2644 106
\pinlabel{$\overset{\widehat{G} 0r}{\leftrightsquigarrow}$}  at 2252 140
\pinlabel{\scriptsize $[i]$}  at 1790 180
\pinlabel{\scriptsize $[j]$}  at 1945 180
\pinlabel{\scriptsize $[i+j]$}  at 2644 180
\endlabellist
\centering
\includegraphics[width=0.98\textwidth]{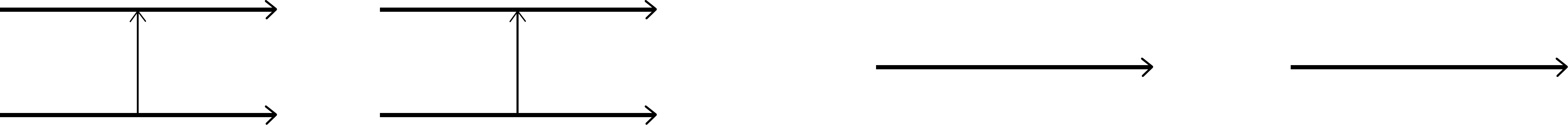}}
\end{equation*}
\vspace*{15pt}
\begin{equation*}
\centre{
\labellist \small \hair 2pt
\pinlabel{\normalsize $  \blacklozenge$}  at 256 20
\pinlabel{\normalsize $  \blacklozenge$}  at 944 20
\pinlabel{\scriptsize $-$}  at 256 -50
\pinlabel{\scriptsize $+$}  at 370 120
\pinlabel{\scriptsize $+$}  at 1057 120
\pinlabel{\scriptsize $+$}  at 944 -50
\pinlabel{\scriptsize $\varepsilon$}  at 1720 160
\pinlabel{\scriptsize $-\varepsilon$}  at 1990 160
\pinlabel{$\overset{\widehat{G} 1f}{\leftrightsquigarrow}$}  at 590 50
\pinlabel{$\overset{\widehat{G} 2}{\leftrightsquigarrow}$}  at 2210 130
\pinlabel{$,$}  at 1423 107
\endlabellist
\centering
\includegraphics[width=0.98\textwidth]{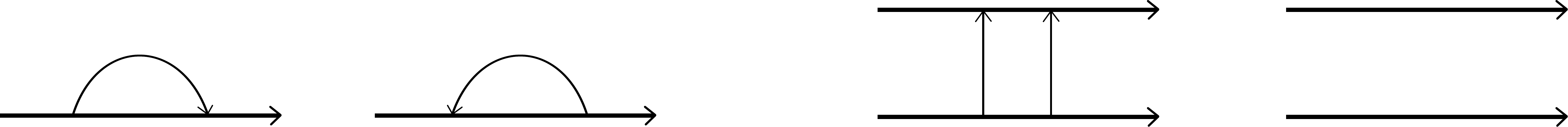}}
\end{equation*}
\vspace*{15pt}
\begin{equation*}
\centre{
\labellist \small \hair 2pt
\pinlabel{$,$}  at 1423 187
\pinlabel{\normalsize $  \blacklozenge$}  at 222 117
\pinlabel{\normalsize $  \blacklozenge$}  at 928 117
\pinlabel{\scriptsize $\varepsilon$}  at 222 50
\pinlabel{\scriptsize $\varepsilon$}  at 928 50
\pinlabel{\scriptsize $\varepsilon$}  at 817 257
\pinlabel{\scriptsize $-\varepsilon$}  at 1030 257
\pinlabel{$\overset{\widehat{G} 2'}{\leftrightsquigarrow}$}  at 590 240
\pinlabel{$\overset{\widehat{G} 3}{\leftrightsquigarrow}$}  at 2217 240
\pinlabel{\scriptsize $+$}  at 1650 140
\pinlabel{\scriptsize $+$}  at 1800 340
\pinlabel{\scriptsize $+$}  at 2080 340
\pinlabel{\scriptsize $+$}  at 2360 340
\pinlabel{\scriptsize $+$}  at 2650 340
\pinlabel{\scriptsize $+$}  at 2810 140
\endlabellist
\centering
\includegraphics[width=0.98\textwidth]{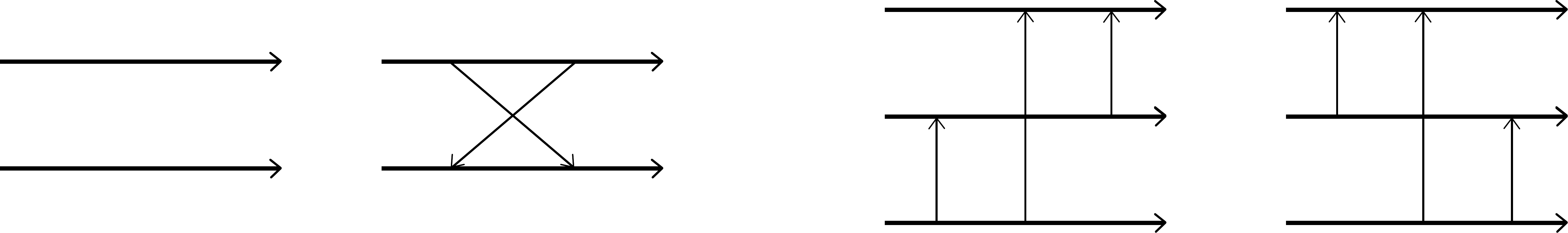}}
\end{equation*}

These relations should be understood in the same way as in  \cref{subsec:virtual_links}, that is, analogously to \eqref{eq:diagram_for_GD} we should have a diagram
\begin{equation}
\begin{tikzcd}[row sep=0.2em]
& (\amalg_k  D^1 , D_L) \rar[hook] & (X, D) \\
(\amalg_k D^1, \emptyset) \arrow[hook]{ur}  \arrow[hook]{dr} & &\\
& (\amalg_k  D^1, D_R) \rar[hook] & (X, D')
\end{tikzcd}
\end{equation}
where the embeddings should of course preserve the marking on $X$ and their signs.

We can now define the category $\mathcal{GD}^{\mathrm{XC}}$ of \textit{XC-Gauss diagrams on polarised open 1-manifolds} as follows: its objects are nonnegative integers, the category has only endomorphisms and 
\begin{equation}\label{eq:def_homGDXC}
\eend{\mathcal{GD}^{\mathrm{XC}}}{n} := \coprod_{X} \mathcal{GD}^{\mathrm{XC}} (X),
\end{equation}
where $X$ runs through homeomorphism classes of polarised open 1-manifolds such that $s(X)= t(X)=+^n$. The composite law is defined as $$(X_2,D_2,P_2) \circ (X_1, D_1,P_1) := (X_1 \cup_{t(X_1)=s(X_2)} X_2, D_1 \amalg D_2, P_1 \amalg P_2)$$ and the identity of the object $n$ is the empty chord diagram on the polarised 1-manifold $\uparrow_{+^n}$ (without diamond decorations). We can furthermore endow $\mathcal{GD}^{\mathrm{XC}}$ with a strict symmetric  monoidal structure: the monoidal product is simply given by addition of integers and disjoint union of XC-Gauss diagrams, and the symmetric braiding  $P_{n,m}: n+m \to n+m$ is the XC-Gauss diagram given by $n+m$ oriented intervals with no chords and no diamond decorations which consists of $n+m$ oriented intervals, which is polarised in the following way: every oriented interval has as initial point the $i$-th bottom endpoint and as final point the  $(i+n)$-th top endpoint (mod $n+m$).

\begin{construction}\label{constr:XCT->XCGD}
We will now associate, to every XC-tangle, a XC-Gauss diagram on a certain polarised open 1-manifold.

Let $T$ be an XC-tangle with $n$ strands. The skeleton $X$ of its XC-Gauss diagram will be a disjoint union of $n$ oriented intervals. We polarise $X$ as follows: first, we declare all positive (resp. negative) points of $\partial X$ to be in $\partial_{top}X$ (resp. $\partial_{bot}X$). Then, choose an arbitrary bijection $\mathcal{S}(T) \toiso \pi_0(X)$ and  set on $\pi_0(X)$ the total order induced by it. Identifying bottom points of $X$ with outcoming univalent vertices of $T$ and top points with incoming vertices according to this bijection induces further bijections
$$ V_{out} \toiso  \partial_{bot}X \qquad , \qquad V_{in} \toiso  \partial_{top}X, $$ and we set on $\partial_{bot}X$ and $\partial_{top}X$ the total order induced by them.

Next we attach as many chords as tetravalent vertices in $T$, with their corresponding sign. Using the preferred bijection $\mathcal{S}(T) \toiso \pi_0(X)$, we divide each path-component of $X$ in as many subintervals as edges the corresponding  strand of $T$ has. Then, for every tetravalent vertex in $T$, we place a chord with the same sign between the corresponding points on $X$ oriented from the overpass to the underpass (according to \eqref{eq:pos_neg_crossing}). Moreover, we place a diamond decoration on $X$ for every edge on $T$ with rotation number $\pm 1$, carrying the same sign.

It is clear that the construction is independent of the choice of the preferred identification $\mathcal{S}(T) \toiso \pi_0(X)$, since any other will give rise to a homeomorphism of pairs $(X\cup D, X) \toiso (X \cup D, X)$ permuting the path-components and preserving all the structure data. More importantly, under this construction the XC-Reidemeister moves are taken precisely to their analogues in terms of XC-Gauss diagrams, so the passage descents to equivalence classes.
\end{construction}

\begin{theorem}\label{thm:TXC=GDXC}
The previous construction defines an isomorphism of strict symmetric monoidal categories $$ \TXC \overset{\cong}{\to}   \mathcal{GD}^{\mathrm{XC}}.$$
\end{theorem}
\begin{proof}
By the paragraph above, we have a well-defined map $$\eend{\TXC}{n} \to   \eend{\mathcal{GD}^{\mathrm{XC}}}{n}. $$ Its inverse is constructed similarly, considering a disjoint union of as many crossings as chords the XC-Gauss diagram has and merging them as we read off the vertices of $D$ on $X$ along the orientation, and  adding additional edges with rotation number $\pm 1$ as we encounter diamond decorations on $X$.  This means that \cref{constr:XCT->XCGD} induces an isomorphism of categories $$ \TXC \overset{\cong}{\to}   \mathcal{GD}^{\mathrm{XC}}$$ which is the identity on objects.

That the functor is strict symmetric monoidal essentially follows by construction: the disjoint union of XC-tangles is sent to the disjoint union of XC-Gauss diagrams and the symmetric braiding on one side is sent to the symmetric braiding on the other side. 
\end{proof}

\subsection{Relation with virtual upwards tangles}\label{subsec:vTup}

Now we would like to explain how the setup of XC-tangles and XC-Gauss diagrams can be related with the classical virtual tangles. In fact, there are two different functorial passages that we can construct.

At the level of Gauss diagrams, there is a canonical strict symmetric monoidal full forgetful functor $$ \mathcal{GD}^{\mathrm{XC}} \longtwoheadrightarrow  \mathcal{GD}^{\mathrm{up}} $$ that simply forgets  the diamond decorations. We would like to describe the corresponding functor at the level of XC- and virtual upwards tangles.

\begin{construction}\label{constr:XC->virtual}
Given an XC-tangle $T$ with $n$ strands, we will construct an $n$-component  virtual upwards tangle. First, place on the interior of $D^1 \times D^1$ as many crossings as tetravalent vertices $T$ has, according to the pattern \eqref{eq:pos_neg_crossing} depending on the sign attached to the vertex.  Next subdivide $D^1 \times \partial D^1$ in $n$ uniformly distributed points, label the points of $D^1 \times \{ -1 \}$ and $D^1 \times \{ +1 \}$  with $1, \ldots , n$ from left to right. Write  $\sigma: V_{out } \toiso V_{in}$ for the bijection induced by taking endpoints. For every $i=1, \ldots, n$, merge the point of $D^1 \times \{ -1 \}$ labelled with $i$ with the point of $D^1 \times \{ +1 \}$ labelled with $\sigma(i)$  passing through the placed crossings according to the order dictated by the $i$-th strand of $T$. Whenever self-intersections are produced, place a singular crossing on it. Note that rotation numbers are completely ignored in this description.
\end{construction}

\begin{example}
The XC-tangle from \cref{ex:XC_tangle} gives rise, through the previous construction, to the following virtual upwards tangle:
\begin{equation*}
\centre{
\centering
\includegraphics[width=0.27\textwidth]{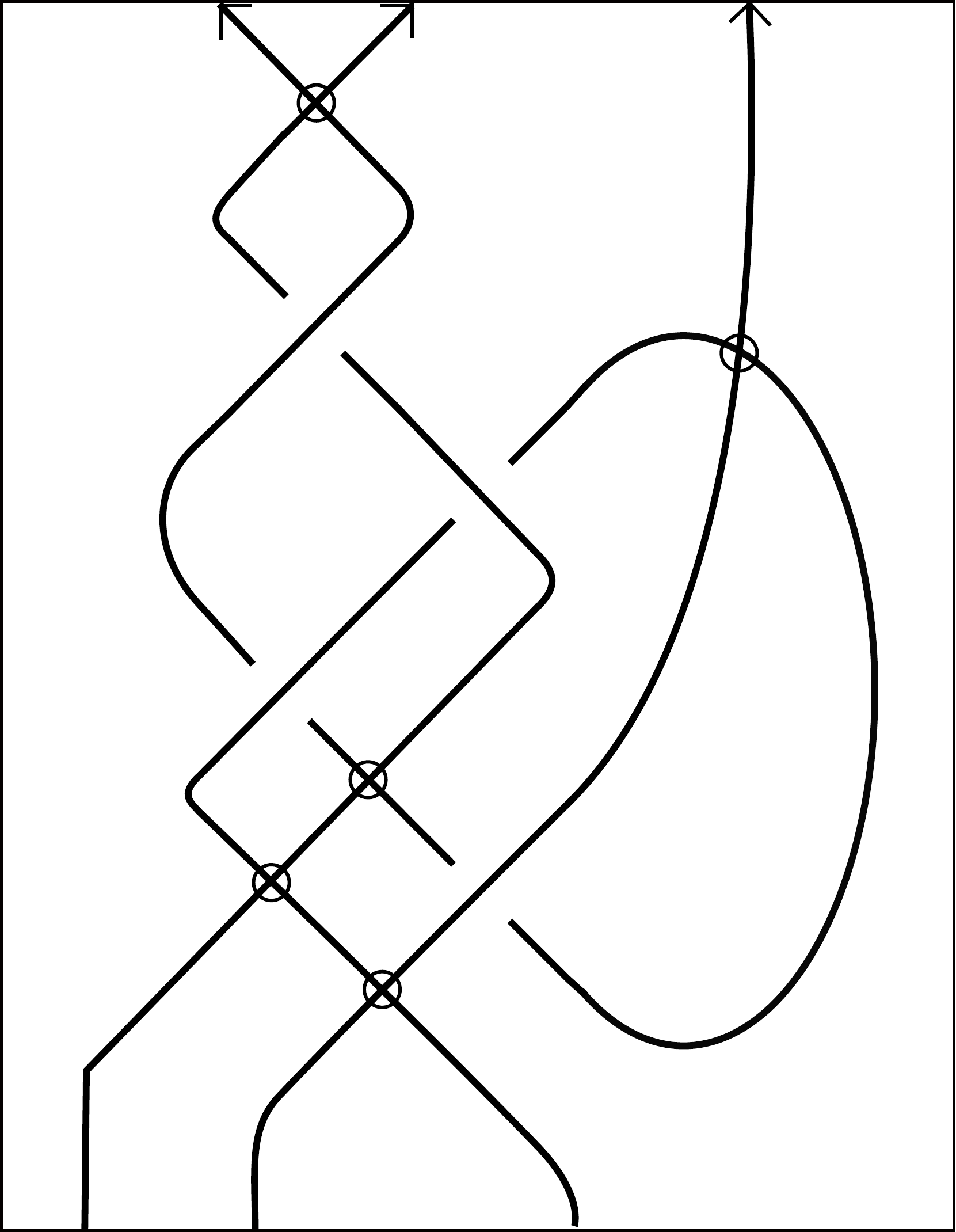}}
\end{equation*}
\end{example}

\begin{proposition}\label{prop:forgetful_XC_virtual}
The previous construction gives rise to a strict symmetric monoidal full  functor  $$ U: \T^{\mathrm{XC}} \longtwoheadrightarrow  v\T^{\mathrm{up}} $$ 
that makes the following diagram commute,
$$
\begin{tikzcd}
v\T^{\mathrm{XC}} \dar[two heads] \rar{\cong} & \mathcal{GD}^{\mathrm{XC}} \dar[two heads] \\
v\T^{\mathrm{up}} \rar{\cong} & \mathcal{GD}^{\mathrm{up}}
\end{tikzcd}
$$
\end{proposition}
\begin{proof}
The first step is to check that \cref{constr:XC->virtual} is independent of the choices involved. That the resulting virtual upwards tangle does not depend on the precise arcs chosen to merge the crossings is a consequence of the detour move (in the framed setting). Similarly, to check the independence with respect to the exact placement of the crossings in $D^1 \times D^1$  we can argue as follows: given a crossing and a different placement in the square, choose a path between the two different locations that intersects the arcs of the tangle transversally. For each of these intersections we can pass the crossing through the arc by means of the detour move (and planar isotopy),  
\begin{equation*}
\centre{
\labellist \scriptsize \hair 2pt
\pinlabel{$ =$}  at 410 200
\endlabellist
\centering
\includegraphics[width=0.35\textwidth]{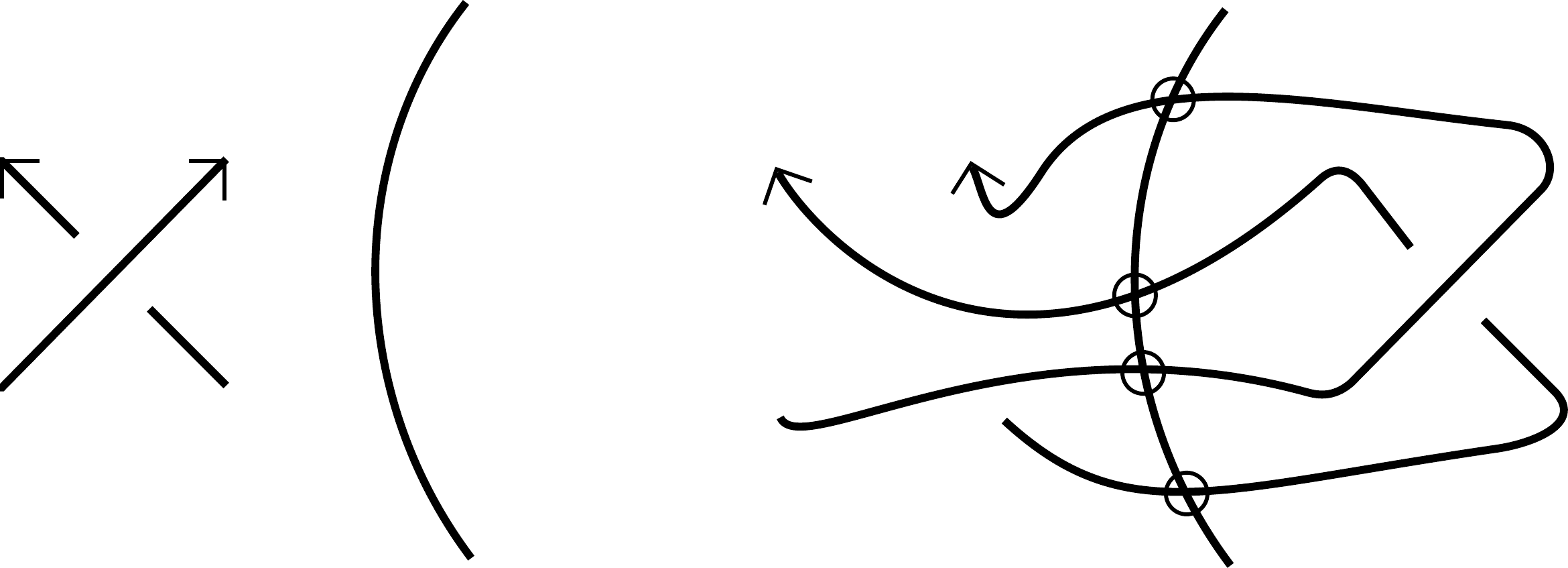}}
\end{equation*}
That \cref{constr:XC->virtual} is well-defined at the level of equivalence classes of XC-tangle diagrams is immediate: on the one hand, the relations $(\widehat{\Omega} 0a)$ -- $(\widehat{\Omega} 0c)$ are trivially preserved as the construction ignores the rotation numbers, and on the other hand the remaining moves  $(\widehat{\Omega} 1f)$ -- $(\widehat{\Omega} 3)$ are preserved as they are realised precisely as depicted in \cref{subsec:XC-tangles}.

All together, this implies that there is  a well-defined map $$\eend{\T^{\mathrm{XC}}}{n} \longtwoheadrightarrow \eend{v\T^{\mathrm{up}}}{n}$$
which is trivially surjective: for a given virtual upwards diagram $D$, an XC-tangle diagram with all edges having rotation number zero and mapping to $D$ under the previous map can be easily built simply by merging copies of $X^\pm$ according to the pattern dictated by $D$ (ignoring virtual crossings) and ordering the sets of univalent outgoing and incoming vertices  as dictated by the left-to-right order of the top and bottom endpoints of $D$. It is also clear that \cref{constr:XC->virtual} preserves the composition and the disjoint union of diagrams so  it indeed defines a full symmetric monoidal functor $U: \T^{\mathrm{XC}} \longtwoheadrightarrow  v\T^{\mathrm{up}}. $ Finally, the commutativity of the diagram follows from comparing \cref{constr:XCT->XCGD}, \cref{constr:XC->virtual}  and the description of the isomorphism $v\T^{\mathrm{up}} \cong  \mathcal{GD}^{\mathrm{up}}$ given in \cref{subsec:virtual_links}.
\end{proof}

It turns out that the forgetful functor $ U: \T^{\mathrm{XC}} \longtwoheadrightarrow  v\T^{\mathrm{up}} $, which arises canonically (at least  from the perspective of Gauss diagrams via \cref{prop:forgetful_XC_virtual}) admits a non-trivial section.  To construct this we first have to introduce rotational diagrams of upwards virtual tangles. In the non-virtual setting, these have already appeared in  \cite{becerra_gaussians,becerra_thesis, becerra_refined,BH_reidemeister} inspired by \cite{barnatanveenpolytime,barnatanveengaussians}

A diagram $D$ of a virtual upwards tangle is said to be  \textit{rotational} if all crossings  of $D$ (real and virtual) point upwards and all maxima and minima appear in pairs of the following two forms,
\begin{equation}\label{eq:full_rotations}
\centre{
\centering
\includegraphics[width=0.27\textwidth]{figures/rot_tangle}}
\end{equation}
where the dashed discs denote some neighbourhoods of that piece of strand in the tangle diagram.  
We regard rotational tangle diagrams up to \textit{Morse isotopy}\index{Morse isotopy}, that is, planar isotopy that preserve all maxima and minima. In other words, we do not allow isolated cups and caps (``half rotations''), instead they must appear in pairs, either  $\capl$ and  $\cupr$ or  $\capr$ and  $\cupl$,  forming full rotations, as depicted in \eqref{eq:full_rotations}. 

It was shown in \cite[Lemma 3.2]{becerra_refined} that every (non-virtual) upwards tangle has a rotational diagram. The proof   carries over verbatim to the virtual setting, which yields

\begin{lemma}\label{lem:vT_rot_diagrams}
Every virtual upwards tangle has a rotational diagram.
\end{lemma}

Roughly, this is achieved by using only planar isotopies in the original  upwards virtual tangle diagram, rotating all crossings so that they point up, and then adding full rotations wherever needed removing possibly created zig-zags, see below for an example:
\begin{equation*}
\centre{
\centering
\includegraphics[width=0.5\textwidth]{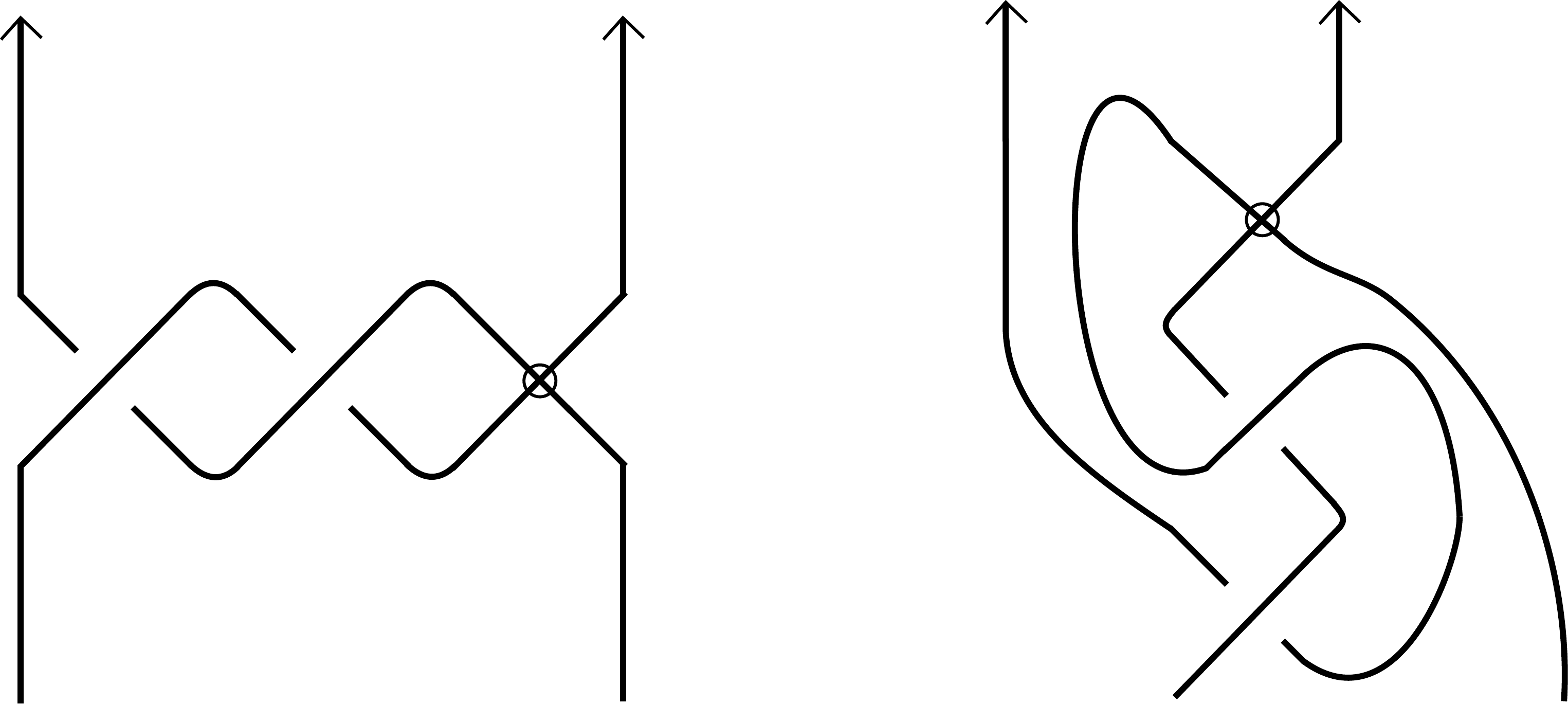}}
\end{equation*}

More remarkably, it is possible to stay in the realm of rotational diagrams by considering a rotational version of the Reidemeister moves. In the non-virtual setting, several generating sets of rotational  Reidemeister moves for framed and unframed tangles are described in \cite{BH_reidemeister}. Adapting the arguments therein, it is not difficult to see that the following family constitutes a generating set of rotational Reidemeister moves for virtual framed tangles: 
\begin{equation*}
\centre{
\labellist \small \hair 2pt
\pinlabel{$\leftrightsquigarrow$}  at 439 190
\pinlabel{{\scriptsize $\Omega 0a$}}  at 445 260
\pinlabel{$\leftrightsquigarrow$}  at 1625 190
\pinlabel{{\scriptsize $\Omega 0b$}}  at 1630 260
\pinlabel{$\leftrightsquigarrow$}  at 2865 190
\pinlabel{{\scriptsize $\Omega 0c$}}  at 2869 260
\endlabellist
\centering
\includegraphics[width=0.98\textwidth]{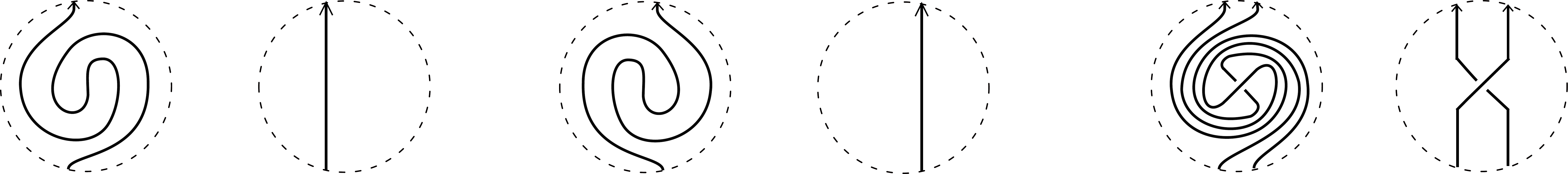}}
\end{equation*}
\begin{equation*}
\centre{
\labellist \small \hair 2pt
\pinlabel{$\leftrightsquigarrow$}  at 439 190
\pinlabel{{\scriptsize $\Omega 0d$}}  at 445 260
\pinlabel{$\leftrightsquigarrow$}  at 1625 190
\pinlabel{{\scriptsize $\Omega 1 \textup{f}e$}}  at 1630 260
\pinlabel{$\leftrightsquigarrow$}  at 2890 190
\pinlabel{{\scriptsize $\Omega 2a$}}  at 2890 260
\endlabellist
\centering
\includegraphics[width=0.98\textwidth]{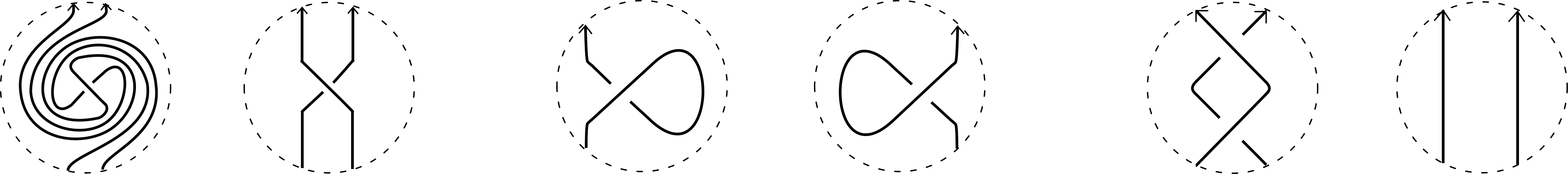}}
\end{equation*}
\begin{equation*}
\centre{
\labellist \small \hair 2pt
\pinlabel{$\leftrightsquigarrow$}  at 439 190
\pinlabel{{\scriptsize $\Omega 2b$}}  at 445 260
\pinlabel{$\leftrightsquigarrow$}  at 1635 190
\pinlabel{{\scriptsize $\Omega 2c1$}}  at 1650 260
\pinlabel{$\leftrightsquigarrow$}  at 2910 190
\pinlabel{{\scriptsize $\Omega 2d1$}}  at 2910 260
\endlabellist
\centering
\includegraphics[width=0.98\textwidth]{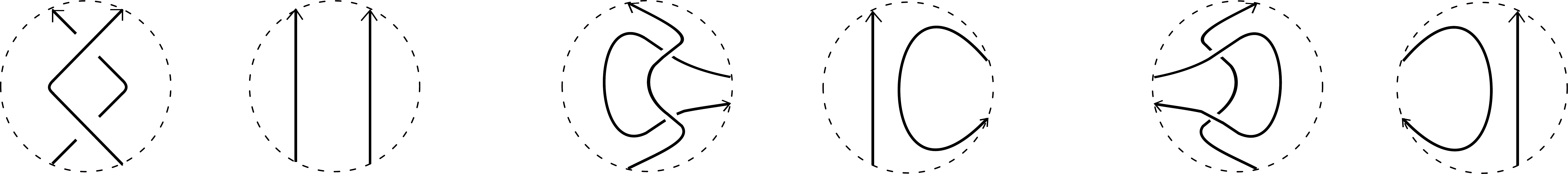}}
\end{equation*}
\begin{equation*}
\centre{
\labellist \small \hair 2pt
\pinlabel{$\leftrightsquigarrow$}  at 439 190
\pinlabel{{\scriptsize $\Omega 3b$}}  at 445 260
\pinlabel{$\leftrightsquigarrow$}  at 1635 190
\pinlabel{{\scriptsize $v\Omega 0$}}  at 1650 260
\pinlabel{$\leftrightsquigarrow$}  at 2910 190
\pinlabel{{\scriptsize $v\Omega 2$}}  at 2910 260
\endlabellist
\centering
\includegraphics[width=0.98\textwidth]{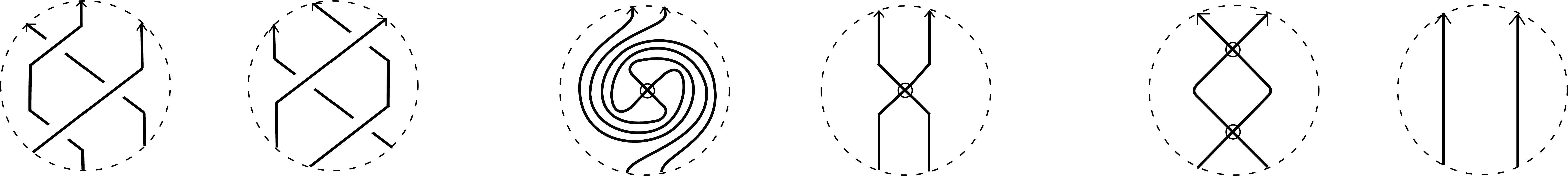}}
\end{equation*}
\begin{equation*}
\centre{
\labellist \small \hair 2pt
\pinlabel{$\leftrightsquigarrow$}  at 439 190
\pinlabel{{\scriptsize $v\Omega 2c$}}  at 445 260
\pinlabel{$\leftrightsquigarrow$}  at 1635 190
\pinlabel{{\scriptsize $v\Omega 2d$}}  at 1650 260
\pinlabel{$\leftrightsquigarrow$}  at 2910 190
\pinlabel{{\scriptsize $v\Omega 3b$}}  at 2910 260
\endlabellist
\centering
\includegraphics[width=0.98\textwidth]{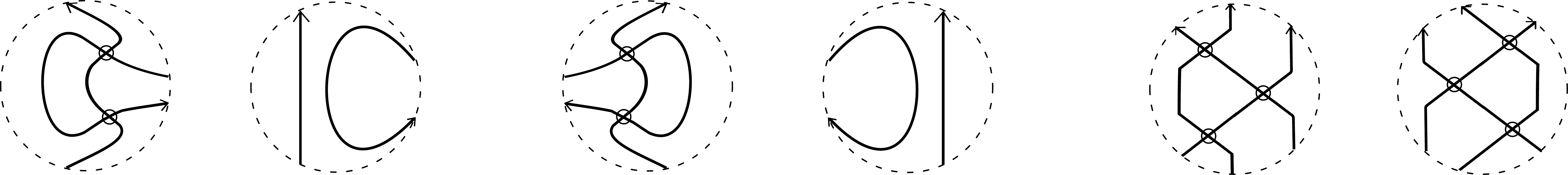}}
\end{equation*}
\begin{equation*}
\centre{
\labellist \small \hair 2pt
\pinlabel{$\leftrightsquigarrow$}  at 439 190
\pinlabel{{\scriptsize $m\Omega 3b$}}  at 445 260
\endlabellist
\centering
\includegraphics[width=0.27\textwidth]{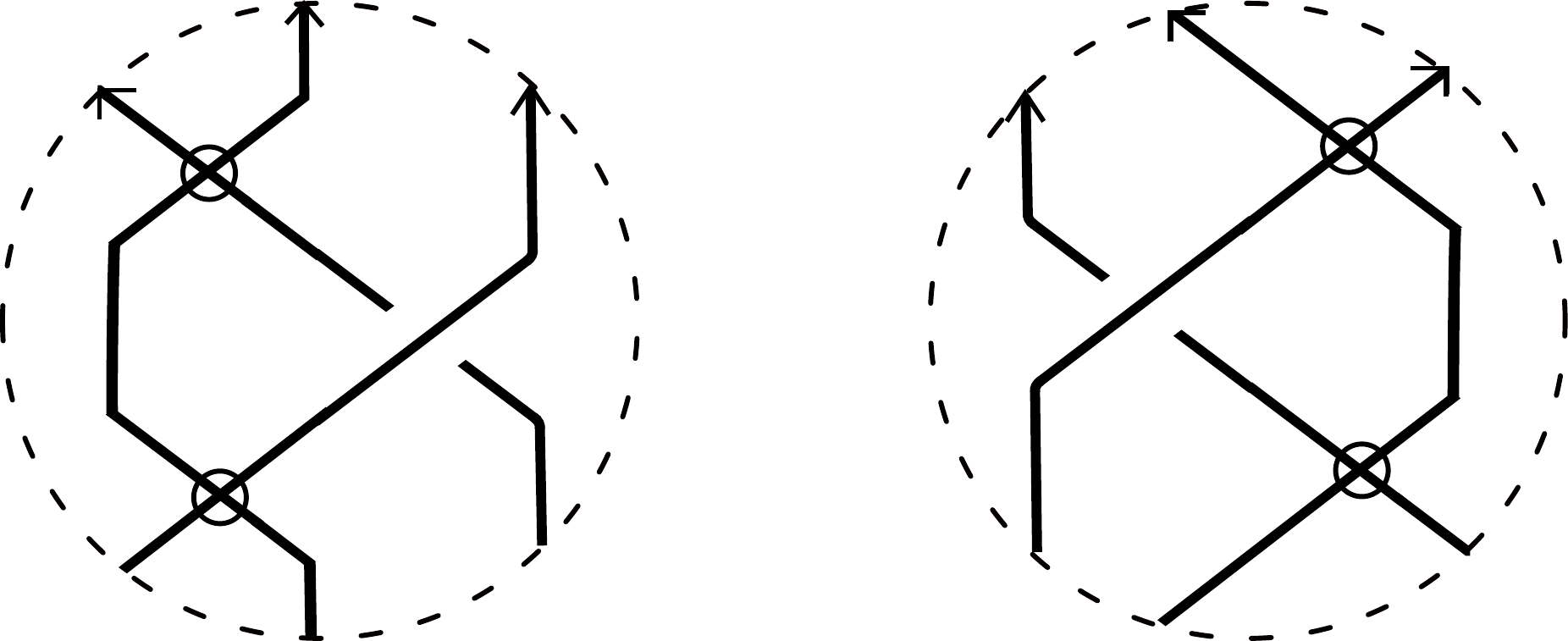}}
\end{equation*}

In particular, this implies the following alternative description of $v\Tup$: for each $n \geq 0$, the set $\mathrm{End}_{v\Tup}(n)$ can be regarded as the set of upwards virtual tangle diagrams modulo the real, virtual and mixed rotational Reidemeister moves depicted above.

\cref{lem:vT_rot_diagrams} is the key ingredient in the following

\begin{construction}\label{constr:I}
Given a virtual upwards tangle, let $D$ be a rotational diagram of the tangle. We can associate an XC-tangle to it as follows: we can view $D$ as a planar unitetravalent graph. We can lift this graph to another graph $G$ with vertices of valence $1,2$ and 4 by splitting the vertices corresponding to the virtual crossings into two different bivalent vertices joining opposite edges. We can also endow each edge of $G$ with a rotation number $-1,0,1$ following the rule that we assign $0$ to the edges which are upwards in $D$, rotation number $1$ to the edges that were like the left-hand side of \eqref{eq:rotations} (clockwise rotations) in $D$, and $-1$ to those that were like the right-hand side of  \eqref{eq:rotations} in $D$.

Further, tetravalent vertices are given the sign $+$ or $-$ according to the sign of the crossing. The splitting $V_{out} \amalg V_{in}$ is determined based on whether the univalent vertex lies in $D^1 \times \{-1\}$ or  $D^1 \times \{+1\}$, respectively, and the total order in each of the parts is decided by reading the vertices from left to right.
\end{construction}

\begin{theorem}\label{thm:vTup->TXC}
The previous construction gives rise to a strict symmetric monoidal embedding  $$ I: v\T^{\mathrm{up}} \hooklongrightarrow \T^{\mathrm{XC}}    $$ 
which is a section for $U$.
\end{theorem}
\begin{proof}
For each $n \geq 0$, \cref{constr:I} induces a well-defined map $$ I: \mathrm{End}_{v\Tup}(n) \to  \mathrm{End}_{\TXC}(n)  .$$ Indeed the real rotational Reidemeister moves are obviously preserved because the relations $(\widehat{\Omega}0a)$ -- $(\widehat{\Omega}3)$ are exactly the images under \cref{constr:I} of the moves $(\Omega 0a)$ --  $(\Omega 3b)$. The virtual rotational Reidemeister moves are preserved by means of the relation $(\widehat{\Omega}0r)$, and so is the mixed relation as both sides are mapped to the XC-tangle diagram
\vspace*{5pt}
\begin{equation*}
\centre{
\labellist \scriptsize \hair 2pt
\pinlabel{$ 1$}  at 0 -30
\pinlabel{$ 1$}  at 0 200
\pinlabel{$ 3$}  at 160 -30
\pinlabel{$ 3$}  at 160 200
\pinlabel{$ 2$}  at 320 -30
\pinlabel{$ 2$}  at 320 200
\endlabellist
\centering
\includegraphics[width=0.12\textwidth]{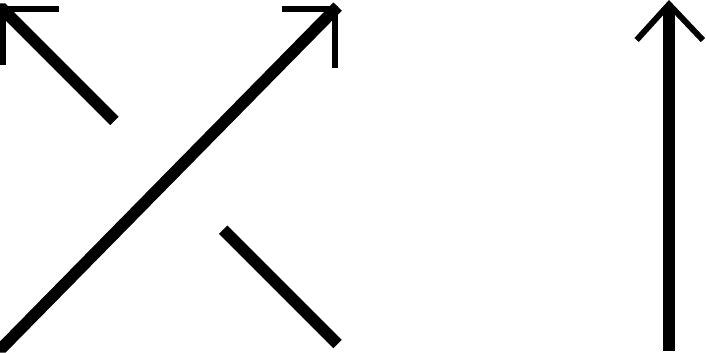}}.
\end{equation*}
\vspace*{5pt}

\noindent Setting $I$ to be the identity on objects defines a strict symmetric monoidal functor $I: v\T^{\mathrm{up}} \to \T^{\mathrm{XC}} $; functoriality and the behaviour with respect to disjoint unions follow comparing the constructions. Lastly that $I$ is a section for $U$, that is $U \circ I=\id$, follows from reglueing in a virtual crossing the original virtual crossing that \cref{constr:I} split into two bivalent vertices. That $I$ has a left inverse  implies in particular that $I$ is an embedding.
\end{proof}

\begin{remark}
Because $U$ fails to be an embedding, $I$ cannot be an isomorphism.
\end{remark}

Obviously, we can view such an embedding $I$ at the Gauss diagram level: there exists a unique strict symmetric monoidal embedding
$$\mathcal{GD}^{\mathrm{up}}  \hooklongrightarrow \mathcal{GD}^{\mathrm{XC}} $$
making the following square commutative:
$$
\begin{tikzcd}
\T^{\mathrm{XC}}  \rar{\cong} & \mathcal{GD}^{\mathrm{XC}}  \\
v\T^{\mathrm{up}}\uar[hook] \rar{\cong} & \mathcal{GD}^{\mathrm{up}} \uar[hook]
\end{tikzcd}
$$
This embedding takes an upwards Gauss diagram to a XC-Gauss diagram which is identical but diamond decorations are added. At the present, it is unclear to the author how to give an intrinsic description of this embedding directly from the upwards Gauss diagram, without invoking the algorithm from \cref{lem:vT_rot_diagrams}.

\begin{remark}
We warn the reader that  naïvely regarding a upwards Gauss diagram as a XC-Gauss diagram with no diamond decoration does \textit{not} descend to a well-defined map in equivalence classes.
\end{remark}

\section{Universal invariants}

In this section we show the precise way in which XC-tangles are the geometric counterpart of XC-algebras. In particular, we will produce a natural functorial invariant of XC-tangles for any choice of XC-algebra.

\subsection{XC-algebras} We start by recalling the notion of XC-algebra from \cite{becerra_refined,BH_reidemeister}. For later usage we will state this more abstractly in an arbitrary symmetric monoidal category, but the reader can safely think of the category of modules over some commutative ring, as in the Introduction or  \cref{ex:usual_XC_alg} below.

Let $(\C, \otimes, \bm{1}, P)$ be a symmetric monoidal category. For convenience we remove the associativity and unital constraints for the formulas, this is possible by Mac Lane's coherence theorem, see \cite{maclane,becerra_strictification}. Recall that  an \textit{algebra object} in $\C$ is an object $A \in \C$ together with two arrows $\mu: A \otimes A \to A$ and $\eta: \bm{1} \to A$, called the \textit{multiplication} and the \textit{unit}, satisfying the associativity and unit conditions
\begin{equation}\label{eq:algebra_object}
\mu \circ (\mu \otimes \id) = \mu \circ (\id \otimes \mu) \qquad , \qquad \mu \circ (\id \otimes \eta) = \id = \mu \circ (\eta \otimes \id).
\end{equation}
For $k \geq 1$, we consider iterated multiplication arrows $\mu^{[k]}: A^{\otimes 2k} \to A^{\otimes k}$  and $\mu_{k}:A^{\otimes k} \to A$  defined as
\begin{gather*}
\mu^{[1]} := \mu \quad , \quad \mu^{[k]} := (\mu^{[k-1]} \otimes \mu)(\id_{ A^{\otimes k-1}} \otimes  P_{A,  A^{\otimes k-1}} \otimes \id_A), \quad k \geq 2.\\
\mu_0 := \eta \quad , \quad \mu_1 := \id \quad , \quad \mu_k := \mu \circ (\mu_{k-1} \otimes \id), \quad k \geq 2
\end{gather*}
Given   arrows $a,b: \bm{1} \to A^{\otimes k}$, we will write $a \cdot b := \mu^{[k]} (a \otimes b)$. Further,  say that an arrow $a:\bm{1} \to A^{\otimes k}$ is \textit{multiplicatively invertible} if there exists another arrow $a^{-1}: \bm{1} \to  A^{\otimes k}$ such that  $ a \cdot a^{-1} = \eta^{\otimes k} = a^{-1} \cdot a .$

Now, let $(A, \mu, \eta)$ be an algebra object in $\C$.  An \textit{XC-structure} on $A$ is a choice of
two preferred,  multiplicatively invertible arrows
$$ R : \bm{1} \to A \otimes A \qquad , \qquad \kappa : \bm{1} \to A,  $$
called the \textit{universal $R$-matrix} and the \textit{balancing element}, respectively, satisfying the following axioms:

\begin{enumerate}[leftmargin=4\parindent, itemsep=2mm]
\item[(XC0)] $R^{\pm 1}=(\kappa \otimes \kappa) \cdot R^{\pm 1} \cdot (\kappa^{-1} \otimes \kappa^{-1})$,
\item[(XC1f)] $\mu_3(R_{31}\cdot  \kappa_2 )=\mu_3(R_{13}\cdot  \kappa^{-1}_2) $,
\item[(XC2c)] $  1\otimes \kappa^{-1} = (\mu\otimes \mu_3 )(R_{15}\cdot R_{23}^{-1}\cdot \kappa^{-1}_4 )$,
\item[(XC2d)] $\kappa \otimes 1 = (\mu_3\otimes\mu )(R_{15}^{-1}\cdot R_{34}\cdot \kappa_2 )$,
\item[(XC3)] $R_{12}R_{13}R_{23}=R_{23}R_{13}R_{12}$,
\end{enumerate}
where   for $1 \leq i,j \leq n$, $i \neq j$ we have put
\begin{equation}\label{eq:R_ij}
 R_{ij}^{\pm 1}:= \begin{cases}  (\eta^{\otimes i-1} \otimes \id_A \otimes \eta^{\otimes j-i-1} \otimes \id_A \otimes \eta^{n-j})R^{\pm 1}, & i<j\\
(\eta^{\otimes j-1} \otimes \id_A \otimes \eta^{\otimes i-j-1} \otimes \id_A \otimes \eta^{n-i})P_{A,A}R^{\pm 1}, & i>j
\end{cases} 
\end{equation}
and similarly $\kappa^{\pm 1}_i = (\eta^{\otimes i-1} \otimes \id_A \otimes  \eta^{\otimes n-i})\kappa^{\pm 1}$. The triple $(A, R, \kappa)$ consisting of an algebra object and an XC-structure on it will be called an \textit{XC-algebra} in $\C$. In \cite{barnatanveenpolytime} a  structure of a similar type  (subject to 20 axioms) appears under the name of ``snarl algebra''.

\begin{example}\label{ex:usual_XC_alg}
If $\Bbbk$ is a commutative ring with unit, and $(\mathsf{Mod}_\Bbbk, \otimes_\Bbbk, \Bbbk)$ denotes the symmetric monoidal category of $\Bbbk$-modules, then XC-algebra objects in $\mathsf{Mod}_\Bbbk$ are identified with the XC-algebras as described in the Introduction or in  \cite{becerra_refined,BH_reidemeister} via the natural bijection
\begin{equation}\label{eq:repres_forgetful}
\hom{\Bbbk}{\Bbbk}{M} \toiso M
\end{equation}
for any $\Bbbk$-module $M$ (the right-hand side of the bijection stands for the underlying set of $M$). 
\end{example}

\begin{example}\label{ex:XC_ribbon}
The main source of examples of XC-algebras comes from ribbon Hopf algebras. In \cite[Proposition 4.4]{becerra_refined} it was shown that if $A$ is a ribbon Hopf algebra over some commutative ring $\Bbbk$ with universal $R$-matrix $R \in A \otimes_{\Bbbk} A$ and ribbon element $v \in A$, then the triple $(A, R, \kappa:= uv^{-1})$ is an XC-algebra, where $u \in A$ is the Drinfeld element.

Even further, any representation $V$ of an XC-algebra induces an XC-algebra structure on $\eend{\Bbbk}{V}$, in particular any finite-free representation of a ribbon Hopf algebra, see \cite[Lemma 4.11]{becerra_refined}. More precisely, if $\rho: A \to \eend{\Bbbk}{V}$ is a finite-free representation and   $( R, \kappa)$ is an XC-structure on $A$, then  $( (\rho \otimes_{\Bbbk} \rho )(R), \rho(\kappa))$ is an XC-structure over $\eend{\Bbbk}{V}$.

For instance, if $V$ is a finite-free representation over the Jimbo-Drinfeld quantum group $U_h(\mathfrak{g})$, then $\eend{\mathbb{C}[[h]]}{V}$ is an XC-algebra. This arguments extends to quantum groups at the roots of unity or to  unrolled quantum groups \cite{CGP}, which do not have a well-defined ribbon structure but there are ``R-matrix'' and ``ribbon'' operators acting on their representations.
\end{example}

\subsection{Generators for \texorpdfstring{$\TXC$}{T^XC}}

We will know explain the precise sense in which the elementary XC-tangles displayed in \eqref{eq:crossings_and_spinners_INTRO} are generators for the category $\TXC$. For that we need to introduce some algebraic preliminaries.

Recall that a (classical, monocoloured) PROP is a symmetric monoidal category generated by a single object.  One could think of a PROP as a gadget to study abstract algebraic operations, whereas an \textit{algebra} over a PROP is a concrete realisation of these operations. We refer the reader to e.g. \cite{pirashvili,zanasi,becerra_combinatorial} for basics about PROPs. We write $\mathsf{A}$ for the   PROP induced by a symmetric monoidal theory with generators $$ \mu: 2 \to 1 \quad , \quad \eta: 0 \to 1  $$  and equations given by \eqref{eq:algebra_object}. The identity $\id_1: 1 \to 1 $, the symmetry $P_{1,1}: 2 \to 2$ and the axioms of a symmetric monoidal category  are implicitly among generators and relations. The PROP $\mathsf{A}$ is called the \textit{PROP for algebras}. Alternatively, $\mathsf{A}$ is ``the symmetric monoidal category generated by an algebra object''.

Let  $\mathcal{M}=(M_i)_{i \in \N}$ be a bundle monoidal category as in \cite[\S 2.5]{becerra_refined}, that is, $\mathcal{M}$ is the monoidal category obtained by a family of monoids $M_i$ with monoidal product determined by some structure monoid homomorphisms $\rho_{n,m}:M_n \times M_m \to M_{n+m}$. Inspired by \cite[\S 6.2]{habiro}, we define an \textit{external algebra action} on $\mathcal{M}$  as a functor $F_{\mathcal{M}}:\mathsf{A} \to \mathsf{Set}$ sending $n$ to $M_n$ which is compatible with the monoidal  product and composition of $\mathcal{M}$ in the sense that $\id_{M_n} =F_{\mathcal{M}}(\eta^{\otimes n})$ and  the diagram
\begin{equation}
\begin{tikzcd}
M_n \times M_n \arrow{dr} \rar{\rho_{n,m}} & M_{2n} \dar{F_{\mathcal{M}}(\mu^{[n]})} \\
& M_n
\end{tikzcd}
\end{equation}
commutes for all $n$, where the diagonal map is the multiplication of $M_n$.

We say that a bundle monoidal category $\mathcal{M}$ equipped with an external algebra action is monoidally generated by a collection of elements $(m_{i_k} \in M_{i_k})_k$ and the external algebra action if every arrow $m \in M_i$ can be obtained from the identities of $M_n$ and the elements $(m_{i_k})_k$ applying finitely many times the operations from the external algebra action, the composition of $\mathcal{M}$ and the monoidal product.

The category $\TXC$ (as well as the categories $\Tup$ and $v\Tup$) can be seen as a bundle monoidal category with $M_n= \eend{\TXC}{n} =: \TXC_n$.  There is an external algebra action that roughly speaking merges components of  a given XC-tangle diagram and introduces edges with rotation number zero. More precisely, according to \cite[Lemma 2]{habiro_fg} or \cite[Theorem 3.5]{becerra_combinatorial}, such an external algebra action is determined by morphisms $F_{\TXC} (P_\sigma): \TXC_n \to \TXC_n$ and $F_{\TXC} (\mu_{k_1, \ldots,k_r}): \TXC_k \to \TXC_r$ where $P_\sigma : n \to n$ is the arrow in $\mathsf{A}$ induced by the symmetric braiding and a permutation $\sigma \in \mathfrak{S}_n$, and $ \mu_{k_1, \ldots,k_r} := \mu_{k_1} \otimes \cdots \otimes \mu_{k_r}: \sum_i k_i = k \to r$,  with $k_i \geq 0$. The external algebra action $F_{\TXC}: \mathsf{A} \to \mathsf{Set}$ is then defined as follows:
\begin{itemize}
\item Given an XC-tangle $T$ with $n$ components, $F_{\TXC} (P_\sigma)(T)$ is the XC-tangle obtained from $T$ by changing the order of the sets of incoming and outcoming vertices according to $\sigma$,
\item Given $T$ an XC-tangle with $k$ components and  non-negative integers $k_1, \ldots , k_r$ such that $\sum_i k_i=k$, we have that $F_{\TXC} (\mu_{k_1, \ldots,k_r})(T)$ is the $r$-component XC-tangle obtained from $T$ by merging the first $k_1$ strands glueing inwards with outwards vertices, then merging the next $k_2$ strands etc. If $k_i=0$, this is to be understood as inserting an edge with rotation number zero (if $k_1=1$, then this procedure does nothing).
\end{itemize}

The following result is fundamental and highlights one of the advantages of working in the setup of XC-tangles as opposed to working with upwards tangles, see \cite[Remark 3.5]{becerra_refined}.

\begin{proposition}\label{prop:generators_TXC}
The category $\TXC$ is generated by the external algebra action and the morphisms
\begin{equation}\label{eq:crossings_and_spinners}
\hspace*{-1.5cm}
\centre{
\labellist \small \hair 2pt
\pinlabel{$X$}  at 350 -78
 \pinlabel{$ X^- $}  at 800 -78
 \pinlabel{$ C$}  at 1140 -78
  \pinlabel{$ C^- $}  at 1560 -78
\endlabellist
\centering
\includegraphics[width=0.6\textwidth]{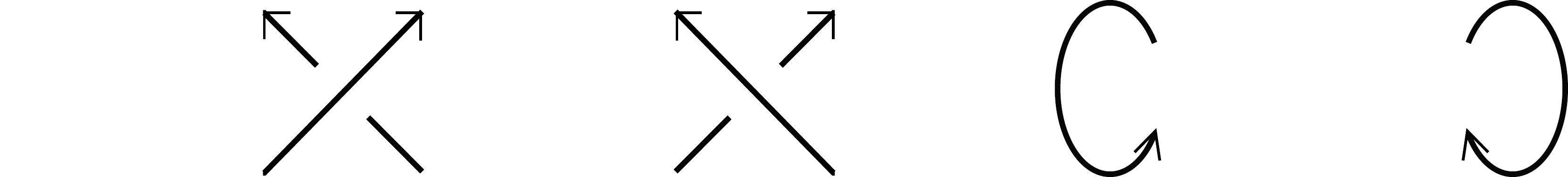}}
\end{equation}
\vspace{8pt}

\noindent subject to the relations $(\widehat{\Omega} 0a)$ -- $(\widehat{\Omega} 3)$, which are seen as morphisms in $\TXC$ choosing an arbitrary but fixed order in the sets of outgoing and incoming univalent vertices.
\end{proposition}
\begin{proof}
Given an XC-tangle, we can always find a representative XC-tangle diagram $D$ with the property that any two tetravalent vertices in a given strand are connected by at least two edges (that is, there is at least one bivalent vertex between them) in view of $(\widehat{\Omega} 0r)$.  Consider then the XC-tangle diagram $D'$ resulting from splitting $D$ at bivalent vertices.  Then $D'$ must be a disjoint union of the elementary XC-tangle diagrams $X^\pm$, $C^\pm$ and $\uparrow$ the edge with rotation number zero. The ordering of the newly created univalent vertices is of course not unique but needs to be chosen in a coherent way so that $D$ arises from $D'$ after glueing the new univalent vertices. Using the external algebra action we can glue the connected components of $D'$, which yields an XC-tangle diagram $\bar{D}$ which only differs from $D$ by a sequence of $(\widehat{\Omega} 0r)$ moves. This shows that $X^\pm$ and $C^\pm$ are generators, which is enough as the relations are exactly those defining the equivalence of XC-tangle diagrams.
\end{proof}

\subsection{The PROP \texorpdfstring{$\mathsf{XC}$}{XC}  for XC-algebras}

In the following we will make use of the so-called ``PROP for XC-algebras''. We write $\XC$ for the PROP induced by a symmetric monoidal theory with generators $$ \mu: 2 \to 1 \quad , \quad \eta: 0 \to 1  \quad , \quad R^{\pm 1}: 0\to 2 \quad , \quad \kappa^{\pm 1} : 0 \to 1 $$  and equations given by \eqref{eq:algebra_object}, (XC0)--(XC3) and the requirement that  $R^{\pm 1}$ and $\kappa^{\pm 1}$ are multiplicatively inverses of each other. 

If $\C$ is a symmetric monoidal category, let us write $\mathsf{Alg}_\XC(\C)$ for the set of algebras in $\C$ over the PROP $\XC$, that is, the of symmetric  monoidal functors $$A:\XC \to \C.$$
It is readily seen that there is a natural bijection
\begin{equation}\label{eq:AlgXC=alg_objs}
\begin{tikzcd}[column sep=3em]
\mathsf{Alg}_\XC(\C) \ar[-,double line with arrow={-,-}]{r} &     \left\{ \parbox[c]{6em}{\centering     {\small  \textnormal{ XC-algebra objects in $\C$}}} \right\} ,
\end{tikzcd}
\end{equation}
which sends $A:\XC \to \C$ to $A(1)$. In other words, $\XC$-algebras in $\C$ are the same thing as XC-algebras in $\C$ (note the difference in fonts).

\subsection{The virtual category of elements}\label{subsec:vEA}

We will now make use of the PROP $\XC$ for XC-algebras to define certain categories in a very neat way.

Let us introduce some notation first. For $n \geq 0$ we write $\SS_n$ for the symmetric group of $n$ elements. If $(\C, \otimes, \bm{1}, P)$ is a symmetric monoidal category, $\sigma \in \SS_n$ and $X \in \C$, we write $P_{\sigma, X}: X^{\otimes n } \to X^{\otimes n}$ for the symmetry induced by $\sigma$. More precisely, $P_{\sigma, X}$ is determined by the requirement that $$P_{(i,i+1), X} = \id_X^{\otimes (i-1)}\otimes P_{X,X} \otimes  \id_X^{\otimes (n-i-1)}$$ and that $P_{\sigma \tau, X}= P_{\sigma, X} \circ P_{\tau, X}$.

Let $(\C, \otimes, \bm{1}, P)$ be a symmetric monoidal category, and let $A$ be an XC-algebra object in $\C$, that we view as a symmetric strong monoidal functor $A: \XC \to \C$ under \eqref{eq:AlgXC=alg_objs}. We write $v\E(A)$ for the category described as follows: its objects are the nonnegative integers, there are only endomorphisms and
\begin{equation}
\eend{v\E(A)}{n} := A(\hom{\XC}{0}{n}) \times \SS_n.
\end{equation}
The composite law is given by
\begin{equation}\label{eq:composition_vE(A)}
(u, \sigma) \circ (v,\tau):= (P_{\tau^{-1},A}(u) \cdot v \ , \  \sigma \tau) ,
\end{equation}
and the identity of the object $n$ is given by the pair $(\eta_A^{\otimes n}, \id)$. Setting $n\otimes m := n+m$ and $$(u, \sigma) \otimes (v,\xi) := (u \otimes v, \sigma \otimes \xi)$$ defines a strict monoidal structure with unit $0$, where $\sigma \otimes \xi$ denotes the block product of permutations. Furthermore, $v\E(A)$ admits a symmetric braiding given by $P_{n,m} := (\eta_A^{\otimes (n+m)}, \Psi_{n,m})  $ where $\Psi_{n,m} \in \SS_{n+m}$ is the permutation that maps $i$ to $i+m$ (mod $n+m$).

The strict symmetric monoidal category $v \E(A)$ just defined is called the \textit{virtual category of elements} of $A$. In the special case of the XC-algebra object $1$ in $\XC$, that is, the XC-algebra that corresponds with $\id: \XC \to \XC$ under \eqref{eq:AlgXC=alg_objs}, we call the resulting category the\textit{ universal virtual category of elements} of XC-algebras, and we will denote it as $v\E (\XC)$.

Note that any XC-algebra $A$ in $\C$, viewed as a symmetric monoidal functor $A: \XC \to \C$, canonically induces a symmetric monoidal functor
\begin{equation}\label{eq:vEXC->vEA}
v\E (\XC) \to v\E (A).
\end{equation}

We can of course view the virtual category of elements $v\E (A)$ as a bundle monoidal category, and we can define an external algebra action $F_{v\E(A)}$ on it as follows: given $(A(f),\sigma)  \in \eend{v \E(A)}{n} = A(\hom{\XC}{0}{n}) \times \SS_n$ and an arrow $g: n \to k$ in $\mathsf{A}$, we define $F_{v\E(A)}(g)(A(f),\sigma) := (A(g \circ f), \tilde{\sigma})$   , where $\tilde{\sigma}$ is determined by the condition that $\tilde{\sigma}=\sigma \tau^{-1}$ when $g=P_\tau$, and when $g=\mu_{k_1,\ldots, k_r}$ it is given by the removing the corresponding indices and  shifting the rest accordingly.

Given two bundle monoidal categories $\mathcal{M}=(M_i)$ and $\mathcal{N}=(N_i)$ equipped with external algebra actions $F_{\mathcal{M}}$ and $F_{\mathcal{N}}$, we say that a monoidal functor $T: \mathcal{M} \to \mathcal{N}$ which is the identity on objects is \textit{compatible} with the external algebra action if the components $T_i: M_i \to N_i$ of $T$ define a natural transformation $F_{\mathcal{M}} \Longrightarrow F_{\mathcal{N}}$.

The following should be interpreted as an XC-analogue of the Reshetikhin-Turaev theorem:

\begin{theorem}\label{thm:Z_A}
Let $A$ be an XC-algebra object in a symmetric monoidal category $\C$. Then there exists a unique symmetric monoidal functor $$Z_A: \TXC \to v \E (A)  $$ which is the identity on objects, compatible with the external algebra action and
\begin{alignat*}{3}
Z_A(X) &= (R, (12)) \qquad  &, \qquad Z_A(X^-) &= (R_{21}^{-1}, (12))\\ Z_A(C) &= (\kappa^{-1}, \id) \qquad &, \qquad  Z_A(C^-)  &= (\kappa, \id).
\end{alignat*}
Furthermore, this functor is full.
\end{theorem}
We call the functor $Z_A$ the \textit{universal XC-tangle invariant}.
\begin{proof}
Uniqueness follows directly from \cref{prop:generators_TXC}. For the existence, we construct $Z_A$ as follows: it is the identity on objects and on morphisms it is defined on the generators of $\TXC$ as in the statement. This implies that for an XC-tangle $T$, the second component of $Z_A(T)$ is precisely the unique permutation $\sigma$ making the following diagram of ordered sets and bijections commute,
$$
\begin{tikzcd}
V_{out} \dar \rar & V_{in} \dar \\
\{1, \ldots , n \} \rar{\sigma} & \{1, \ldots , n \}
\end{tikzcd},
$$
where the vertical arrows are the bijections induced by the total orders of the sets of outgoing and incoming univalent vertices and the top horizontal map is the preferred bijection obtained by taking endpoints of each strand. This functor is well-defined on equivalence classes of XC-tangles tautologically: the relations $(\widehat{\Omega}0a)$ -- $(\widehat{\Omega}3)$ are precisely mapped to the relations (XC0) -- (XC3).

Lastly, to check fullness we argue as follows: given a pair $(a,\sigma) \in \mathrm{End}_{v\E(A)}(n)$ with $a:\bm{1} \to A^{\otimes n}$, we first note that the choice of $\sigma$ is irrelevant as we can always obtain a different one using the symmetric braiding of $v\E(A)$. Now by definition $a$ is a composition of tensor products of identities, the symmetric braiding $P_{A,A}$ and the structure morphisms 
 $$ \mu: A^{\otimes 2} \to A \quad , \quad \eta: \bm{1} \to A  \quad , \quad R^{\pm 1}: \bm{1}\to A^{\otimes 2} \quad , \quad \kappa^{\pm 1} : \bm{1} \to A. $$ This gives us guidelines to construct an XC-tangle from copies of $X^\pm$, $C^\pm$ and the edge with rotation number zero, merging the strands using the external algebra action  and changing the order of the sets of univalent vertices with the symmetric braiding.
\end{proof}

In the case of the XC-algebra object 1 in $\XC$, we simply write $$Z: \TXC \to v\E(\XC)$$ for  the symmetric monoidal functor resulting from \cref{thm:Z_A} with values in the universal virtual category of elements for XC-algebras. The name of the latter is explained by the following theorem, which can be viewed as an ``algebraisation'' of the category of XC-tangles.

\begin{theorem}\label{thm:TXC=vEXC}
The functor $$Z: \TXC \toiso v\E(\XC)$$ is an isomorphism of symmetric monoidal categories. Moreover, for any symmetric monoidal category $\C$ and any XC-algebra object $A$, the functor $Z_A$ from \cref{thm:Z_A} factors through $Z$,
$$
\begin{tikzcd}
\TXC \rar{Z_A} \dar{\rotatebox[origin=c]{-90}{$\cong$}}[swap]{Z} & v \E (A) \\
 v\E(\XC) \arrow{ru}
\end{tikzcd}
$$
where the diagonal arrow is the canonical functor from \eqref{eq:vEXC->vEA}.
\end{theorem}
\begin{proof}
For the first part is suffices to see that $Z$ is faithful. If $Z(T)=Z(T')$, then $T=T'$ as $\hom{\XC}{0}{n}$ contains the same relations (and no more) as the XC-Reidemeister moves, besides the ordering of the sets of univalent vertices must be essentially the same as the two associated permutations coincide. The commutativity of the diagrams holds by construction.
\end{proof}

\subsection{The subcategory of pure XC-tangles}

We would like to present a mild restriction of the functor $Z_A: \TXC \to v \E (A)  $ that is sometimes useful in calculations. 

An XC-tangle diagram is called \textit{pure} if the bijection  $V_{out} \toiso V_{in}  $ induced by taking endpoints of each strand is an isomorphism of totally ordered sets. Therefore, for pure XC-tangle diagrams, the axiom that imposes a total order on the sets $V_{out}$ and $V_{in}$ can be replaced by an axiom that simply imposes a total order on the set $\mathcal{S}(T)$ of strands. We will write $P\TXC$ for the wide monoidal subcategory of $\TXC$ on pure XC-tangles.

Similarly, we will write $P \mathcal{GD}^{\mathrm{XC}}$ for the wide monoidal subcategory of $\mathcal{GD}^{\mathrm{XC}}$ with $$\eend{P\mathcal{GD}^{\mathrm{XC}}}{n} :=  \mathcal{GD} (X)$$ where $X$ is the (homeomorphism class of the) polarised 1-manifold given by a disjoint union of $n$ oriented intervals with $s(X)=t(X)=+^n$ and the endpoints of each component having the same order. In practice, this means that we simply have to order the components of $X$ in the skeleton of an arrow of $P\mathcal{GD}^{\mathrm{XC}}$.

Lastly, if $A$ is an XC-algebra object in a symmetric monoidal category $\C$, we will write $Pv\E(A)$ for the wide monoidal subcategory of $v\E(A)$ whose arrows have as second component the identity permutation. This means that we can describe the arrows of this category simply as $$\eend{Pv\E(A)}{n} = A(\hom{\XC}{0}{n}) $$ with composite law
\begin{equation*}
u \circ v= u \cdot v
\end{equation*}
for $u,v: \bm{1} \to A^{\otimes n}$ in $A(\hom{\XC}{0}{n})$. We also write $Pv\E(\XC)$ when $A$ is the XC-algebra object 1 in $\XC$.

We record the following result, which is immediate.

\begin{corollary}\label{cor:pure}
\begin{enumerate}[label=(\arabic*)]
\item The isomorphism $\TXC \overset{\cong}{\to}   \mathcal{GD}^{\mathrm{XC}}$   from \cref{thm:TXC=GDXC} restricts to an isomorphism $P\TXC \overset{\cong}{\to}  P \mathcal{GD}^{\mathrm{XC}}.$
\item The symmetric monoidal functor $Z_A: \TXC \to v \E (A)  $  from \cref{thm:Z_A} restricts to a monoidal functor $Z_A: P\TXC \to Pv \E (A)  $.
\item The isomorphism $Z: \TXC \toiso v\E(\XC)$ from \cref{thm:TXC=vEXC} restricts to an isomorphism $Z: P\TXC \toiso Pv\E(\XC)$.
\end{enumerate}
\end{corollary}

The ``snarl diagrams'' that appear in \cite{barnatanveenpolytime} could be seen as pure XC-tangles, after appropriately modifying that definition.

\subsection{Comparison with previous work}

In this subsection we would like to relate the constructions carried out in this section in the context of XC-tangles with some of the results from \cite{becerra_refined}.

Let $\T$ be the category of framed, oriented tangles in a cube, that is, the free strict ribbon category generated by a single object. Recall that we write $\Tup \subset \T$ for the monoidal subcategory on the objects sequences of the sign $+$ (hence non-negative integers) and arrows tangles without closed components. We have therefore monoidal embeddings
\begin{equation}
\Tup \hooklongrightarrow v\Tup \hooklongrightarrow \TXC,
\end{equation}
where the first one is canonical and the second is the embedding from \cref{thm:vTup->TXC}.

In \cite[Theorem 3.10]{becerra_refined} it was shown that the category $\Tup$ enjoys a universal property, namely it is the ``free strict open traced monoidal category generated by a single object''. This means that $\Tup$ is freely generated by a braiding, twist and  an open trace, the latter being a mild modification of Joyal-Street-Verity's notion of trace in a balanced category \cite{joyal_street_traced}.

Moreover, if $A$ is an XC-algebra in $\mathsf{Mod}_\Bbbk$, where $\Bbbk$ is a commutative ring, it was shown in \cite[Theorem 3.14]{becerra_refined} that one can construct an open-traced monoidal category $\mathcal{E}(A)$, called the \textit{category of elements} of $A$. This category has the same objects as $\Tup$ (i.e., non-negative integers) and $$ \eend{\mathcal{E}(A)}{n} \subset A^{\otimes n} \times \SS_n  $$ (this is a set-theoretical inclusion). By the universal property of $\Tup$, there exists a unique strict monoidal functor
\begin{equation}\label{eq:Z_A_up}
Z_A: \Tup \to \E(A)
\end{equation}
that sends $+$ to $1$ and preserves the braiding, twist and open trace; furthermore it is full. This functor is called the \textit{universal upwards tangle invariant}.

The relation between the universal XC- and upwards tangle invariant is explained in the following

\begin{theorem}\label{thm:comparison_Z_As}
Let $A$ be an XC-algebra in $\mathsf{Mod}_\Bbbk$.  There is a monoidal embedding $$ \E(A) \hooklongrightarrow v\E(A) $$
making the following diagram commute:
$$
\begin{tikzcd}
\T^{\mathrm{up}}\dar[hook] \rar[two heads]{Z_A} & \E(A) \dar[hook]\\
\T^{\mathrm{XC}}  \rar[two heads]{Z_A} & v\E(A)  
\end{tikzcd}
$$
\end{theorem}
\begin{proof}
In \cite[Construction 4.12 and 4.13]{becerra_refined} it was defined $\mathrm{End}_{\E(A)}(n) := \mathcal{I}_n$ where  $\mathcal{I}_n$ was a certain monoid defined algebraically. By construction, $\mathcal{I}_n$ is a subset of $A(\hom{\XC}{0}{n}) \times \SS_n =\eend{v\E(A)}{n}$. This inclusion induces the embedding $ \E(A) \hooklongrightarrow v\E(A) $. The commutativity of the diagram is immediate comparing the constructions of the two horizontal functors labelled as $Z_A$.
\end{proof}

The previous theorem expresses that the functor $Z_A: \TXC \to v\E(A)$ from \cref{thm:Z_A} generalises the upwards tangle invariant \eqref{eq:Z_A_up}, in particular it extends \eqref{eq:Z_A_up} to virtual upwards tangles (e.g. to virtual framed long knots).

In fact, even more is true. It $A$ is a ribbon Hopf algebra and  $V$ is a finite-free $A$-module, it was shown in \cite[Theorem 5.4]{becerra_refined} that the universal upwards tangle invariant  $Z_{\mathrm{End}_\Bbbk(V)}$ with respect to the XC-algebra $\mathrm{End}_\Bbbk(V)$ from \cref{ex:XC_ribbon} is essentially the restriction of the Reshetikhin-Turaev invariant $RT_V: \T \to \mathsf{fMod}_A^{\mathrm{str}}$ to the category $\Tup$ of upwards tangles. More precisely, it was proved that there is a monoidal embedding $$\mathcal{E}(\mathrm{End}_\Bbbk (V)) \hooklongrightarrow  \mathsf{fMod}_A^{\mathrm{str}}$$ making the following diagram commute:
$$
\begin{tikzcd}
& \Tup \arrow{dl}[swap]{ Z_{\mathrm{End}_\Bbbk(V)}} \arrow{dr}{RT_V} & \\
\mathcal{E}(\mathrm{End}_\Bbbk(V)) \arrow[hook]{rr} &&  \mathsf{fMod}_A^{\mathrm{str}}
\end{tikzcd}
$$

Combining this with \cref{thm:comparison_Z_As}, we obtain

\begin{corollary}\label{cor:extension_RT}
Let $A$ be a ribbon Hopf algebra and let $V$ be a finite-free $A$-module. Then the universal XC-tangle invariant $$ Z_{\mathrm{End}_\Bbbk(V)}: \TXC \to v\E(\mathrm{End}_\Bbbk(V))  $$ extends the Reshetikhin-Turaev invariant $RT_V: \Tup \to \mathsf{fMod}_A^{\mathrm{str}}$ to XC-tangles.
\end{corollary}

\subsection{Comparison with the universal property of $v\T$}\label{subsec:brochier}

Given a finite-free $A$-module $V$ of a ribbon Hopf algebra $A$, \cref{cor:extension_RT} produces an invariant of virtual upwards tangles (precomposing with $I: v\Tup \hooklongrightarrow \TXC$) that extends the Reshetikhin-Turaev invariant $RT_V$.  In this subsection we would like to compare this invariant with an invariant $\widetilde{RT}_V:v \Tup \to \mathsf{fMod}_\Bbbk^{\mathrm{str}}$ that canonically extends the Reshetikhin-Turaev invariant. 

We start by recalling the universal property of the category $v\T$ of virtual framed tangles from \cite{brochier}, that we slightly adapt for convenience. Let $\C$ be a strict ribbon category, $V$ an object in $\C$ and $\mathcal{S}$ a strict symmetric monoidal category. If $G: \C \to \mathcal{S}$ is a strict monoidal functor,  there exists a unique strict symmetric monoidal functor such that $G_V(+)=V$ and $G_V$ extends the composite $ G \circ RT_V$,
$$
\begin{tikzcd}
v\T \arrow[dashed]{drr}{G_V}  & & \\
\T \arrow[hook]{u}  \rar{RT_V} & \C \rar{G \ \ } & \mathcal{S}
\end{tikzcd}
$$
In particular for $\C= \mathsf{fMod}_A^{\mathrm{str}}$, $\mathcal{S}= \mathsf{fMod}_\Bbbk^{\mathrm{str}}$   and $G$ the canonical forgetful functor, we obtain a natural extension $$\widetilde{RT}_V := G_V : v\T \to \mathsf{fMod}_\Bbbk^{\mathrm{str}} $$ of the Reshetikhin-Turaev functor. We now would like to compare the restriction $$\widetilde{RT}_V: v\Tup \to \mathsf{fMod}_\Bbbk^{\mathrm{str}} $$ of this functor to $v\Tup$ (that we still denote in the same way) with the restriction $$ Z_{\mathrm{End}_\Bbbk(V)}: v\Tup \to v\E(\mathrm{End}_\Bbbk(V))  $$ of the functor from \cref{cor:extension_RT} to $ v\Tup $ as well. 

As we mentioned, these functors are canonical extensions of the Reshetikhin-Turaev functor, and in fact we will see that both are essentially the same invariant.

\begin{construction}
Let $A$ be a ribbon Hopf algebra and let $V$ be a finite-free $A$-module. We can consider a monoidal embedding $$\iota_V :v\mathcal{E}(\mathrm{End}_\Bbbk (V)) \hooklongrightarrow  \mathsf{fMod}_\Bbbk^{\mathrm{str}}$$
in a completely analogous way as that in \cite[Construction 5.3]{becerra_refined}. For simplicity we use the identification $\hom{\Bbbk}{\Bbbk}{M} \cong M$ from \cref{ex:usual_XC_alg}. On objects, declare $\iota_V (n) := V^{\otimes n}$. On arrows, given a morphism $$\left( \sum f_1 \otimes \cdots \otimes f_n, \sigma \right) : n \to n$$ in $v\mathcal{E}(\mathrm{End}_\Bbbk (V))$, set $$ \iota_V \left( \sum f_1 \otimes \cdots \otimes f_n, \sigma \right) := \sum \sigma_*(f_1 \otimes \cdots \otimes f_n)  : V^{\otimes n} \to V^{\otimes n},$$
where we view each $f_1 \otimes \cdots \otimes f_n$ in the right-hand side as an element of $\mathrm{End}_\Bbbk (V^{\otimes n})$ via the canonical isomorphism $$\mathrm{End}_\Bbbk (V)^{\otimes n} \toiso \mathrm{End}_\Bbbk (V^{\otimes n}),$$ and $\sigma_* : V^{\otimes n} \to V^{\otimes n}$ permutes the factors. That $\iota_V$ is indeed a functor and preserves the monoidal product follows verbatim from the computations in \cite[Construction 5.3]{becerra_refined}. In fact by construction we have the following commutative diagram,
\begin{equation}\label{eq:square_constr}
\begin{tikzcd}
\mathcal{E}(\mathrm{End}_\Bbbk (V)) \arrow[hook]{d} \arrow[hook]{r} &  \mathsf{fMod}_A^{\mathrm{str}} \arrow[hook]{d} \\
v\mathcal{E}(\mathrm{End}_\Bbbk (V)) \arrow[hook]{r} &  \mathsf{fMod}_\Bbbk^{\mathrm{str}}
\end{tikzcd}
\end{equation}
where the top horizontal functor is the one described in \cite[Construction 5.3]{becerra_refined}, the left vertical functor is the one described in \cref{thm:comparison_Z_As} and the right vertical functor is the forgetful.
\end{construction}

The following theorem should be viewed as the virtual analogue of \cite[Theorem 5.4]{becerra_refined}.

\begin{theorem}\label{thm:comparison_Brochier}
Let $A$ be a ribbon Hopf algebra and let $V$ be a finite-free $A$-module. Then we have the following commutative diagram:
$$
\begin{tikzcd}
& v\Tup \arrow{dl}[swap]{ Z_{\mathrm{End}_\Bbbk(V)}} \arrow{dr}{\widetilde{RT}_V} & \\
v\mathcal{E}(\mathrm{End}_\Bbbk(V)) \arrow[hook]{rr}{\iota_V} &&  \mathsf{fMod}_\Bbbk^{\mathrm{str}}
\end{tikzcd}
$$
That is, up to the embedding $\iota_V$, the functors $Z_{\mathrm{End}_\Bbbk(V)}$ and  $\widetilde{RT}_V$ coincide.
\end{theorem}
\begin{proof}
For a given non-virtual upwards tangle $T$, we have by \cite[Theorem 5.4]{becerra_refined} $$\widetilde{RT}_V(T)=  RT_V(T) = \iota_V(Z_{\mathrm{End}_\Bbbk(V)}(T)).$$
On the other hand, if $X_v$ denotes a virtual crossing, then $\widetilde{RT}_V(X_v)$ is exactly the symmetric braiding of $\mathsf{fMod}_\Bbbk^{\mathrm{str}}$, and under $\iota_V$ this is exactly the same as $Z_{\mathrm{End}_\Bbbk(V)}(T)$.
\end{proof}

All in all, we obtain the following commutative cube relating the different constructions that have appeared (compare with \cite[Corollary 5.5]{becerra_refined}):
 
\begin{corollary}\label{cor:big_diag}
We have the following commutative cube of strict monoidal functors and strict monoidal categories:
\begin{equation*}
\begin{tikzcd}
             & v\T \arrow[rr, "\widetilde{RT}_V "]  \arrow[from=dd, hook]      &                                                              & \mathsf{fMod}_\Bbbk^{\mathrm{str}}        \\
v\Tup \arrow[ru, hook] \arrow[rr,pos=0.8, "Z_{\mathrm{End}_\Bbbk(V)}",crossing over]                 &                                             & v\E (\mathrm{End}_\Bbbk(V)) \arrow[ru,  hook,"\iota_V"]         &                                           \\
 & \T \arrow[rr, pos=0.3, "RT_V"] &                                                              & \mathsf{fMod}_A^{\mathrm{str}} \arrow[uu] \\
\Tup \arrow[uu, hook] \arrow[rr, "Z_{\mathrm{End}_\Bbbk(V)}"] \arrow[ru, hook] &                                             & \E (\mathrm{End}_\Bbbk(V)) \arrow[uu, hook,crossing over] \arrow[ru, hook] &                                          
\end{tikzcd}
\end{equation*}
\end{corollary}
\begin{proof}
The commutativity of the base is \cite[Theorem 5.4]{becerra_refined}, and the left face commutes as the functors involved are all the natural inclusions. The commutativity of the front face  is a direct consequence of \cref{thm:comparison_Z_As}, and the back square commutes by the universal property of $v\T$. Finally the top of the cube is precisely \cref{thm:comparison_Brochier}, and the right face \eqref{eq:square_constr}.
\end{proof}

\section{Perspectives}

We would like to conclude with a list of future directions and open problems, that will be the subject of forthcoming publications.

\subsection{A topological model for the PROP \texorpdfstring{$\mathsf{XC}$}{XC}  }

In \cref{thm:TXC=vEXC} we used the PROP $\XC$ for XC-algebras to construct an algebraic model for the category $\TXC$ of XC-tangles. We claim that the converse is also true, using pure XC-tangles.

Recall that if $\C$ is an ordinary (locally small) category, an object $X$ in $\C$ is a \textit{separator} if the representable functor $$\hom{\C}{X}{-} : \C \to \mathsf{Set}$$ is faithful.

\begin{conjecture}\label{conj:the_conj}
The object $0$ is a separator of $\XC$.
\end{conjecture}

Giving a proof of the above conjecture is a highly non-trivial task, beyond the scope of this paper. It would require to investigate  not only a theory of normal forms for $\XC$ but also --and most importantly-- to construct for every pair $f,g: n \to m$ with $f \neq g$ an arrow $u: 0 \to n$ such that $f \circ u \neq g \circ u$. Showing inequalities as the latter in a PROP is a very difficult problem in general.

Yet let us explain the relevance that this conjecture would have. First, observe that the isomorphism $Z: P\TXC \toiso Pv\E(\XC)$ from \cref{cor:pure}.(3) induces bijections $$ \hom{\XC}{0}{n} = \eend{Pv\E(\XC)}{n} \cong  \eend{P\TXC}{n} =: P\TXC_n. $$ Under these bijections, the representable functor $\hom{\XC}{0}{-}$ can be viewed as a functor $$F: \XC \to \mathsf{Set} \qquad , \qquad n \mapsto   P\TXC_n.$$ The maps $P\TXC_n \to P\TXC_m$ in the image of this functor obtain a topological meaning: in generators, the multiplication $\mu: 2 \to 1 $ induces a map $$ \widetilde{\mu}: P\TXC_2 \to P\TXC_1 $$ that merges the components of every 2-component pure XC-tangle, according to the order and orientations. The arrows $\eta: 0 \to 1 $, $R^{\pm 1}: 0 \to 2$ and $\kappa^{\pm 1}: 0 \to 1$ induce  maps $$ \widetilde{\eta} :* \to P\TXC_1   \qquad , \qquad  \widetilde{R^{\pm 1}}: * \to P\TXC_2   \qquad , \qquad  \widetilde{\kappa^{\pm 1}}:  * \to P\TXC_1   $$
from the one-point set $*\cong \hom{\XC}{0}{0}$ that simply pick the trivial edge, the crossing $X^\pm$ and the full rotation $C^{\mp}$, respectively. The symmetric braiding induces a map  $ P\TXC_2 \to  P\TXC_2 $ that simply exchanges the order of the strands. In general, any map $P\TXC_n \to P\TXC_m$ in the image of this functor will be a combination of these maps. Therefore we can define a category $\widetilde{P\TXC}$ whose objects are non-negative integers and with $$ \hom{\widetilde{P\TXC}}{n}{m} := F(\hom{\XC}{n}{m}). $$ The composition and identities are taken from those in $\XC$.

\begin{corollary}
If \cref{conj:the_conj} holds, then there is an isomorphism of categories $$\XC \toiso \widetilde{P\TXC},$$ which can in fact be promoted to an isomorphism of PROPs.
\end{corollary}

\subsection{Finite type invariants of XC-tangles}\label{sec:fti}

XC-tangles admit a theory of finite type invariants that is very closely related to that for framed upwards tangles and virtual tangles. For simplicity we will restrict to the set $\mathcal{K}^{\mathrm{XC}} := \eend{\TXC}{1}$ of one-component XC-tangles or \textit{XC-knots}.

The isomorphism of \cref{thm:TXC=GDXC} restricts to a bijection $$\mathcal{K}^{\mathrm{XC}} \toiso \mathcal{GD}^{\mathrm{XC}}(\uparrow), $$
where $\mathcal{GD}^{\mathrm{XC}}(\uparrow)$ is the quotient of the set $\mathsf{GD}^{\mathrm{XC}}_{\mathrm{diag}}$ of XC-Gauss diagrams on an oriented interval modulo the relations $(\widehat{G} 0)$ -- $(\widehat{G} 3)$ displayed in \cref{sec:XC-Gauss}.

Let us fix $\Bbbk$ a commutative ring with unit, and write $\Bbbk \mathsf{GD}^{\mathrm{XC}}_{\mathrm{diag}}$ for the free $\Bbbk$-module spanned by $\mathsf{GD}^{\mathrm{XC}}_{\mathrm{diag}}$. There is a $\Bbbk$-linear isomorphism
\begin{equation}\label{eq:I}
I: \Bbbk \mathsf{GD}^{\mathrm{XC}}_{\mathrm{diag}} \toiso \Bbbk \mathsf{GD}^{\mathrm{XC}}_{\mathrm{diag}}
\end{equation}
that maps an XC-Gauss diagram to the sum of all its subdiagrams. This isomorphism descents to another isomorphism 
\begin{equation}\label{eq:I2}
I: \frac{ \Bbbk \mathsf{GD}^{\mathrm{XC}}_{\mathrm{diag}} }{(\widehat{G} 0)-(\widehat{G} 3)}\toiso \frac{ \Bbbk \mathsf{GD}^{\mathrm{XC}}_{\mathrm{diag}} }{I(\widehat{G} 0)-I(\widehat{G} 3)}
\end{equation}
where $I(\widehat{G} 0)$ -- $I(\widehat{G} 3)$ denote the images of the moves $(\widehat{G} 0)$ -- $(\widehat{G} 3)$ under \eqref{eq:I}. We call the right-hand side of the isomorphism \eqref{eq:I2} the \textit{XC-Polyak algebra}, and will denote it by $\mathcal{P}^{\mathrm{XC}}$. Note we then have a chain of $\Bbbk$-linear isomorphisms
\begin{equation}\label{eq:chain_isos}
\Bbbk \mathcal{K}^{\mathrm{XC}} \toiso \Bbbk \mathcal{GD}^{\mathrm{XC}}(\uparrow) \toiso \mathcal{P}^{\mathrm{XC}}.
\end{equation}
We can define a filtration on the XC-Polyak algebra. For $n \geq 0$, let $\mathcal{F}^{\mathrm{XC}}_n \subset \mathcal{P}^{\mathrm{XC}}$ be the submodule spanned by XC-Gauss diagrams with at least $n$ decorations (chords or diamonds). This gives rise to a decreasing filtration 
$$\cdots \subset \mathcal{F}^{\mathrm{XC}}_n  \subset \mathcal{F}^{\mathrm{XC}}_{n-1}  \subset  \cdots \subset \mathcal{F}^{\mathrm{XC}}_1 \subset \mathcal{F}^{\mathrm{XC}}_0 =\mathcal{P}^{\mathrm{XC}}$$
which can be seen as an analogue of the Vassiliev-Goussarov filtration for singular knots.

If $\mathcal{P}^{\mathrm{XC}}_n := \mathcal{P}^{\mathrm{XC}}/\mathcal{F}^{\mathrm{XC}}_n $, then the filtration induces a sequence 
\begin{equation}\label{eq:sequence_FXC}
 \cdots \to \mathcal{P}^{\mathrm{XC}}_n \to \mathcal{P}^{\mathrm{XC}}_{n-1} \to \cdots \to \mathcal{P}^{\mathrm{XC}}_1 \to \mathcal{P}^{\mathrm{XC}}_0=0.
\end{equation}

Let $n \geq 0$. A \textit{finite type invariant of degree $n$} for XC-tangles with values in a $\Bbbk$-module $M$ is a $\Bbbk$-module map $$ v:   \mathcal{P}^{\mathrm{XC}}_{n+1} \to M. $$
Composing with the projection  $\mathcal{P}^{\mathrm{XC}} \to \mathcal{P}^{\mathrm{XC}}_{n+1}   $ and with the chain of isomorphisms \eqref{eq:chain_isos}, we get a ``genuine'' XC-knot invariant $v: \Bbbk \mathcal{K}^{\mathrm{XC}} \to M$.

We will write $$\mathcal{V}^{\mathrm{XC}}_n(M):= \hom{\Bbbk}{\mathcal{P}^{\mathrm{XC}}_{n+1} }{M}$$ for the $\Bbbk$-module of finite type invariants of XC-knots of degree $n$. Observe that applying $\hom{\Bbbk}{-}{M}$ to the sequence \eqref{eq:sequence_FXC} yields a sequence
\begin{equation}
0 = \mathcal{V}^{\mathrm{XC}}_{-1}(M) \hooklongrightarrow \mathcal{V}^{\mathrm{XC}}_0(M) \hooklongrightarrow \mathcal{V}^{\mathrm{XC}}_1(M)  \hooklongrightarrow \mathcal{V}^{\mathrm{XC}}_2(M) \hooklongrightarrow \cdots
\end{equation}
whose colimit $\mathcal{V}^{\mathrm{XC}}(M)$ is the set of all finite type invariants of XC-knots.

Let us briefly explain now how this relates to the theory of finite type invariants of framed, oriented, long knots and long virtual knots. There is a commutative diagram of $\Bbbk$-linear maps
\begin{equation}\label{eq:diagram_fti}
\begin{tikzcd}
\Bbbk \mathcal{K} \rar[hook] \dar  & \Bbbk v \mathcal{K} \rar[hook] \dar & \Bbbk\mathcal{K}^{\mathrm{XC}} \dar\\
\Bbbk \mathcal{K}_{n+1} \rar & \mathcal{P}_{n+1} \rar & \mathcal{P}_{n+1}^{\mathrm{XC}} 
\end{tikzcd}
\end{equation}
that we explain now. We have written $\mathcal{K}:= \eend{\Tup}{1}$ and  $v\mathcal{K}:= \eend{v\Tup}{1}$, and the top row is induced by the embeddings $ \Tup \hooklongrightarrow v \Tup \hooklongrightarrow \TXC $.  On the other hand, we have $\Bbbk \mathcal{K}_{n+1} := \Bbbk \mathcal{K} /\mathcal{F}^{\mathrm{VG}}_{n+1}$, where $\mathcal{F}^{\mathrm{VG}}_{n+1}$ is the submodule spanned by singular knots with at least $n+1$ singular points via the Vassiliev relation; these submodules form the Vassiliev-Goussarov filtration. Lastly,  $\mathcal{P}_{n+1} := \mathcal{P}/\mathcal{F}^{v}_{n+1}$ is the quotient of the (framed) Polyak algebra $\mathcal{P}$ modulo the submodule $\mathcal{F}^{v}_{n+1}$ spanned by arrow diagrams with at least $n+1$ chords, see \cite{polyak00}. The vertical arrows in \eqref{eq:diagram_fti} are all the projection maps.

It is well-known that  the image of $\mathcal{F}^{\mathrm{VG}}_{n+1}$  under the map $\Bbbk \mathcal{K} \hooklongrightarrow \Bbbk v \mathcal{K}$ lies in $\mathcal{F}^{v}_{n+1}$, see \cite[\S 7.4]{ohtsukibook}. Similarly, the injection $\Bbbk v \mathcal{K} \hooklongrightarrow \Bbbk\mathcal{K}^{\mathrm{XC}}$ maps $\mathcal{F}^{v}_{n+1}$ into  $\mathcal{F}^{\mathrm{VG}}_{n+1}$, as it will add at most some diamond decorations. The bottom horizontal maps in \eqref{eq:diagram_fti} are the induced by the top ones.

Applying $\hom{\Bbbk}{-}{M}$ to the bottom row of the diagram, we obtain a sequence
\begin{equation}\label{eq:maps_V}
\mathcal{V}^{\mathrm{XC}}_n(M) \to v\mathcal{V}_n(M) \to \mathcal{V}_n(M).
\end{equation}
It is a classical theorem by Goussarov that the second arrow $ v\mathcal{V}_n(M) \to \mathcal{V}_n(M)$ is surjective when $\Bbbk = \Z$, see \cite[Theorem 3.A]{GPV} (the statement there is phrased for unframed long knots but the same is true in the framed setting). It turns out that the same property applies to the first arrow.

\begin{proposition}\label{prop:surjection}
For any commutative ring $\Bbbk$, the map $\mathcal{V}^{\mathrm{XC}}_n(M) \to v\mathcal{V}_n(M)$ from \eqref{eq:maps_V} is a surjection.
\end{proposition}
\begin{proof}
It is readily seen that the map $\mathcal{P}_{n+1} \to \mathcal{P}_{n+1}^{\mathrm{XC}} $ is a split monomorphism, hence we get the result applying $\hom{\Bbbk}{-}{M}$.
\end{proof}

However, it is not clear at all how far is this morphism to be an isomorphism.

\begin{problem}\label{prob:kernels}
Determine the kernel of $\mathcal{V}^{\mathrm{XC}}_n(M) \to v\mathcal{V}_n(M)$ for all $n \geq 0$.
\end{problem}

Another interesting question is related to the size of the spaces of finite type invariants. If $\Bbbk$ is a field, then the dimensions of the spaces of $\Bbbk$-valued finite type invariants of framed knots are known for low degrees,
\begin{center}
\begin{tabular}{c|c|c|c|c|c|c}
$n$ & 0 & 1 & 2 & 3 & 4 & 5 \\ 
\hline 
$\dim  \mathcal{V}_n(\Bbbk)$ & 1 & 2 & 4 & 7 & 13 & 23 \\ 
\end{tabular}
\end{center}
see \cite[\S 4.6]{CDM}. So we pose the following 

\begin{problem}
Determine the dimensions of  $\mathcal{V}^{\mathrm{XC}}_n(\Bbbk)$ for all $n \geq 0$, when $\Bbbk$ is a field.
\end{problem}

This is  a problem of combinatorial nature, and we should be able to compute the dimensions for low degrees using some modification of the computer program used in \cite{BNH} (in such a publication, the authors compute the dimensions of various spaces of finite type invariants of virtual long knots, in particular one where the Reidemeister 1 move does not appear; but they do not compute the dimensions of $v\mathcal{V}_n(\Bbbk)$ where the framed Reidemeister 1 move is used, so computing these dimensions is also an open problem).

\subsection{XC-Gauss diagram formulas for finite type invariants of knots}

An immediate corollary of Goussarov's theorem is that integer-valued finite type invariants of long (framed) knots admit Gauss diagram formulas. One of the most important consequences of \cref{prop:surjection} is that the same occurs for XC-Gauss diagrams.

We need a bit of notation first. Given $D,D' \in \mathsf{GD}^{\mathrm{XC}}_{\mathrm{diag}} $  two XC-Gauss diagrams in an oriented interval, we define the pairing $\langle D,D' \rangle$ as the coefficient of $D$ in $I(D')$, where $I$ is the isomorphism of \eqref{eq:I}, that is, $$ I(D')= \sum_D  \langle D,D' \rangle D  $$ ($\Bbbk = \Z$ here). In other words, $\langle D,D' \rangle$ is the number of times that $D$ appears as a subdiagram of $D'$.

\begin{corollary}\label{cor:GD_formulas}
Let $\Bbbk = \Z$ and let $v$ be an  integer-valued finite type invariant of long framed knots of degree $n$. Then there exists an element $G_v \in \Z  \mathsf{GD}^{\mathrm{XC}}_{\mathrm{diag}}$ consisting of diagrams with at most $n$ decorations (chords or diamonds) such that $$v(K) = \langle G_v, D_K  \rangle  $$ for any framed long knot $K$, where  $D_K$ is an XC-Gauss diagram representing $K$, and the pairing is extended to elements of $\Z  \mathsf{GD}^{\mathrm{XC}}_{\mathrm{diag}}$ by linearlity.
\end{corollary}

We would like to exemplify this result with the simplest non-trivial finite type invariant of framed knots. As usual, an unsigned diagram stands for the signed sum of all possible (XC-)Gauss diagrams obtained by assigning all possible signs to the given diagram.

\begin{proposition}\label{prop:framing_XC_formula}
For any framed long knot $K$, its framing $\mathrm{fr}(K)$ can be presented as
$$  \mathrm{fr}(K) = \Big\langle 2 \  \begin{array}{c}
\centering
\labellist \scriptsize \hair 2pt
\endlabellist
\includegraphics[scale=0.1]{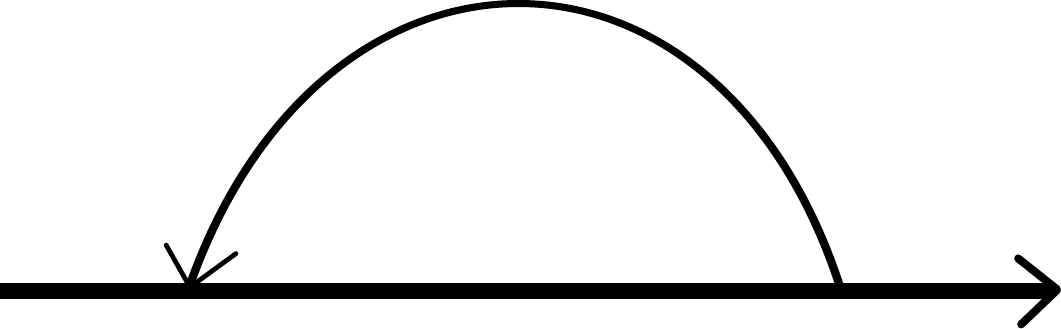} \end{array}  - \   \begin{array}{c}
\centering
\labellist \scriptsize \hair 2pt
\pinlabel{\normalsize $  \blacklozenge$} at 250 17
\endlabellist
\includegraphics[scale=0.1]{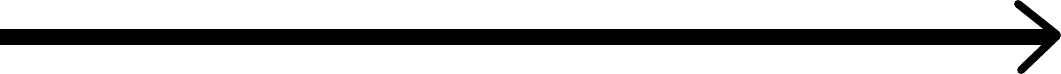} \end{array}  , D_K \Big\rangle .$$
\end{proposition}
\begin{proof}
It was shown in \cite[Corollary 3.7]{becerra_refined} that for a rotational diagram $D$ of a long, framed knot, we have $$ \mathrm{rot}(D)+ \mathrm{wr}(D) = 2  \cdot \sum_{ \substack{c \\ \mathrm{under \ first}}}    \sign (c),   $$ whose translation to XC-Gauss diagram formulas amounts precisely to the statement.
\end{proof}

XC-Gauss diagram formulas appear naturally when studying finite type invariants that are produced by universal invariants \cite[\S 4.5]{becerra_thesis}. It is in general a highly non-trivial task to express (if possible) a XC-Gauss diagram formula for framed, long knots where diamonds appear into one without them (this is of course related to \cref{prob:kernels}). Adding up to  \cref{prop:framing_XC_formula}, we will show in a future publication that
\begin{align*}
&\Big\langle  \begin{array}{c}
\centering
\labellist \scriptsize \hair 2pt
\pinlabel{\normalsize $  \blacklozenge$} at 150 17
\endlabellist
\includegraphics[scale=0.1]{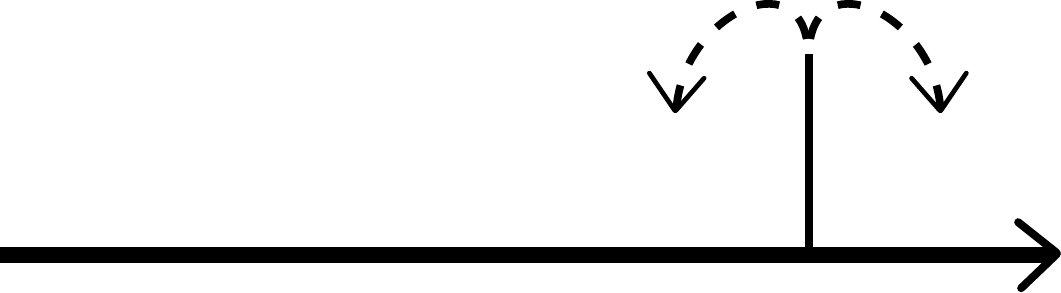} \end{array} + \begin{array}{c}
\centering
\labellist \scriptsize \hair 2pt
\pinlabel{\normalsize $  \blacklozenge$} at 350 17
\endlabellist
\includegraphics[scale=0.1]{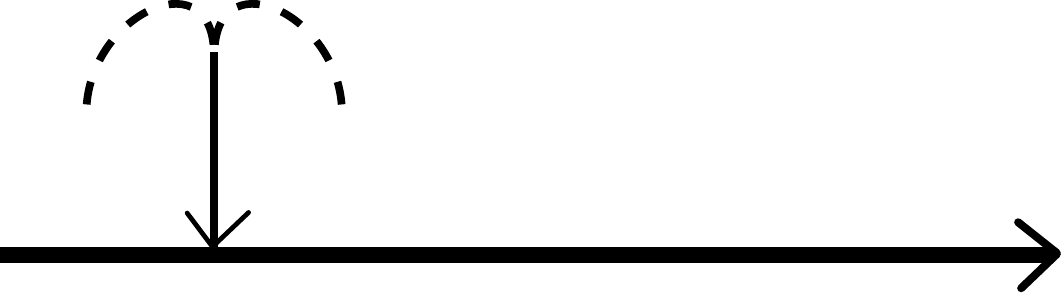} \end{array}  , D_K \Big\rangle = \\
&  \Big\langle   2 \Big (  \begin{array}{c}
\centering
\labellist \scriptsize \hair 2pt
\pinlabel{ $  \pm $}  at 80 110
\endlabellist
\includegraphics[scale=0.1]{figures/gauss_fr1} \end{array}  +  \begin{array}{c}
\centering
\labellist \scriptsize \hair 2pt
\endlabellist
\includegraphics[scale=0.1]{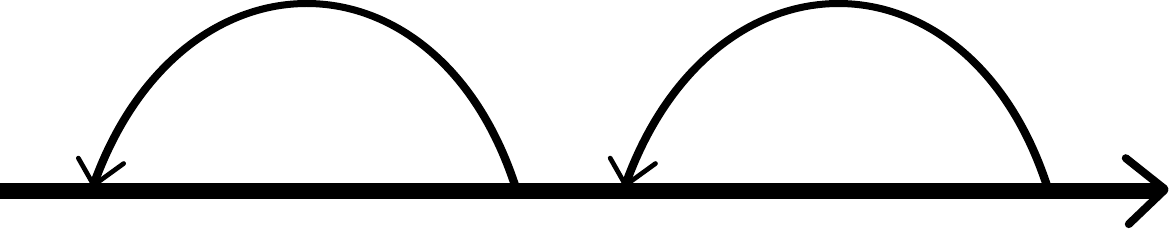} \end{array} -  \begin{array}{c}
\centering
\labellist \scriptsize \hair 2pt
\endlabellist
\includegraphics[scale=0.1]{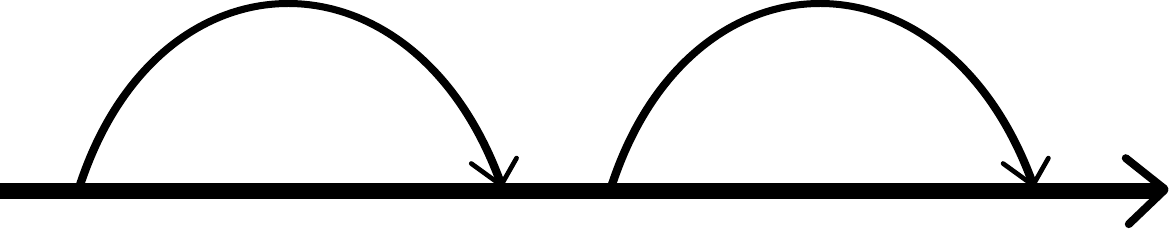} \end{array} \Big) +3 \begin{array}{c}
\centering
\labellist \scriptsize \hair 2pt
\endlabellist
\includegraphics[scale=0.1]{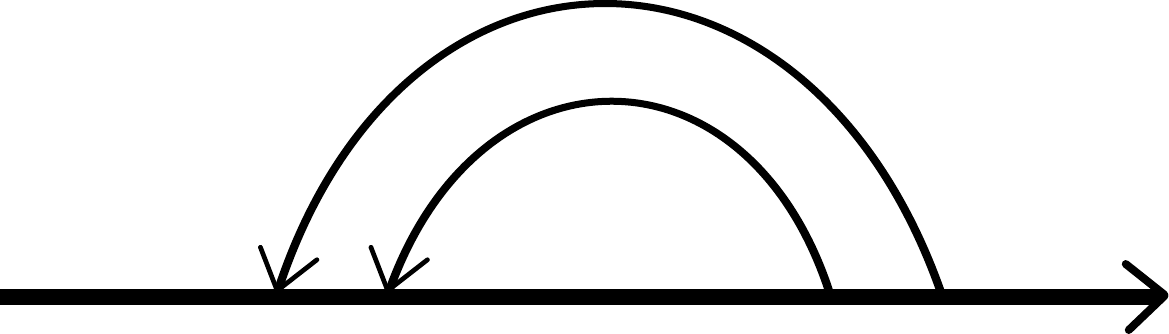} \end{array} \\
&- \begin{array}{c}
\centering
\labellist \scriptsize \hair 2pt
\endlabellist
\includegraphics[scale=0.1]{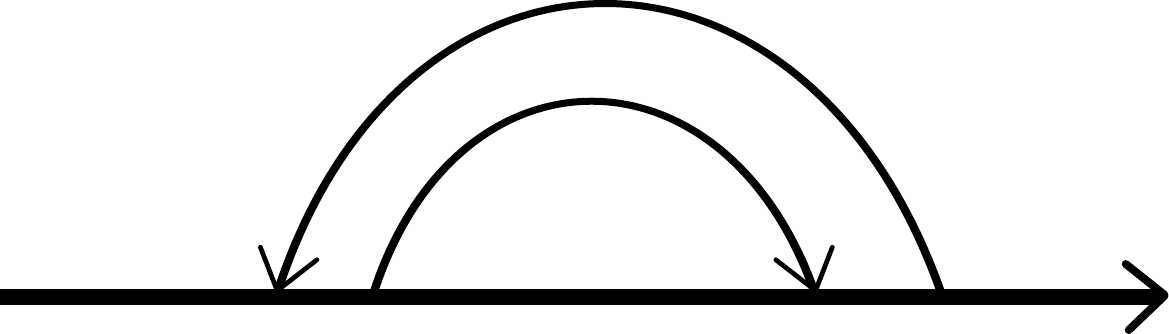}\end{array} - \begin{array}{c}
\centering
\labellist \scriptsize \hair 2pt
\endlabellist
\includegraphics[scale=0.1]{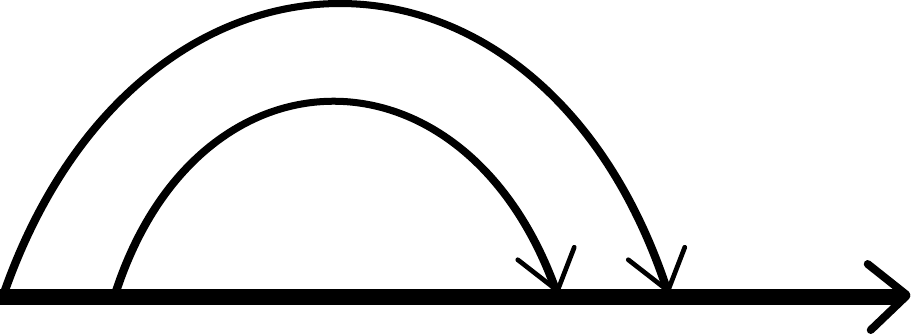}\end{array} - \begin{array}{c}
\centering
\labellist \scriptsize \hair 2pt
\endlabellist
\includegraphics[scale=0.1]{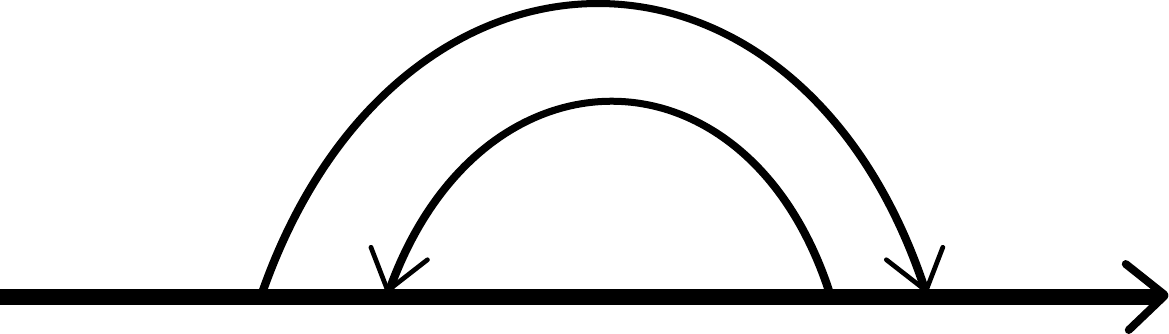} \end{array}   , D_K \Big\rangle.
\end{align*}
for any framed, long knot $K$. In this formula, the dashed, double heads or tails in the ``anchors''  stands for all possible combinations where the solid part is fixed and the dashed part can land anywhere in the diagram (so each of the anchored diagrams is a placeholder for a sum of three unsigned XC-Gauss diagrams).

It remains unclear under what conditions a given linear combination of XC-Gauss diagrams with diamond decorations produces a XC-Gauss diagram formula (that may or not be a isotopy invariant) that can be expressed as a standard Gauss diagram formula (without diamonds).

\subsection{Naturality of the universal XC-tangle invariant}

If $A$ is a ribbon Hopf algebra (hence in particular an XC-algebra), it is well-known that the universal invariant $Z_A$ has ``naturality'' properties with respect to strand doubling, strand deletion and strand orientation-reversal \cite{habiro,becerra_gaussians}. In a forthcoming publication we will see how the setup of XC-tangles provides a very convenient framework to study these naturality properties from a PROPeradic, operadic and double-categorical perspective.

\bibliographystyle{halpha-abbrv}
\bibliography{bibliografia}

\end{document}